\documentclass[12pt]{report}

\usepackage{amsmath}
\usepackage{amsfonts}
\usepackage{amssymb}
\usepackage{amsthm}
\usepackage{setspace}
\usepackage{hyperref}

\newcommand{\h}{\mathfrak H}

\newcommand{\Gb}{{\boldsymbol G}}
\newcommand{\Gbp}{{\boldsymbol G'}}

\newcommand{\gb}{{\boldsymbol g}}
\newcommand{\betab}{{\boldsymbol \beta}}
\newcommand{\Fb}{{\boldsymbol F}}
\newcommand{\Wb}{{\boldsymbol W}}

\newcommand{\vp}{\varphi}
\newcommand{\vpb}{{\boldsymbol \vp}}

\newcommand{\lmod}{\backslash}
\newcommand{\ddx}[1]{\textstyle\frac{\partial}{\partial#1}}
\newcommand{\ddp}[2]{\textstyle\frac{\partial^{#2}}{\partial#1^{#2}}}
\newcommand{\dvh}[2]{\tfrac{d#1\,d#2}{#2^2}}
\newcommand{\Sz}{\mathcal S}
\newcommand{\F}{\mathcal F}

\newcommand{\W}{\mathcal W}

\newcommand{\Hal}{\mathcal H}
\newcommand{\ord}{{\mathcal O}}
\newcommand{\ordx}{{\mathcal O}^\times}
\newcommand{\ordp}{{\mathcal O}_p}
\newcommand{\ordpx}{{{\mathcal O}^\times_p}}

\newcommand{\diva}{D}
\newcommand{\valr}{R}

\newcommand{\gs}{\mathcal G}
\newcommand{\Prj}{\mathbb P}

\newcommand{\Z}{\mathbb Z}
\newcommand{\Zp}{{\mathbb Z}_p}
\newcommand{\Zpx}{{\mathbb Z}_p^\times}
\newcommand{\Q}{\mathbb Q}

\newcommand{\R}{\mathbb R}

\newcommand{\C}{\mathbb C}
\newcommand{\A}{\mathbb A}

\newcommand{\tr}{{\rm tr}}
\newcommand{\dist}{{\rm dist}}
\newcommand{\diag}{{\rm diag}}

\newcommand{\sgn}{{\rm sgn}}
\newcommand{\e}{{\rm e}}

\newcommand{\vol}{{\rm vol}}
\newcommand{\volx}{{{\rm vol}^\times\!}}
\newcommand{\volo}{{{\rm vol}^{(1)}\!}}
\newcommand{\dti}{{d^\times\!}}
\newcommand{\dvt}{{d_v^\times\!}}
\newcommand{\dpt}{{d_p^\times\!}}
\newcommand{\dit}{{d_\infty^\times}}
\newcommand{\don}{{d^{(1)}\!}}
\newcommand{\GL}{{\rm GL}}
\newcommand{\PGL}{{\rm PGL}}
\newcommand{\SL}{{\rm SL}}
\newcommand{\PSL}{{\rm PSL}}
\newcommand{\SO}{{\rm SO}}
\newcommand{\GO}{{\rm GO}}
\newcommand{\GSp}{{\rm GSp}}

\newcommand{\M}{{\rm M}}
\newcommand{\Ad}{{\rm Ad}}
\newcommand{\We}{\mathsf W}
\newcommand{\Nil}{\mathsf N}
\newcommand{\spp}{\mathsf{sp}}

\newcommand{\ov}[1]{\overline{#1}}
\newcommand{\und}[1]{\underline{#1}}
\newcommand{\eps}{\epsilon}
\newcommand{\veps}{\varepsilon}
\newcommand{\hf}{\tfrac{1}{2}}
\newcommand{\ind}{\mathbb I}
\newcommand{\fin}{{\rm fin}}
\newcommand{\ti}{^\times}
\newcommand{\ic}{{\rm i}}
\newcommand{\comment}[1]{}

\newtheorem{thm}{Theorem}
\newtheorem{lem}{Lemma}

\newtheorem{cor}{Corollary}

\newenvironment{mat}[1]{\begin{array}{#1}}{\end{array}}

\begin{document}







\doublespacing \pagenumbering{roman}

 {\thispagestyle{empty}
 \sc
 \vspace*{0in}
 \begin{center}
   \LARGE Rankin Triple Products\\ and Quantum Chaos
 \end{center}
 \vspace{.6in}
 \begin{center}
   Thomas C. Watson
 \end{center}
 \vspace{.6in}
 \begin{center}
   A Dissertation \\
   Presented to the Faculty \\
   of Princeton University \\
   in Candidacy for the Degree \\
   of Doctor of Philosophy
 \end{center}
 \vspace{.3in}
 \begin{center}
   Recommended for Acceptance \\
   by the Department of \\
   Mathematics
 \end{center}
 \vspace{.3in}
 \begin{center}
   June 2002\\ Revised 10/22/2008
 \end{center}
 \clearpage}

 \thispagestyle{empty}
 \vspace*{0in}
 \begin{center}
   \copyright\ Copyright by Thomas C. Watson, 2002. \\
   All Rights Reserved
 \end{center}
 \clearpage

\singlespacing

 \newpage
 \addcontentsline{toc}{section}{Abstract}
 \begin{center}
 \Large \textbf{Abstract}
 \end{center}

In this paper we demonstrate the chaotic nature of certain
special arithmetic quantum dynamical systems, using machinery
from analytic number theory.

Consider the quantized geodesic flow on a finite-volume
hyperbolic surface $\Gamma\lmod\h$, with $\Gamma\subset\SL_2\R$
consisting of the norm-1 units of an Eichler order in an
indefinite quaternion algebra $B$ over $\Q$. Such $\Gamma$
generalize the congruence subgroups of $\SL_2\Z$ and are
co-compact whenever $B$ is ramified.  For $\Gamma=\SL_2\Z$, we
prove that high energy bound eigenstates obey the Random Wave
conjecture of Berry/Hejhal for third moments. In fact we show
that the third moment of a wave's amplitude distribution decays
as $E^{-\frac1{12}}$ in the high energy limit. In the more
general case of maximal orders, we reduce an optimal
quantitative version of the Quantum Unique Ergodicity
conjecture of Rudnick--Sarnak to the Lindel\"of Hypothesis,
itself a consequence of the Riemann Hypothesis, for particular
families of automorphic $L$-functions. Furthermore, our
analysis shows that any lowering of the exponent in the
Phragmen--Lindel\"of convexity bound implies QUE\@. In the
moment problem as well, a decisive role is played by `breaking
convexity.' That is to say, the boundaries of non-trivial
exponents precisely agree when translated between physical and
arithmetical formulations, for both of these problems.

We accomplish this translation by proving identities expressing
triple-correlation integrals of eigenforms in terms of central
values of the corresponding Rankin triple-product
$L$-functions. Very general forms of such identities were
proved by Harris--Kudla, and certain (more explicit) classical
versions of these were given by Gross--Kudla and
B\"ocherer--Schulze-Pillot, for definite quaternion algebras.
In using the Harris--Kudla method to prove our own classical
identities, we have to solve two main problems.

The first problem is to explicitly compute the adjoint of
Shimizu's theta lift, which realizes the Jacquet--Langlands
correspondence by transferring automorphic forms from $\GL_2$
to $\GO(B)$, the latter being nearly the same as
$B\ti{\times}B\ti$. As is well known, theta liftings from
metaplectic to orthogonal groups are generally more difficult
to characterize than lifts in the opposite direction, which can
be evaluated directly in terms of Whittaker functions. Since
$B\ti$ and hence $\GO(B)$ have `multiplicity-one'---as
Jacquet--Langlands proved with the trace-formula---we are able
to determine the adjoint of Shimizu's lift by duality, from
explicit knowledge of Shimizu's lift itself. It is, however,
necessary to generalize Shimizu's original calculations, since
he only considered averages of lifts over isotypic bases of
forms, which allowed him to employ Godement's theory of
spherical functions. In order to deal with individual ramified
forms, we replace this argument by explicit calculations
involving Hecke operators. Thus we determine the Shimizu lifts
of oldforms and newforms of square-free level, with (possibly
imprimitive) neben-characters.

As a byproduct of these calculations, we obtain explicit
formulas for all relevant $\GL_2$ Whittaker functions. These
play an important role in our second main problem: evaluation
of Garrett/Piatetski-Shapiro--Rallis local zeta integrals in
terms of the standard functorial triple-product $L$-factors.
Our contribution lies in the calculation of archimedean zeta
integrals with various types of ramification.  In all of these
ramified cases, the canonical choices of local data for theta
lifting are apparently useful only for evaluating
\emph{central} values of the associated zeta integrals. This
was observed by Gross--Kudla in the non-archimedean case, and
it appears to be a significant unexplained feature of the
Harris--Kudla method. Finally, we re-prove (by elementary
brute-force computation) the important result of
Piatetski-Shapiro--Rallis on unramified non-archimedean zeta
integrals, in anticipation of future generalizations to the
many ramified non-archimedean cases.

\addcontentsline{toc}{section}{Acknowledgements}

\begin{proof}[\indent Acknowledgements.]
This work is the product of my doctoral research at Princeton
University, directed by Peter Sarnak.  I am very grateful to
him for explaining this fascinating problem to me, and for his
generous help and encouragement, which cannot be overstated.  I
also thank Henryk Iwaniec and Goro Shimura for teaching me
analytic and algebraic number theory, as well as Michael
Boshernitzan and Yakov Sinai for teaching me ergodic theory.  I
gratefully acknowledge receiving funding through an NSF
Graduate Fellowship and a Charlotte E. Proctor Honorific
Fellowship. Finally, I thank my friends and family for their
love and support.
\renewcommand{\qed}{}
\end{proof}

\clearpage \tableofcontents \clearpage \pagenumbering{arabic}

\chapter{Automorphic Forms}

\section{Orders and Units of Quaternion Algebras}

Let $B$ be a quaternion algebra over $\Q$. Then
$B_v=B\otimes_\Q \Q_v$ is a quaternion algebra over $\Q_v$, for
each place $v$ of $\Q$. $B$ is called ramified at $v$ if $B_v$
is a division algebra. We will always assume that $B$ is
unramified at $\infty$, i.e.\ $B$ is indefinite. Then the
reduced discriminant $d_B$ of $B$ over $\Q$ equals the product
of the finite even number of finite places $p$ where $B$ is
ramified. The $2\times2$ matrix algebra $M=\M_2$ over $\Q$ is
totally unramified, and hence $d_M=1$.  We denote by $\diva_v$
the unique division quaternion algebra over $\Q_v$, while
$M_{v}$ is the unique split (i.e.\ non-division) quaternion
algebra. We will frequently use the coordinate functions
$\alpha_{ij}$ on $\alpha=\left[\begin{mat}{cc}
\alpha_{11}&\alpha_{12}\\\alpha_{21}&\alpha_{22}\end{mat}\right]
\in M_{v},$ and also the notation $[X_{ij}]=\{\alpha\in
M_{v}\,;\,\alpha_{ij}\in X_{ij}\subset\Q_v\}.$

Write the canonical anti-involution $\iota$ of $B$ or $B_v$ as
$\alpha\mapsto\alpha^{\iota}$, noting its compatibility with field
extensions, and define the reduced trace and norm by
$\tau(\alpha)=\alpha+\alpha^{\iota}$ and
$\nu(\alpha)=\alpha\alpha^\iota$. On $M$ and $M_v$, $\iota$ is
given explicitly as
\begin{align*}
\left[\begin{mat}{cc}a&b\\c&d\end{mat}\right]^{\iota}=
\left[\begin{mat}{rr}d&-b\\-c&a\end{mat}\right],
\end{align*}
so $\tau$ and $\nu$ correspond to the usual trace and
determinant.

$\diva_p$ has the valuation ${\rm ord}_p\circ\nu$, and its
valuation ring $\valr_p$ is also the unique maximal order of
$\diva_p$. Choose a uniformizer $\varpi_p$\,; all bilateral
$\valr_p$ ideals in $\diva_p$ are principal and of the form
$\varpi_p^n \valr_p=\valr_p\varpi_p^n$, $n\in\Z$.

Every maximal order of $M_p$ is conjugate to $\M_2(\Z_p)$. We will
be interested in those of the form $\left[\begin{mat}{rr}\Z_p &
p^{-n}\Z_p\\p^n\Z_p &
\Z_p\end{mat}\right]=\left[\begin{mat}{cc}p^n & \\ & 1
\end{mat}\right]^{-1}\M_2(\Z_p)\left[\begin{mat}{cc}p^n & \\ & 1
\end{mat}\right]$ for \hbox{$n\in\Z$}, although the set of
all maximal orders is larger, having the structure of a
$(p+1)$-homogeneous tree for a natural notion of adjacency. It is
a theorem of Hijikata that every Eichler order, an intersection of
two maximal orders, is conjugate to $\left[\begin{mat}{rr}\Z_p &
\Z_p\\p^n\Z_p & \Z_p\end{mat}\right]$ for some $n\ge0$.  All
Eichler orders containing the previous one are of the form
$\left[\begin{mat}{rr}\Z_p & p^{-\dot n}\Z_p\\p^{\ddot n}\Z_p &
\Z_p\end{mat}\right]$ for $0\le\dot n\le\ddot n\le n$.

Orders $\ord$ of $B$ exist and have the property that
$\ord_p:=\ord\otimes_\Z\Z_p$ is a maximal order in $B_p$ for
almost all $p$. Furthermore, for any other order $\ord'$ of
$B$, $\ord_p'=\ord_p$ for almost all $p$. Conversely, given a
choice over all $p$ of local orders $(\tilde\ord_p)$ which
satisfies $\tilde\ord_p=\ord_p$ for a.a.\ $p$, then
\hbox{$\ord':=\{\alpha\in B\,;\, \alpha_p\in\tilde\ord_p\
\forall p\}$} is an order in $B$ with $p$-adic completions
$\ord_p'=\tilde\ord_p$.

The adele ring $B_\A$ is defined as the restricted direct product
of all $B_v$ relative to $\ord_\fin:=\prod_p\ord_p$, for any order
$\ord$ of $B$, and the idele group $B_\A\ti$ as the restricted
direct product of all $B_v\ti$ relative to $\ord_\fin\ti$. $B$ is
contained diagonally in $B_\A$, and we will write $B=B_\Q\subset
B_\A$ to emphasize this inculsion. By the previous paragraph,
orders of $B$ correspond 1-1 with compact open subgroups of
$B_\fin$.

Maximal orders in $B$ are characterized as orders $\ord$ s.t.\
every $\ord_p$ is maximal. Choose a maximal order for each $B$,
denoted $\ord(d_B)$. Throughout the rest of the paper, for $p\nmid
d_B$ we will identify $B_p$ with $M_p$ in such a way that
$\ord_p(d_B)=\M_2(\Z_p)$. Furthermore, in the case $d_B=1$ we
identify $B$ with $M$, and so $\ord(1)=\M_2(\Z)$.

Now for $\dot N,\ddot N\in\Q^+$ relatively prime to $d_B$ and
s.t.\ $N=\dot N\ddot N\in\Z$, we define the order $\ord(d_B,\vec
N)$ of $B$ by
\begin{align*}\ordp(d_B,\vec N) =
\begin{cases}\hspace{22pt}\valr_p & \text{if }p\mid d_B,\\
\left[\begin{mat}{rr} \Z_p & \dot N\Z_p\\ \ddot N\Z_p &
\Z_p\end{mat}\right] & \text{if }p\nmid d_B.
\end{cases}
\end{align*}
Then $\ord(d_B,\dot N,\dot N^{-1})$ is a maximal order, and
\begin{align*}\ord(d_B,\vec N)=
\ord(d_B,\dot N,\dot N^{-1})\cap\ord(d_B,\ddot N^{-1},\ddot N)\end{align*}
is an Eichler order. It has reduced discriminant $d_BN$ and is
called of level $N$. Using strong approximation and Hijikata's
result, it is easy to show that any Eichler order of level $N$ in
$B$ is conjugate to $\ord(d_B,(1,N))$. Note that the only Eichler
orders containing $\ord(d_B,\vec N)$ are the $\ord(d_B,\vec N')$
s.t.\ $\dot N'\mid\dot N$ and $\ddot N'\mid\ddot N$.

Throughout the rest of the paper, for any $X$ possessing a
character $\nu$, we will use the notations
\begin{align*}
X^{(a)}_*&=\{\beta\in X_*\,;\,\nu(\beta)=a\},\\
X^{(A)}_*&=\{\beta\in X_*\,;\,\nu(\beta)\in A\},\\
X^\pm_\infty&=X^{(\R^\pm)}_\infty,\hspace{10pt} X^{[a]}_p
=X^{(a\Z_p\ti)}_p.
\end{align*}

\section{Automorphic Forms on Quaternion Groups}
\label{autf}

In this section, fix $\ord=\ord(d_B,\vec N)$. Strong approximation
for $\ordx_\fin$ is proved in \cite{mi1}: $B\ti_\A=B\ti_\Q
B_\infty^+\ordx_\fin$ since $B$ is indefinite and
$\nu(\ordx_\fin)=\Z_\fin\ti$. Because of this, and since
$\ord^{(1)}=B\ti_\Q\cap B_\infty^+\ordx_\fin$, there are
homeomorphisms
\begin{align*}
B\ti_\Q\lmod B\ti_\A/\ordx_\fin&\simeq \ord^{(1)}\lmod
B_\infty^+,\\ B\ti_\Q\R^+ \lmod B\ti_\A/\ordx_\fin&\simeq
\ord^{(1)}\lmod\PGL^+_2(\R).
\end{align*}

Given a primitive Dirichlet character $\chi$ mod $N^\chi$, factor
$\chi=\prod_p\chi_p$ and extend its domain to $\Z_\fin\ti$ by
means of the Chinese remainder theorem,
\begin{align*}
(\Z/N^\chi\Z)\ti&\simeq\textstyle\prod_{p\mid N^\chi}
(\Z_p/N^\chi\Z_p)\ti\\
&=\textstyle\prod_{p\mid N^\chi}\Z_p\ti/(1+N^\chi\Z_p).
\end{align*}
Using strong approximation, $\A\ti=\Q\ti\R^+\Z_\fin\ti$\,, we
define the grossen-character $\tilde\chi$ by
$\tilde\chi(qrz)=\ov{\chi(z)}$. Note that
$\tilde\chi_\infty(-1)= \chi(-1)$ and
$\tilde\chi_p(p\Z_p\ti)=\chi(p)$ if $p\nmid N^\chi$.

Since $\ord$ has level $N$, the map
$r_p:\ordp\rightarrow\Z_p/N\Z_p$ defined by
$r_p(\beta)=[\beta_{11}]$ for $p\nmid d_B$\,, is a ring
homomorphism, and hence so is $r:\ord_\fin\rightarrow\Z/N\Z$,
$r(\beta)=(r_p(\beta_p))_{p\mid N}$. This permits us, for
$N^\chi\mid N$, to define $\tilde\chi$ on $\ord_\fin\ti$ and
$\chi$ on $\ordx$ by composition with $r$.\linebreak (Note
$\chi(\beta)=\ov{\tilde\chi(\beta)}$ for $\beta\in\ordx$.) Since
$\tilde\chi$ is well-defined on
$\Q_\fin\ti\cap\ord_\fin\ti=\Z_\fin\ti$, it extends to a character
of $\Q_\fin\ti\ord_\fin\ti$.

As usual, $B_\infty\ti=\GL_2(\R)$ acts on $\C\setminus\R$ by
linear fractional transformations,
\begin{align*}
z\mapsto\beta|z=\frac{az+b}{cz+d}\
\hspace{20pt}\text{for \,}\beta=\left[\begin{mat}{cc}a&b\\
c&d\end{mat}\right],
\end{align*}
preserving the metric $ds^2=(dx^2+dy^2)/y^2$. Furthermore,
$B_\infty^+$ acts on \hbox{weight-$k$} Maass forms ($k\in\Z$)
on $\h=\R+\ic\,\R^+$ by $\psi\mapsto\psi|_\beta^k$,
\begin{align*}\psi|_\beta^k(z)=\big(\tfrac{cz+d}{|cz+d|}\big)^{-k}\,\psi(\beta|z).\end{align*}
Consider the space $A_k(d_B,\vec N,\chi)$ of smooth weight-$k$
Maass forms $\psi$ on $\h$ which have moderate growth and satisfy
the automorphy condition
\begin{align*}
\psi|_{\gamma_\infty}^{k}=\chi(\gamma)\,\psi\hspace{20pt}\text{for \,}
\gamma\in\ord^{(1)}(d_B,\vec N).
\end{align*}
Since $-1\in\ord^{(1)}$, such $\psi$ must vanish identically
unless $\chi(-1)=(-1)^k$, so we assume this from now on. Let
$C_k(d_B,\vec N,\chi)\subset A_k(d_B,\vec N,\chi)$ denote the
cuspidal subspace, equipped with the Petersson inner-product
\begin{align*}
\langle\psi_1,\psi_2\rangle:=\int_{\ord^{(1)}\lmod\h}
\psi_1(z)\,\ov{\psi_2(z)}\,\dvh{x}{y}.
\end{align*}

In the adelic setting, we define $L^2(B\ti_\Q\lmod
B\ti_\A,\tilde\chi)$ as usual to consist of automorphic
functions $\Psi$ with central character $\tilde\chi$,
\begin{align*}
\Psi(\gamma z\beta)=\tilde\chi(z)\,\Psi(\beta)\hspace{20pt}\text{for \,}
\gamma\in B\ti_\Q\,,\ z\in\A\ti,
\end{align*}
and finite norm under the inner-product
\begin{align*}
\langle\Psi_1,\Psi_2\rangle:=\tfrac12 \int_{PB\ti_\Q\lmod
PB\ti_\A}\Psi_1(\beta)\,\ov{\Psi_2(\beta)}\,d\ti\!\beta.
\end{align*}
The Tamagawa measure $d\ti\!\beta$ is defined in \S\ref{tmeas}.
Note that these two inner-products are normalized differently.
We denote the right regular action of $\alpha\in B\ti_\A$ as
\begin{align*}
\alpha|\Psi(\beta)=\Psi(\beta\alpha)\hspace{20pt}\text{for \,}
\Psi\in L^2(B\ti_\Q\lmod B\ti_\A,\tilde\chi),
\end{align*}
and define $L^2_0(B\ti_\Q\lmod B\ti_\A,\tilde\chi)$ as the
closed invariant subspace of cuspidal $\Psi$.

Now let
\begin{align*}
K_v(d_B,\vec N)=
\begin{cases}
\ordpx(d_B,\vec N)\hspace{10pt}&\text{if }v=p,\\
\SO(2,\R)&\text{if }v=\infty,\end{cases}
\end{align*}
and extend $\tilde\chi_\infty$ (dependent on $k$) to
$\R^+K_\infty$ from $\R\ti$ by setting
\begin{align*}
\tilde\chi_\infty(r\kappa_\theta)=e^{ik\theta}\hspace{20pt}\text{for \,}
\kappa_\theta=\left[\begin{mat}{rr}\cos\theta&\sin\theta\\
-\sin\theta&\cos\theta\end{mat}\right].
\end{align*}
We define $\tilde C_{k}(d_B,\vec N,\tilde\chi)$ to consist of
all smooth \hbox{$\Psi\in L^2_0(B\ti_\Q\lmod
B\ti_\A,\tilde\chi)$} satisfying
\begin{align*}
\kappa_v|\Psi=\tilde\chi_v(\kappa_v)\,\Psi\hspace{20pt}\text{for \,}
\kappa_v\in K_v(d_B,\vec N).
\end{align*}
(Note this is consistent with the inclusion of the center
$\A\ti\subset B\ti_\Q\R^+\ordx_\fin$.) Then we have an
isomorphism $C_{k}(d_B,\vec N,\chi)
\stackrel\sim\longrightarrow \tilde C_{k}(d_B,\vec
N,\tilde\chi)$,  given by
\begin{align*}
\Psi(\gamma\beta_\infty^+\kappa_\fin)
=\tilde\chi(\kappa_\fin)\,\psi|_{\beta_\infty^+}^{k}\!(\ic).
\end{align*}

If $\ord(d_B,\vec N')\supset\ord(d_B,\vec N)$ and $N^\chi\mid
N'$, there is an inclusion \begin{align*}C_k(d_B,\vec
N',\chi)\hookrightarrow C_k(d_B,\vec N,\chi).\end{align*} Any
form in the image of such a map with $pN'\mid N$ is called
$p$-old, as are linear combinations of such forms, and any form
orthogonal to all $p$-old forms is called $p$-new. Likewise at
the archimedean place, we have raising and lowering operators,
\begin{gather*}
R_k^\pm:C_k(d_B,\vec N,\chi)\longrightarrow C_{k\pm2}
(d_B,\vec N,\chi),\\[5pt]
R_k^+=(z-\ov{z})\ddx{z}+\tfrac{k}2\,,\hspace{20pt}
R_k^-=(\ov{z}-z)\ddx{\ov{z}}-\tfrac{k}2\,.
\end{gather*}
A weight-$k$ form $\psi$ is called $\infty$-old if
$\psi=R_{k\mp2}^\pm\,\psi'$ for $\pm k\ge0$, and $\infty$-new
if it is orthogonal to all $\infty$-old forms.  Since $R_{k\mp
2}^\pm$ and $R_k^\mp$ are adjoints, $\psi$ is $\infty$-new iff
$R_k^\mp\,\psi=0$ for $\pm k\ge0$, and this is equivalent to
$y^{-\frac{|k|}2}\,\psi(z)$ being a holomorphic/
anti-holomorphic function of $z$. All forms with $k=0$ are
$\infty$-old.

For the rest of this paper we assume that $N$ is square-free, and
factoring it into $N=N^\flat N^\sharp N^\chi$, we define
$S_k(d_B,\vec N,\chi,N^\flat)\subset C_k(d_B,\vec N,\chi)$ as the
subspace of forms which are $p$-old for $p\mid N^\flat$, $p$-new
for $p\mid N^\sharp$, and $\infty$-new if $|k|>0$. We will also
assume that $|k|\neq1$.

\section{Hecke Operators}
\label{hops}

In this section, fix $\ord=\ord(d_B,\vec N)$ and $\chi$ primitive
with conductor $N^\chi\mid N$, and assume that $N=N^\flat N^\sharp
N^\chi$ is square-free.  For $p\mid N^\flat N^\sharp$, set
\begin{align*}
\dot\ord_p=\ord_p(d_B,(\dot N,\tfrac{\ddot N}p)),\hspace{30pt}
\ddot\ord_p=\ord_p(d_B,(\tfrac{\dot N}p,\ddot N)).
\end{align*}

\subsection{Double Cosets}
\label{cosets}

For $\alpha\in B_p\ti$, define the double coset
$T_p(\alpha)=\ordpx\,\alpha\,\ordpx$. If $p\mid d_B$, every
$T_p(\alpha)$ is equal to some
$T_p(\varpi_p^n)=\varpi_p^n\valr_p\ti=\valr_p\ti\varpi_p^n$ for
$n\in\Z$, and $T_p(\varpi^n)\subset\ord_p$ iff $n\ge0$. Now for
$p\nmid d_B$ and $\vec a\in(\Q_p\ti)^2$, define
\begin{align*}
T_p(\vec a) =T_p\left(\left[\begin{mat}{cc}\dot a & \\ & \ddot
a\end{mat}\right]\right),\hspace{30pt} R_p(\vec a)
=T_p\left(\left[\begin{mat}{cc} & \dot a \\ \ddot a &
\end{mat}\right]\right).
\end{align*}
Since all $a\in\Q_p\ti$ and $\alpha_{\vec
N}=\left[\begin{mat}{cc} &\dot N\\\ddot N&\end{mat}\right]$
normalize $\ordpx$,
\begin{align*}
T_p(a)=a\ordpx=\ordpx a,\hspace{30pt} R_p(\vec N)= \alpha_{\vec
N}\ordpx=\ordpx\alpha_{\vec N}.
\end{align*}

If $p\nmid d_BN$, every $T_p(\alpha)$ is equal to $T_p(p^{\vec
n})$ for some $\dot n\ge \ddot n$ (set $n=\dot n-\ddot n$), and
$T_p(p^{\vec n})\subset\ord_p$ iff $\ddot n\ge0$. Given $\alpha\in
B_p\ti$, the parameters $\dot n,\ddot n$ are determined by
\begin{align*}
\alpha\in p^{\ddot n}(\ordp\setminus p\ordp),\hspace{40pt}
\nu(\alpha)\in p^{\dot n+\ddot n}\,\Zpx.
\end{align*}
Furthermore, we have the following left coset partition: (For
right cosets, apply $\iota$\,; for lower triangular
representatives, conjugate by $R_p(\vec N)$.)
\begin{align*}
T_p(p^{\vec n})&=\bigcup_{m=0}^{n}\bigcup_{
\stackrel{\scriptstyle x\ {\rm mod}\,p^m} {\scriptstyle p\nmid
x\text{ if }0<m<n}} \left[\begin{mat}{rr}p^{m+\ddot n} & \dot
Nxp^{\ddot n}\\ & p^{n-m+\ddot n}\end{mat}\right]\ordpx.
\end{align*}

We prove both well-known assertions of the previous paragraph at
the same time. Note that $T_p(p^{\vec n})$ contains the above
union of left cosets, and they are disjoint since
\begin{align*}
&\left[\begin{mat}{rr}p^{m+\ddot n}&
\dot Nxp^{\ddot n}\\& p^{n-m+\ddot n}
\end{mat}\right]^{-1} \left[\begin{mat}{rr}p^{m'+\ddot n}&
\dot Nx'p^{\ddot n}\\ & p^{n-m'+\ddot n}
\end{mat}\right]\\&\hspace{50pt}=\left[\begin{mat}{rr}p^{m'-m}&
\dot N(x'p^{-m}-x\hspace{1pt}p^{-m'})\\ &
p^{m-m'}\end{mat}\right]\notin\ordpx
\end{align*}
unless $m=m'$, $x\equiv x'\mod p^{m}$.  Now it suffices to show
that the union of these left cosets over all $\dot n\ge \ddot
n\ge0$ has additive volume (see \S\ref{tmeas}) equal to that of
$\ordp\cap B_p\ti$, since for any $\alpha\in\ordp\cap B_p\ti$,
$\vol(\alpha\ordpx)/\vol(\ordpx)=|\alpha|_{B_p}=
|\nu(\alpha)|_p^2>0.$ Finally,
\begin{align*}
\ord_p\setminus\ordpx&=\left[\begin{mat}{rr}\Zpx & \dot N\Zpx\\
\ddot N\Zpx& \Zpx\end{mat}\right]^{(p\Z_p)}
\cup\left[\begin{mat}{rr}p\Zp & \dot Np\Zp\\ \ddot N\Zpx &
\Zpx\end{mat}\right]\cup\ldots\cup\\
&\hspace{30pt}\cup\left[\begin{mat}{rr}p\Zp & \dot Np\Zp
\\ \ddot Np\Zp& \Zpx\end{mat}\right]\cup\ldots\cup\left[\begin{mat}{rr}p\Zp
& \dot Np\Zp\\ \ddot Np\Zp & p\Zp\end{mat}\right],
\hspace{20pt}\text{so}\\[12pt]
\vol(\ordpx)&=1-\left(p^{-1}\left(\tfrac{p-1}p\right)^3+
4p^{-2}\left(\tfrac{p-1}p\right)^2+
4p^{-3}\left(\tfrac{p-1}p\right)^1+p^{-4}\right)\\
&=\left(1-p^{-1}\right)\left(1-p^{-2}\right),
\hspace{20pt}\text{and}\\[12pt]
\zeta_p(1)\zeta_p(2)&=\vol(\ordp)/\vol(\ordpx)  \ge
\vol(\ordp\cap B_p\ti)/\vol(\ordpx)\\&\ge\sum_{\dot n\ge \ddot
n\ge0} \left(1+(p-1)\left(1+p+\cdots+p^{\dot n-\ddot
n-2}\right)+p^{\dot n-\ddot n}\right) \left(p^{-2\dot n-2\ddot
n}\right)\\ &=\left(1+(p-1)\sum_{n>0}p^{-n-1}\right)\sum_{\ddot
n\ge0} p^{-4\ddot n}\\&=
\left(\frac{1+p^{-2}}{1-p^{-1}}\right)\frac{1}{\left(1-p^{-4}\right)}
 = \zeta_p(1)\zeta_p(2).
\end{align*}

If $p\mid N$, every $T_p(\alpha)$ (with $\alpha\in B_p\ti$) is
equal to either $T_p(p^{\vec n})$ or $R_p(p^{\vec n}\vec N)$
for some $\dot n,\ddot n\in\Z$, and $T_p(\alpha)\subset\ordp$
iff $\dot n,\ddot n\ge0$. We distinguish among these as
follows: First, $m=\min\{\dot n,\ddot n\}$ satisfies $\alpha\in
p^{m}(\ordp\setminus p\ordp)$.  Now determine the type of
double coset and sign of $n=\dot n-\ddot n$ using the partition
\begin{align*}
\ordp\setminus p\ordp&=
\stackrel{T,+}{\left[\begin{mat}{rr}p\Zp & \dot N\Zp\\\ddot N\Zp &
\Zpx\end{mat}\right]}\cup \stackrel{T,0}{\left[\begin{mat}{rr}\Zpx
& \dot N\Zp\\\ddot N\Zp & \Zpx\end{mat}\right]}
\cup\stackrel{T,-}{\left[\begin{mat}{rr}\Zpx & \dot N\Zp\\\ddot
N\Zp & p\Zp\end{mat}\right]}\cup\\
&\hspace{30pt}\cup\stackrel{R,+}{\left[\begin{mat}{rr}p\Zp & \dot
Np\Zp\\\ddot N\Zpx & p\Zp\end{mat}\right]}\cup
\stackrel{R,0}{\left[\begin{mat}{rr}p\Zp & \dot N\Zpx\\\ddot N\Zpx
& p\Zp\end{mat}\right]}\cup
\stackrel{R,-}{\left[\begin{mat}{rr}p\Zp & \dot N\Zpx\\\ddot Np\Zp
& p\Zp\end{mat}\right]}.
\end{align*}
The remaining parameter $\max\{\dot n,\ddot n\}$ is recovered from
\begin{align*}
\nu(\alpha)\in\left\{
\begin{aligned}
p^{\dot n+\ddot n}\,\Z_p\ti &\hspace{20pt}
\text{if }\alpha\in T_p(p^{\vec n}),\\
p^{1+\dot n+\ddot n}\,\Z_p\ti &\hspace{20pt}
\text{if }\alpha\in R_p(p^{\vec n}\vec N).
\end{aligned}\right.
\end{align*}
These $T_p(\alpha)$, taking $\alpha$ diagonal or anti-diagonal,
have left coset representatives as follows: (For right cosets,
apply $\iota$ and conjugate by $R_p(\vec N)$.)
\begin{align*}
T_p(\alpha)=\bigcup_{\scriptstyle x\ {\rm mod}\,p^{|n|}}
\alpha_x\ordpx,\hspace{30pt}\alpha_x=\begin{cases}
\left[\begin{mat}{lr}1&\dot Nx\\&1\end{mat}\right]\alpha &
\text{if }n\ge0,\vspace{5pt}\\
\left[\begin{mat}{lr}1&\\\ddot Nx&1\end{mat}\right]\alpha &
\text{if }n\le0.\end{cases}
\end{align*}

As in the case $p\nmid d_BN$, we prove all of these claims
together by showing that the left cosets above (with $\dot n,\ddot
n\ge0$) completely partition $\ordp\cap B_p\ti$. The double cosets
listed above are disjoint for unequal types/parameters by the
preceding characterization, and the left cosets $\alpha_x\ordpx$
listed within the same double coset are disjoint, since
$\alpha_{x}^{-1}\,\alpha_{x'}\notin\ordpx$ unless $x\equiv x'\mod
p^{|n|}$.

Finally, \hbox{$\ordpx=\left[\begin{mat}{rr}\Zpx & \dot N\Z_p
\\ \ddot N\Zp & \Zpx
\end{mat}\right]$} has $\vol(\ordpx)=p^{-1}(1-p^{-1})^2$, and
\begin{align*}
\zeta_p(1)^2&=\vol(\ord_p)/\vol(\ordpx)  \ge
\vol(\ord_p\cap B_p\ti)/\vol(\ordpx)\\ &\ge\sum_{\ddot
n\ge0}p^{-4\ddot n}+2\sum_{\dot n>\ddot n\ge0}p^{\dot n-\ddot
n-2\dot n-2\ddot n}\\ &\hspace{30pt}+\sum_{\ddot n\ge0}p^{-4\ddot
n-2}+2\sum_{\dot n>\ddot n\ge0}p^{\dot n-\ddot n-2\dot n-2\ddot
n-2}\\ &=\left(1+p^{-2}\right)\left(1+2\sum_{n>0}p^{-n}\right)
\sum_{\ddot n\ge0}p^{-4\ddot n}\\
&=\left(1+p^{-2}\right)\left(\frac{1+p^{-1}}{1-p^{-1}}\right)
\frac1{\left(1-p^{-4}\right)}=\zeta_p(1)^2.
\end{align*}

In the subcase $p\mid N^\flat N^\sharp$, we write double cosets of
$\dot\ord_p\ti$,\,$\ddot\ord_p\ti$ as $\dot T_p(\alpha)$,\,$\ddot
T_p(\alpha)$.  It is easy to check by the disjointness/volume
argument that
\begin{align*}
\dot T_p(1)&=T_p(1)\cup R_p(\dot N,\tfrac{\ddot N}p),\\
\ddot T_p(1)&=T_p(1)\cup R_p(\tfrac{\dot N}p,\ddot N).
\end{align*}

\subsection{Non-Archimedean Operators}
\label{hring}

The space $\Sz(B_p\ti)$ of $\C$-valued Schwarz functions on
$B_p\ti$ forms an algebra under convolution (using the measure
$\dpt\beta$ defined in \S\ref{tmeas}),
\begin{align*}
(\phi_p*\phi_p')(\alpha):=\int_{B_p\ti}\phi_p(\alpha\beta^{-1})\,
\phi_p'(\beta)\,\dpt\beta,
\end{align*}
and it acts on $L^2_0(B_\Q\ti\lmod B_\A\ti,\tilde\chi)$ via the
right regular representation,
\begin{align*}
\phi_p|\Psi:=\int_{B_p\ti}
\beta|\Psi\,\phi_p(\beta)\,\dpt\beta.
\end{align*}
Now we will define and study the subquotient algebras
$\Hal^\star_p=\Hal_p,\,\Hal_p^\flat,\,\Hal_p^\sharp,\,\Hal_p^\chi$,
in the respective cases $p\nmid N,\,p\mid N^\flat,\, p\mid
N^\sharp,\, p\mid N^\chi$, acting on $\tilde S_k(d_B,\vec
N,\tilde\chi,N^\flat)$.

For any $p$, define $\Hal_p$ to be the subalgebra of $\Sz(B_p\ti)$
generated by the functions
\begin{align*}
\phi_{T_p(\alpha)}:=\tfrac1{\volx(\ordpx)}\,
\ind_{T_p(\alpha)}\hspace{20pt}\text{with \,}\alpha\in B_p\ti.
\end{align*}
We abuse notation by writing $T_p(\alpha)$ in place of
$\phi_{T_p(\alpha)}$. For $a\in\Q_p\ti$, let
\begin{align*}
T_p^{[a]}:=\textstyle\sum_{\alpha\in\ordpx\lmod
\ordp^{[a]}/\ordpx}T_p(\alpha)\in\Hal_p.
\end{align*}

If $p\nmid N^\chi$, $\Hal_p$ acts on $\tilde S_k(d_B,\vec
N,\tilde\chi,N^\flat)$ non-trivially. The element
\begin{align*}
T_p(\alpha)^\vee:=T_p(\alpha^{-1})\in\Hal_p
\end{align*}
acts as the adjoint of $T_p(\alpha)$, and $T_p(p)$ acts as the
scalar $\chi(p)$.

If $p\mid d_B$ then $\Hal_p\simeq\C\big(T_p(\varpi_p)\big)$ and
$T_p^{[p^n]}=\begin{cases}T_p(\varpi_p)^n& \text{if }n\ge0,\\
\quad 0 & \text{if }n<0.
\end{cases}$

If $p\nmid d_BN$, it is well known (see \cite{gs}) that
$\Hal_p\simeq\C\big[T_p(p,1),T_p(p),T_p(p)^{-1}\big]$,
\begin{align*}
T_p^{[p^n]}&=\sum_{\stackrel{\scriptstyle \dot n+\ddot
n=n}{\scriptstyle \dot n\ge \ddot n\ge\hspace{1pt}0}}T_p(p^{\vec
n}),\\ \Hal_p[[X]] \ni \sum_{n=0}^\infty
T_p^{[p^n]}X^n&=\left(T_p(1)-T_p(p,1)X+p\,T_p(p)X^2\right)^{-1}.
\end{align*}

For $p\mid N$, we begin by computing the multiplication rules in
$\Hal_p$\,:
\begin{gather*}
T_p(\vec a)^n = T_p({\vec a}^{\,n})\text{ \,for \,}n\ge0,
\hspace{20pt}R_p(\vec N)^2 = T_p(p),\\[8pt]
\begin{alignedat}{2}
T_p(\vec a)\,T_p(p)&=T_p(\vec ap)&&=T_p(p)\,T_p(\vec a),\\
T_p(\vec a)\,R_p(\vec N)&=R_p(\vec a\vec N)&&= R_p(\vec
N)\,T_p(\ddot a,\dot a),
\end{alignedat}\displaybreak[0]\\[8pt]
\begin{aligned}
T_p(p,1)\,T_p(1,p)&=p\,T_p(p)+(p-1)\,T_p(p,1)\,R_p(\vec N),\\
T_p(1,p)\,T_p(p,1)&=p\,T_p(p)+(p-1)\,T_p(1,p)\,R_p(\vec N).
\end{aligned}
\end{gather*}
Also,
\begin{align*}
T_p^{[p^n]}=\sum_{\stackrel{\scriptstyle
\dot n+\ddot n=n}{\scriptstyle \dot n,\ddot n\ge0}}T_p(p^{\vec
n})+\sum_{\stackrel{\scriptstyle 1+\dot n+\ddot n=n}{\scriptstyle
\dot n,\ddot n\ge0}}R_p(p^{\vec n}\vec N).
\end{align*}

Now define $\Hal^\flat_p$ as the centralizer of $R_p(\vec N)$ in
$\Hal_p$. Shimizu proved in \cite{sz2}\linebreak (see also
\cite{pi}) that $\Hal^\flat_p$ is a maximal commutative subalgebra
of $\Hal_p$, that it contains all $T_p^{[p^n]}$, and that
\begin{align*}
\sum_{n=0}^\infty T_p^{[p^n]}X^n =
\frac{T_p(1)+p\,R_p(\vec N)X} {T_p(1)-T_p^{[p]}X+p\,R_p(\vec
N)X+p\,T_p(p)X^2}.
\end{align*}

To define $\Hal^\sharp_p$, we first note that
\begin{align*}
\dot T_p(1)&=\tfrac1{p+1}\big( T_p(1)+R_p(\dot N,\tfrac{\ddot
N}p)\big),\\
\ddot T_p(1)&=\tfrac1{p+1}\big( T_p(1)+ R_p(\tfrac{\dot
N}p,\ddot N)\big),
\end{align*}
and compute the relations
\begin{alignat*}{2}
T_p(1,p)\,\dot T_p(1)&=\tfrac{p}{p+1}\big(T_p(1,p)+R_p(\vec
N)\big)&&=\dot T_p(1)\,T_p(1,p),\vspace{5pt}\\ T_p(p,1)\,\ddot
T_p(1)&=\tfrac{p}{p+1}\big(T_p(p,1)+R_p(\vec N)\big)&&=\ddot
T_p(1)\,T_p(p,1).
\end{alignat*}
If $p\mid N^\sharp$, the elements $\dot T_p(1)$, $\ddot T_p(1)$
of $\Hal_p$ annihilate $\tilde S_k(d_B,\vec
N,\tilde\chi,N^\flat)$, while $T_p(1,p)$, $T_p(p,1)$ are
invertible, as we will see in \S\ref{autr}. Such a
representation of $\Hal_p$ factors through the quotient
\begin{align*}
\Hal_p^\sharp:={}&\Hal_p\big/\big(T_p(1,p)+R_p(\vec N)\big)\\
={}&\Hal_p\big/\big(T_p(p,1)+R_p(\vec N)\big)\simeq
\C\big(R_p^\sharp(\vec N)\big),\\
T_p^{[p^n]}\equiv{}&
\begin{cases}
\big({-}R_p(\vec N)\big)^n & \text{if }n\ge0,\\
\hspace{8pt} 0 & \text{if }n<0,
\end{cases}
\hspace{20pt}\text{in }\Hal_p^\sharp.
\end{align*}

Now suppose $p\mid N^\chi$. For $\beta\in B_p\ti,\
\delta=\det\beta,\ m\in\{1,2\},\ m'=3-m$, we define
\begin{align*}
\phi^{\chi\,m,m}_{T_p(p^{(n_1,n_2)})}(\beta)
&=\tfrac1{\volx(\ordpx)}\,\ind_{T_p(p^{(n_1,n_2)})}(\beta)
\begin{cases}
\ov{\tilde\chi_p(\beta_{mm})} & \text{if }n_m\le n_{m'},\\
\ov{\tilde\chi_p(\delta/\beta_{m'm'})} & \text{if }n_m\ge n_{m'}
,\end{cases}.\\
\phi^{\chi\,m,m'}_{T_p(p^{(n_1,n_2)})}(\beta)&=\quad0,\displaybreak[0]\\
\phi^{\chi\,m,m'}_{R_p(p^{(n_1,n_2)}\vec N)}(\beta)
&=\tfrac1{\volx(\ordpx)}\,\ind_{R_p(p^{(n_1,n_2)}\vec
N)}(\beta)
\begin{cases}
\ov{\tilde\chi_p(\beta_{mm'})} & \text{if }n_m\le n_{m'},\\
\ov{\tilde\chi_p(-\delta/\beta_{m'm})} & \text{if }n_m\ge n_{m'},
\end{cases}\\
\phi^{\chi\,m,m}_{R_p(p^{(n_1,n_2)}\vec N)}(\beta)&=\quad0.
\end{align*}
These cases were treated somewhat differently by Miyake in
\cite{mi1,mi2}. It is straightforward to check, using the
partition in \S\ref{cosets}, that
\begin{align*}
\phi^{\chi\,\vec m}_{T_p(\alpha)}(\dot\kappa\beta\ddot\kappa)
= \ov{\tilde\chi_p(\dot\kappa_{\dot m\dot m}\ddot\kappa_{\ddot
m\ddot m})}\,\phi^{\chi\,\vec
m}_{T_p(\alpha)}(\beta)\hspace{20pt}\text{for \,}
\dot\kappa,\ddot\kappa\in\ordpx.
\end{align*}
Any function $\phi^\chi\in\Sz(B_p\ti)$ satisfying this
transformation property and supported on a double-coset
$T_p(\alpha)\in\Hal_p$ is determined by it's value
$\phi^\chi(\beta)$ on any $\beta\in T_p(\alpha)$. By taking
$\beta$ to be diagonal or anti-diagonal and considering the
left and right actions of diagonal $\kappa\in\ordpx$, we see
that in the diagonal case, $\dot m\neq\ddot m\ \Rightarrow\
\phi^\chi=0$, and in the anti-diagonal case, $\dot m=\ddot m\
\Rightarrow\ \phi^\chi=0$. Thus, any such $\phi^\chi$ is
proportional to the corresponding $\phi^{\chi\,\vec
m}_{T_p(\alpha)}$ listed above. These functions are normalized
so that their values are multiplicative,
\begin{align*}
\phi^{\chi\,\vec m_1}_{T_p(\alpha_1)}(\alpha_1)\cdot
\phi^{\chi\,\vec m_2}_{T_p(\alpha_2)}(\alpha_2) =
\phi^{\chi\,\dot m_1,\ddot m_2}_{T_p(\alpha_1\alpha_2)}(\alpha_1\alpha_2),
\end{align*}
for diagonal or anti-diagonal $\alpha_j\in B_p\ti$, whenever
neither side vanishes. Now it follows from convolving supports
in $\Hal_p$ that the $\phi^\chi$ convolve as
\begin{gather*}
\big(\phi^{\chi\,m,m}_{T_p(\vec a)}\big)^{*n}
=\phi^{\chi\,m,m}_{T_p(\vec a^{\,n})} \ \,\text{for}\ \,
n\ge0,\hspace{20pt} \phi^{\chi\,m,m'}_{R_p(\vec N)}*
\phi^{\chi\,m',m}_{R_p(\vec N)}=\phi^{\chi\,m,m}_{T_p(p)},\\[10pt]
\begin{alignedat}{2}
\phi^{\chi\,m,m}_{T_p(\vec a)}*\phi^{\chi\,m,m}_{T_p(p)}
&=\phi^{\chi\,m,m}_{T_p(\vec ap)}&
&=\phi^{\chi\,m,m}_{T_p(p)}*\phi^{\chi\,m,m}_{T_p(\vec a)},\\
\phi^{\chi\,m,m}_{T_p(\vec a)}*\phi^{\chi\,m,m'}_{R_p(\vec N)}
&=\phi^{\chi\,m,m'}_{R_p(\vec a\vec N)}&
&=\phi^{\chi\,m,m'}_{R_p(\vec N)}*
\phi^{\chi\,m',m'}_{T_p(\ddot a,\dot a)},\\
\phi^{\chi\,m,m}_{T_p(p,1)}*\phi^{\chi\,m,m}_{T_p(1,p)}
&=p\,\phi^{\chi\,m,m}_{T_p(p)}& &=
\phi^{\chi\,m,m}_{T_p(1,p)}*\phi^{\chi\,m,m}_{T_p(p,1)}\,.
\end{alignedat}
\end{gather*}
Define $\Hal_p^{\chi}$ to be the $\C$-algebra of finite linear
combinations of the functions
\begin{gather*}
T_p^{\chi}(\vec a) := \phi^{\chi\,1,1}_{T_p(\vec
a)}+\phi^{\chi\,2,2}_{T_p(\vec a)}\, \hspace{20pt}\text{for
\,}\vec a\in(\Q_p\ti)^2,\\[10pt]
\text{so \,}\Hal_p^\chi\simeq
\C\big(T_p^\chi(p,1),T_p^\chi(1,p)\big).
\end{gather*}
Also define
\begin{align*}
R_p^{\chi}(\vec a\vec N) =  \phi_{R_p(\vec
a\vec N)}^{\chi\,1,2}+ \phi_{R_p(\vec a\vec N)}^{\chi\,2,1}.
\end{align*}
All of the functions $\phi^\chi$ satisfy
$\phi^\chi(z\beta)=\ov{\tilde\chi_p(z)}\,\phi^\chi(\beta)$ for
\hbox{$z\in\Z_p\ti$}, and hence act non-trivially on
$L^2_0(B_\Q\ti\lmod B_\A\ti,\tilde\chi)$.

Given $\eta^\chi\in\prod_{p\mid N^\chi}(\Z/2\Z)$, define
\begin{align*}
\Psi^{\eta^\chi}&=\textstyle{\big(\prod_{p\mid N^\chi}R_p^\chi(\vec
N)^{\eta^\chi_p}\big)}\Psi\hspace{20pt}\text{and}\\
\tilde C_k^{\eta^\chi}(d_B,\vec N,\tilde\chi) &=
\textstyle{\big(\prod_{p\mid N^\chi}R_p^\chi(\vec N)^{\eta^\chi_p}\big)}
\tilde C_k(d_B,\vec N,\tilde\chi),
\end{align*}
which consists of $\Psi\in L^2_0(B_\Q\ti\lmod
B_\A\ti,\tilde\chi)$ satisfying
\begin{align*}
\kappa|\Psi=\begin{cases}
\tilde\chi_p(\kappa_{11})\,\Psi & \text{if }\eta^\chi_p=0,\\
\tilde\chi_p(\kappa_{22})\,\Psi & \text{if }\eta^\chi_p=1,
\end{cases}\hspace{20pt}\text{for \,}
\kappa\in\ordpx,\ p\mid N^\chi,
\end{align*}
plus the usual $K_v$ condition for $v\nmid N^\chi$. The action
of $T_p^{\chi}(\vec a)\in\Hal_p^\chi$ stabilizes each $\tilde
C_k^{\eta^\chi}(d_B,\vec N,\tilde\chi)$ and is adjoint to the
action of $T_p^{\chi}(\vec a)^\vee:=T_p^{\chi}(\vec a^{-1})$,
while $T_p^{\chi}(p)$ acts as the identity.

\subsection{Archimedean Operators}

Let $\Hal_\infty^\star$ denote the semigroup generated by
$T_\infty^\Delta$ alone,
\begin{align*}
T_\infty^\Delta\,\Psi&= \left(-y^2\left(\ddp{x}2+\ddp{y}2\right)-
y\tfrac{\partial^2}{\partial
x\partial\theta}\right)\sigma_z\kappa_\theta|\Psi,\\
T_\infty^{-}\,\Psi&= \eps_\infty|\Psi,\hspace{20pt}
\eps_\infty=\left[\begin{mat}{cc}1&\\ &-1\end{mat}\right]\in
B_\infty\ti.
\end{align*}
It is easy to check that these commute and
\begin{align*}
T_\infty^\Delta&:\tilde C_{k}^{\eta^\chi}(d_B,\vec N,\tilde\chi)
\rightarrow\tilde C_{k}^{\eta^\chi}(d_B,\vec N,\tilde\chi),\\
T_\infty^{-}&:\tilde C_{k}^{\eta^\chi}(d_B,\vec
N,\tilde\chi)\rightarrow\tilde C_{-k}^{\eta^\chi}(d_B,\vec
N,\tilde\chi).
\end{align*}
For $\eta\in\prod_v(\Z/2\Z)$, $\eta^\chi=(\eta_p)_{p\mid N^\chi}$,
define
\begin{align*}
\Psi^\eta&=(T_\infty^-)^{\eta_\infty}\,\Psi^{\eta^\chi},\\
\tilde C_{k}^{\eta}(d_B,\vec N,\tilde\chi)&=\tilde
C_{(-1)^{\eta_\infty}k}^{\eta^\chi}(d_B,\vec N,\tilde\chi).
\end{align*}

Now consider the $\C$-valued coordinates,
\begin{align*}
\begin{aligned}
X(\beta)&=\tfrac12(a+d)+\tfrac{i}2(b-c),\vspace{5pt}\\
Y(\beta)&=\tfrac12(a-d)+\tfrac{i}2(b+c),
\end{aligned}\hspace{30pt}
\beta=\left[\begin{mat}{cc}a&b\\c&d\end{mat}\right]\in B_\infty,
\end{align*}
in which we write the $(2,2)$-form $\nu$ and a
positive-definite majorant $P$,
\begin{alignat*}{4}
\nu(\beta)&\phantom{:}=\tfrac12\tr(\beta\beta^\iota)
=X\ov{X}-Y\ov{Y}={}&ad&-bc,\\
P(\beta)&:=\tfrac12\tr(\beta\beta^t)=X\ov{X}+Y\ov{Y}
={}&\tfrac12(a^2+b^2&+c^2+d^2).
\end{alignat*}
It is straightforward to compute, writing $B_\infty^+\ni
\beta=\kappa_{\dot\theta}
\left[\!\begin{mat}{cc}re^{t/2}\!\!&\\&re^{-t/2}
\end{mat}\!\!\right] \kappa_{\ddot\theta}$\,, that
\begin{alignat*}{2}
X(\beta)&=e^{i(+\dot\theta+\ddot\theta)}\,r\cosh\tfrac
t2,&\hspace{30pt} \nu(\beta)&=r^2,\\
Y(\beta)&=e^{i(-\dot\theta+\ddot\theta)}\,r\sinh\tfrac t2,&
P(\beta)&=r^2\cosh t.
\end{alignat*}
Now define $\phi_{X_\infty}^{k,\xi},\phi_{Y_\infty}^{k,\xi}
\in\Sz(\R^+\lmod B_\infty^+)$ by
\begin{align*}
\begin{aligned}
\phi_{X_\infty}^{k,\xi}(\beta)&=
\tfrac1\pi\left(\ov{X}(\beta)/{\nu(\beta)}^{\frac12}\right)^{\und
k}\,\e\left(i\xi P(\beta)/\nu(\beta)\right),\\
&=\tfrac1\pi\,\ov{\tilde\chi(\kappa_{\dot\theta })}
\ov{\tilde\chi(\kappa_{\ddot\theta })}\,(\cosh\tfrac
t2)^{|k|}\,\e(i\xi\cosh t),\\
\phi_{Y_\infty}^{k,\xi}(\beta)&=
\tfrac1\pi\left(\ov{Y}(\beta)/{\nu(\beta)}^{\frac12}\right)^{\und
k}\,\e\left(i\xi P(\beta)/\nu(\beta)\right),\\
&=\tfrac1\pi\,\tilde\chi(\kappa_{\dot\theta })
\ov{\tilde\chi(\kappa_{\ddot\theta })}\, (\sinh\tfrac
t2)^{|k|}\,\e(i\xi\cosh t),
\end{aligned}\hspace{20pt}z^{\und k}=
\begin{cases}
z^k & \text{if }k>0,\\ 1 & \text{if }k=0,\\ \ov{z}^{|k|} & \text{if }k<0.
\end{cases}
\end{align*}
These give rise to Hecke operators again commuting with
$T_\infty^\Delta$,
\begin{align*}
\begin{aligned}
&T_{*_\infty}^{k,\xi}:
C_{k}^{\eta^\chi}(d_B,\vec N,\tilde\chi) \rightarrow
C_{k}^{\eta}(d_B,\vec N,\tilde\chi),\\
&T_{*_\infty}^{k,\xi}\,\Psi=\displaystyle \int_{\R^+\lmod
B_\infty^+} \beta|\Psi\,\phi_{*_\infty}^{k,\xi}(\beta)\,
\dit\beta,
\end{aligned}\hspace{20pt}
\eta_\infty=
\begin{cases}
0 & \text{if }*=X,\\ 1 & \text{if }*=Y.
\end{cases}
\end{align*}

\section{Automorphic Representations}
\label{autr}

Given $\Psi\in\tilde S_{k}(d_B,\vec N,\tilde\chi,N^\flat)$, define
$\pi\subset L^2_0(B_\Q\ti\lmod B_\A\ti,\tilde\chi)$ as the
subspace of $K_\A(d_B)$-finite vectors in the closure of the span
of $\{\beta|\Psi\,;\, \beta\in B_\A\ti\}$. This is the cuspidal
automorphic representation of $B_\A\ti$ attached to $\Psi$. If
$\pi$ is irreducible, then $\pi=\tilde\otimes_v\pi_v$ and
\hbox{$\Psi=\tilde\otimes_v\Psi_v$} factor as restricted tensor
products over all places $v$, relative to a choice of spherical
unit vectors $\Psi_v$ at almost all places.  In this section, we
recall the classification of pre-unitary irreducible admissible
representations of $B_v\ti$ having square-free conductor
\cite{jl,gode,gel}, and explicitly compute the actions of Hecke
operators on these models. As a consequence, we have

\begin{lem}\label{basis} The eigenforms in $\tilde
S_k(d_B,\vec N,\tilde\chi,N^\flat)$ of $\Hal^\star$ comprise an
orthogonal basis, and their tuples of eigenvalues have
multiplicity one.  The corresponding automorphic representations
are each shared by $2^{\#\{p\mid N^\flat\}}$ basis elements, and
these have the same eigenvalues for all $v\nmid N^\flat$.
\end{lem}
{\bf Proof} \hspace{4pt}This follows from multiplicity-one for
automorphic representations of $B_\A\ti$ (proved by
Jacquet--Langlands using the trace-formula) and Casselman's
theorem, plus the local calculations below for $p\mid
N^\flat$.\quad$\square$\vspace{10pt}

For $p\nmid d_BN$, $\pi_p$ is equivalent to the unramified
continuous series $\pi(|\,|_p^{\dot s_p }, |\,|_p^{\ddot s_p })$
with central character $\tilde\chi_p$. Since $\pi_p$ is
pre-unitary, either
\begin{align*}
\dot s_p&=-\ddot
s_p-\log_p(\chi(p))\in \textstyle i\R/\frac{2\pi
i}{\log(p)}\Z,\hspace{10pt}\text{or}\vspace{5pt}\\ \dot
s_p&=-\ov{\ddot s_p} =\sigma_p+\ic t_p\,,\hspace{10pt}\text{with}\\
\sigma_p&\in(0,\tfrac12),\hspace{10pt}
t_p\in\tfrac{i}2\log_p(\chi(p))+
\tfrac{\pi}{\log(p)}\Z/\tfrac{2\pi}{\log(p)}\Z.
\end{align*}
Casselman's theorem implies that the right
$\ord_p\ti$-invariant vector $\Psi_p$ is unique up to scaling
in $\pi_p$, and hence corresponds to $V_p^0\in\pi(|\,|_p^{\dot
s_p }, |\,|_p^{\ddot s_p })$ of the form
\begin{align*}
V_p^0\left(\left[\begin{mat}{cc}\dot a &*\\
&\ddot a \end{mat}\right]\kappa\right)= |\dot a |_p^{\dot
s_p }\,|\ddot a |_p^{\ddot s_p }\left|\tfrac{\dot a }{\ddot a
}\right|_p^{\frac12} V_p^0(1)\hspace{20pt}\text{for \,}\kappa\in\ord_p\ti.
\end{align*}
In terms of the parameters, $\Psi_p$ has $T_p(p,1)$ eigenvalue
$\lambda_p=p^{\frac12}\left(p^{-\dot s_p }+p^{-\ddot s_p }\right)$
and $T_p(p)$ eigenvalue $\chi(p)=p^{-\dot s_p -\ddot s_p }$, and
hence by the analysis in \S\ref{hring}, $T_p^{[p^n]}$ eigenvalue
\begin{align*}
\lambda_{p^n}&=p^{\frac{n}2}\left( p^{-n\dot s_p }+p^{-(n-1)\dot
s_p -\ddot s_p }+ \cdots+p^{-n\ddot s_p }\right)\\
&=p^{\frac{n}2}\left( \frac{p^{-(n+1)\dot s_p }-p^{-(n+1)\ddot
s_p }} {p^{-\dot s_p }-p^{-\ddot s_p }}\right).
\end{align*}

If $p\mid N^\flat$, $\pi_p$ is again equivalent to an unramified
continuous series, exactly as above. However by Casselman's
theorem, the subspace of right $\ordpx$-invariant vectors is
spanned by a right $\dot\ord_p\ti$-invariant vector $\dot V_p^0$
and a right $\ddot\ord_p\ti$-invariant vector $\ddot V_p^0$, each
unique up to scaling.  We relate their normalizations by defining
\begin{align*}
\ddot V_p^0=\left[\begin{mat}{cc}1&\\
&p\end{mat}\right]\hspace{-3pt}\Big|\dot V_p^0 &= R_p(\vec
N)\,\dot V_p^0,\\\Longrightarrow\hspace{10pt}\chi(p)\,
\dot V_p^0=T_p(p)\,\dot V_p^0&=R_p(\vec N)\,\ddot V_p^0.
\end{align*}
Using coset representatives, we see that
\begin{align*}
T_p(p,1)\,\dot V_p^0&=\lambda_p\,\dot
V_p^0-\ddot V_p^0,\\
T_p(p,1)\,\ddot V_p^0&=p\,\chi(p)\,\dot V_p^0.
\end{align*}
We have now determined the actions of $R_p(\vec N)$ and
$T_p(p,1)$, and hence also
\begin{align*}
T_p(1,p)&=R_p(\vec N)^{-1}\,T_p(p,1)\,R_p(\vec N),\\
T_p^{[p]}&=T_p(p,1)+T_p(1,p)+R_p(\vec N).
\end{align*}
With respect to the basis $\left[\dot V_p^0,\ddot
V_p^0\right]$, these operators are represented by the matrices
\begin{alignat*}{2}
T_p(p,1)&\sim\left[\begin{mat}{cc}\lambda_p&p \,\chi(p)\\-1&
\end{mat}\right],&\hspace{10pt}
T_p(1,p)&\sim\left[\begin{mat}{rc}&-\chi(p)\\p& \lambda_p
\end{mat}\right],\\
T_p^{[p]}&\sim\left[\begin{mat}{cc}\lambda_p&p\,\chi(p)\\
\hspace{4.7pt}p\hspace{4.7pt} & \lambda_p\end{mat}\right],&
R_p(\vec N) &\sim
\left[\begin{mat}{rc}&\hspace{7.7pt}\chi(p)\\
\hspace{1pt}1\hspace{.5pt}
\end{mat}\right].
\end{alignat*}
The latter two are normal and commute, and so have orthogonal
eigenvectors
\begin{alignat*}{2}
V_p^\flat&:=\dot V_p^0+\ov{\veps_p}\,\ddot V_p^0,&
\veps_p^2&=\chi(p),\\
T_p^{[p]}\,V_p^\flat&=(\lambda_p+p\,\veps_p)V_p^\flat,
\hspace{30pt}& R_p(\vec N)\,V_p^\flat &=\veps_p\,V_p^\flat.
\end{alignat*}
Then by the considerations of
\S\ref{hring},
\begin{align*}
T_p^{[p^n]}\,V_p^\flat=(\lambda_{p^n}+p\,
\veps_p\,\lambda_{p^{n-1}})V_p^\flat.
\end{align*}
The idempotents $\dot T_p(1),\,\ddot T_p(1)$ act as orthogonal
projectors to $\C\,\dot V_p^0,\,\C\,\ddot V_p^0$, and using the
above calculations we see that $\dot T_p(1)\,\ddot V_p^0=
\tfrac{\lambda_p}{p+1}\,\dot V_p^0$,  so
\begin{align*}
\langle\ddot V_p^0,\dot V_p^0\rangle= \tfrac{\lambda_p}{p+1} =
\chi(p)\,\langle \dot V_p^0,\ddot V_p^0\rangle, \hspace{30pt}
\langle V_p^\flat, V_p^\flat\rangle= 2\big(1+
\tfrac{\ov{\veps_p}\,\lambda_p}{p+1}\big).
\end{align*}

For $p\mid N^\sharp$, $\pi_p$ is equivalent to the special
representation $\sigma(|\,|_p^{\dot s_p },|\,|_p^{\ddot s_p })$,
where
\begin{align*}
\dot s_p =-\ov{\ddot s_p }=\tfrac12+\ic t_p\,,
\hspace{30pt} t_p\in\tfrac{\ic}2
\log_p(\chi(p))+\tfrac{\pi}{\log(p)}\Z/\tfrac{2\pi}{\log(p)}\Z.
\end{align*}
This is an irreducible invariant subspace of ${\rm
Ind}_{T_p}^{G_p} (|\,|_p^{\dot s_p },|\,|_p^{\ddot s_p })$,
consisting of $V_p$ s.t.\
\begin{align*}
\int_{\Q_p}V_p\left(\left[\begin{mat}{cc}1&\\x&1\end{mat}\right]
\right)\,dx=0.
\end{align*}
For an $\ord_p\ti$-invariant vector $V_p^\sharp$, this is
equivalent to
\begin{align*}
|\ddot N|_p\,V_p^\sharp(1)+|\dot N|_p\,V_p^\sharp
\left(\left[\begin{mat}{cc}&1\\1&\end{mat}\right]\right)=0,
\end{align*}
and so for $\kappa\in\ord_p\ti$,
\begin{align*}
V_p^\sharp\left(\left[\begin{mat}{cc} \dot a  & * \\ & \ddot a
\end{mat}\right]\kappa\right)&=\phantom{-}|\dot a
\ddot a |_p^{\ic t_p}\left|\tfrac{\dot a }{\ddot a }\right|_p
V_p^\sharp(1),\\
V_p^\sharp\left(\left[\begin{mat}{cc}*&\dot a\\\ddot
a&\end{mat}\right]\kappa\right)&= -|\dot a \ddot a
|_p^{\ic t_p}\left|\tfrac{\dot a\ddot N}{\ddot a\dot
N}\right|_pV_p^\sharp(1).
\end{align*}
Casselman's theorem implies that $\Psi_p$ corresponds to
$V_p^\sharp$ of this form, and thus has $R_p(\vec N)$
eigenvalue $\veps_p=-p^{-\ic t_p}$, $T_p(p)$ eigenvalue
$\veps_p^2=\chi(p)$, both $T_p(p,1)$ and $T_p(1,p)$ eigenvalues
$-\veps_p$. \big(Note $T_p(1,p)=R_p(\vec N)^{-1}\,\,T_p(p,1)
\,R_p(\vec N)$.\big) Thus by the calculations of \S\ref{hring},
$\Psi_p$ has $T_p^{[p^n]}$ eigenvalues
$\lambda_p^{[p^n]}=p^{-n\ic t_p}$.

For $p\mid N^\chi$, $\pi_p$ is equivalent to the pre-unitary
continuous series $\pi(\tilde\chi_p|\,|_p^{\ic t_{p}},
|\,|_p^{-\ic t_{p}})$, with $t_p\in\R/\frac{2\pi}{\log(p)}\Z$,
and the new vector $V_p^\chi$ has the form, for
$\kappa\in\ord_p\ti$,
\begin{align*}
V_p^\chi\left(\left[\begin{mat}{cc} \dot a  & * \\ & \ddot a
\end{mat}\right]\kappa\right)&=\tilde\chi_p(\dot a \kappa)
\left|\tfrac{\dot a }{\ddot a}\right|_p^{\frac12+\ic t_p}
V_p^\chi(1),\\
V_p^\chi\left(\left[\begin{mat}{cc}*&\dot a\\
\ddot a &\end{mat}\right] \kappa\right)&=0.
\end{align*}
Therefore $\Psi_p$ has $T_p^\chi(p^n,1)$ eigenvalues $p^{\frac
n2-n\ic t_p}$, all $T_p^\chi(p^n)$ eigenvalues 1, and hence
$T_p^\chi(1,p^n)$ eigenvalues $p^{\frac{n}2+n\ic t_p}$.

For $p\mid d_B$, $\pi_p$ is an irreducible representation of
$\diva_p\ti$, trivial on $\valr_p\ti$, with unramified central
character $\tilde\chi_p$.  It is easy to see from this that
$\pi_p$ is one-dimensional and given by the character
$|\nu(\cdot)|_p^{\ic t_p}$, with $t_p\in
\tfrac{i}2\log_p(\chi(p))+ \tfrac\pi{\log(p)}\Z/
\tfrac{2\pi}{\log(p)}\Z$\,.  $\Psi_p$ must be a multiple of
this character, and so has $T_p^{[p^n]}$ eigenvalues
$\lambda_p^{[p^n]}=p^{-n\ic t_p}$.  The representation
$\pi_p^{\rm JL}$ of $\GL_2(\Q_p)$ associated to $\pi_p$ by
Jacquet--Langlands is the special representation
$\sigma(|\,|_p^{\frac12+\ic t_p },|\,|_p^{-\frac12+\ic t_p })$.

For $v=\infty$ and $k=0$, we have $\tilde\chi_\infty=1$, and
$\pi_\infty$ is equivalent to the pre-unitary continuous series
$\pi(\sgn^{\delta_\infty}|\,|_\infty^{s_\infty},
\sgn^{\delta_\infty}|\,|_\infty^{-s_\infty}).$ \ $V_\infty^0$
corresponding to $\Psi_\infty$ has the form,
\begin{align*}
V_\infty^0\left(\left[\begin{mat}{cc}\dot a &*\\
&\ddot a \end{mat}\right]\kappa\right)=
\sgn^{\delta_\infty}\!\!\left(\tfrac{\dot a }{\ddot a }\right)
\left|\tfrac{\dot a }{\ddot a }\right|_\infty^{\frac12+s_\infty}
V_\infty^0(1)\hspace{20pt}\text{for \,}\kappa\in\SO(2,\R).
\end{align*}
Thus $\Psi_\infty$ has $T_\infty^-$ eigenvalue
$(-1)^{\delta_\infty}$, $T_\infty^\Delta$ eigenvalue
$\lambda_\infty=(\tfrac14-s_\infty^2)>0$, and
$T_{X_\infty}^{0,\xi}=T_{Y_\infty}^{0,\xi}$ eigenvalues
\begin{align*}
\lambda_\infty^{0,\xi}&= \int_{\R^+\lmod B_\infty^+}
(V_\infty^0(\beta)/V_\infty^0(1))\, \phi_{X_\infty}^{0,\xi}(\beta)
\,\dit\beta \\&= \int_\R\int_{\R^+}\int_0^{2\pi}
y^{\frac12+s_\infty}\,\tfrac1\pi\exp(-\pi\xi(y^2+x^2+1)y^{-1})
\,y^{-2}\,dx\,dy\,d\theta \\&=2\,\xi^{-\frac12}\int_{\R^+}
\exp({-\pi\xi(y+y^{-1})})\,y^{s_\infty-1}\,dy\\&=
4\,\xi^{-\frac12} \,K_{s_\infty}(2\pi\xi).
\end{align*}

We are not considering $|k|=1$, but we mention that in this case
$\pi_\infty$ is equivalent to the continuous series\hspace{3pt}
$\pi(\sgn^{1+\delta_\infty}|\,|_\infty^{s_\infty},
\sgn^{\delta_\infty}|\,|_\infty^{-s_\infty}).$

If $|k|\ge2$, $\pi_\infty$ is equivalent to the discrete series
$\sigma(\sgn^k|\,|_\infty^{s_\infty}, |\,|_\infty^{-s_\infty})$
for\linebreak $s_\infty=\frac{|k|-1}2$. Upon restriction to
$\SL_2(\R)$, this representation decomposes as the direct sum of
the weight-$|k|$ holomorphic and anti-holomorphic discrete series,
and the vector $V_\infty^k$ corresponding to $\Psi_\infty$ is a
lowest-weight vector in either the first or second of these,
depending on whether $k>0$ or $k<0$. Thus $V_\infty^k$ has the
form,
\begin{align*}
V_\infty^k\left(\left[\begin{mat}{cc}\dot a &*\\
&\ddot a\end{mat}\right]\kappa\right)=
\tilde\chi_\infty(\dot a \kappa) \left|\tfrac{\dot a }{\ddot a
}\right|_\infty^{\frac{|k|}2} V_\infty^k(1)
\hspace{20pt}\text{for \,}\kappa\in\SO(2,\R),
\end{align*}
and $\Psi_\infty$ has $T_\infty^\Delta$ eigenvalue
$\lambda_\infty=\frac{|k|(2-|k|)}4$, while
$T_\infty^-V_\infty^{\pm k}=V_\infty^{\mp k}$. We define
\begin{align*}
T_{X_\infty}^{k,\xi}\Psi_\infty=\lambda_{X_\infty}^{k,\xi}
\Psi_\infty,\hspace{30pt}
T_\infty^-T_{Y_\infty}^{k,\xi}\Psi_\infty=
\lambda_{Y_\infty}^{k,\xi}\Psi_\infty.
\end{align*}
Our calculation is similar to that for $k=0$, but we now use
the mean value and conformal mapping properties of harmonic
functions. These eigenvalues can also be computed by
application of Selberg's lemma on invariant integral operators.
\begin{align*}
\lambda_{X_\infty}^{k,\xi}&= \int_\R
\int_{\R^+}\int_0^{2\pi}y^{\frac{|k|}2}
\tfrac1\pi\left(\tfrac{y+1-ix}{2y^{1/2}}\right)^{\und k}
\exp\left(-\pi\xi\left(\tfrac{y^2+x^2+1}y\right)\right)
y^{-2}\,dx\,dy\,d\theta\\
&= \int_\R \int_{\R^+}2\left(\tfrac{1-iz}2\right)^{\und k}\,
\exp(-2\pi\xi\cosh\dist(i,z))\,y^{-2}\,dx\,dy\\
&= \int_\R \int_{\R^+} 2
\exp\left(-\pi\xi\left(\tfrac{y^2+x^2+1}y\right)\right)
y^{-2}\,dx\,dy\\&= 4\,\xi^{-\frac12}K_{-\frac12}(2\pi\xi)\ =\
2\,\xi^{-1}\,\e(i\xi),\displaybreak[0]\\[8pt]
\lambda_{Y_\infty}^{k,\xi}&= \int_\R
\int_{\R^+}\int_0^{2\pi}y^{\frac{|k|}2}\tfrac1\pi
\left(\tfrac{y-1-ix}{2y^{1/2}}\right)^{\und k}
\exp\left(-\pi\xi\left(\tfrac{y^2+x^2+1}y\right)\right)y^{-2}
\,dx\,dy\,d\theta\\ &= \int_\R
\int_{\R^+}2\left(\tfrac{-1-iz}2\right)^{\und k}\,
\exp(-2\pi\xi\cosh\dist(i,z))\,y^{-2}\,dx\,dy =0.
\end{align*}

\section{Whittaker Models}
\label{whit}

In this section, we restrict our attention to $B\ti=\GL_2=G$
(i.e.\ $d_B=1$) and $\ord(1,(1,M))$, $M$ square-free. Let
$\e_\infty(x)=\e(x)=\exp(2\pi ix)$ for $x\in\R$, and define the
character $\e_p$ of $\Q_p$, with kernel $\Z_p$, on $x\in\Z[p^{-1}]
\subset\Q_p$ by $\e_p(x)=\e(-x)$. Thus $\e_\A=\prod_v\e_v$ is
trivial on $\Q$ and unramified.

Any $F\in L^2_0(G_\Q\lmod G_\A,\tilde\chi)$ has a
Fourier--Whittaker series expansion,
\begin{align*}
F(g)={}&\sum_{\xi\in\Q\ti}\widehat{F}\left(
\left[\begin{mat}{cc}\xi&\\&1\end{mat}\right]g\right),\\
\widehat{F}(g):={}&\int_{\Q\lmod \A}F\left(
\left[\begin{mat}{cc}1&x\\&1\end{mat}\right]g\right)
\ov{\e_\A(x)}\,dx.
\intertext{We will see in \S\ref{adjsqz} that}
\langle F,F\rangle={}&\int_{\A\ti}\left|\widehat{F}
\left(\left[\begin{mat}{cc}a&\\
&1\end{mat}\right]\right)\right|^2|a|^{-1}_\A\,da.
\end{align*}
Define the global Whittaker model $\W=\W(\pi)\simeq\pi$ as the
image of $\pi$ under $F\mapsto\widehat{F}$. Then
$\W=\tilde\otimes_v\W_v$, where all $W_v\in\W_v\simeq\pi_v$ are
smooth functions on $G_{\Q_v}$ s.t.
\begin{align*}
W_v\left(\left[\begin{mat}{cc}1&x\\&1\end{mat}\right]
g\right)= \e_v(x)\,W_v(g)\hspace{20pt}\text{for \,}x\in\Q_v.
\end{align*}
We will explicitly compute those $W_v$ having the same right
$K_v$ types as forms in $\tilde
S_k(1,(1,M),\tilde\chi,M^\flat)$.  By the uniqueness of $\W_v$
proved in \cite{jl}, it suffices for us to write down
candidates $W_v^\star$ and check that they are $\Hal_v^\star$
eigenvectors.  Of course we could solve for these functions
directly, by working backwards through the calculations below,
but we find it easier to `guess' them based on the results of
\S\ref{jls}.

An important feature of the Whittaker models is that they endow
forms on $\GL_2$ with an arithmetic normalization distinct from
the spectral one. An eigenform $F$ is called Hecke normalized if
$\widehat{F}=\prod_v W_v^{\star}$. The $W_v^\star$ have been
scaled so that $W_p^\star(1)=1$ and $W_\infty^\star$ has
prescribed asymptotics in the cusp.  We denote by $\check
W_v^\star\in\W^\vee_v$ the Whittaker functions of $\ov F$,
\begin{align*}
\check W_v^\star(g)=\ov{{W{}}^\star_v(\eps g)}\hspace{20pt}\text{for \,}
\eps=\left[\begin{mat}{cc}-1&\\&1\end{mat}\right].
\end{align*}

For $p\nmid M$, we define $W_p^0$ by
\begin{align*}
W_p^0\left(\left[\begin{mat}{cc}\dot a&\\&\ddot a\end{mat}\right]\right)
=\left(\tfrac{|\dot ap|^{\dot s}\,|\ddot a|^{\ddot s}-|\dot
ap|^{\ddot s}\,|\ddot a|^{\dot s}} {|p|^{\dot s}-|p|^{\ddot
s}}\right)\big|\tfrac{\dot a}{\ddot
a}\big|^{\frac12}\,\ind_{\Z_p}\big(\tfrac{\dot a}{\ddot a}\big),
\end{align*}
and check that
\begin{align*}
T_p(p,1)\,W_p^0\left(\left[
\begin{mat}{cc}a&\\&1\end{mat}\right]\right)
&=\sum_{x\,{\rm mod}\,p}W_p^0\left(\left[
\begin{mat}{cc}ap&ax\\&1\end{mat}\right]\right)+
W_p^0\left(\left[
\begin{mat}{cc}a&\\&p\end{mat}\right]\right)\\
&=\sum_{x\,{\rm mod}\,p} \e_p(ax)\left(\tfrac{|ap^2|^{\dot
s}-|ap^2|^{\ddot s}} {|p|^{\dot s}-|p|^{\ddot s}}\right)
|ap|^{\frac12}\,\ind_{\Z_p}(ap)\\
&\hspace{30pt} +\
\,\left(\tfrac{|ap|^{\dot s}\,|p|^{\ddot s}-|ap|^{\ddot
s}\,|p|^{\dot s}} {|p|^{\dot s}-|p|^{\ddot
s}}\right)\big|\tfrac{a}p\big|^{\frac12}\,
\ind_{\Z_p}\big(\tfrac{a}p\big)\\
&=\left(\tfrac{p^{-\dot s}\,|ap|^{\dot s}-p^{-\ddot
s}\,|ap|^{\ddot s}} {p^{-\dot s}-p^{-\ddot s_p}}\right)
p^{\frac12}|a|^{\frac12}\,\ind_{\Z_p}(a)\\
&\hspace{30pt} +\
\,\left(\tfrac{p^{-\ddot s}\,|ap|^{\dot s}- p^{-\dot
s}\,|ap|^{\ddot s}} {p^{-\dot s}-p^{-\ddot
s}}\right)p^{\frac12}|a|^{\frac12}\,\ind_{\Z_p}(a)\\
&=\lambda_p\,W_p^0\left(\left[
\begin{mat}{cc}a&\\&1\end{mat}\right]\right).
\end{align*}

Recall that the case of $p\mid M^\flat$ is essentially the same as
the previous one, but we must consider a two-dimensional space of
oldvectors. Taking $\dot W_p^0$ as above, we have already shown
$\dot W_p^0\in\W_p$\,, and we simply define
\begin{align*}
\ddot W_p^0&=\left[\begin{mat}{cc}1&\\
&p\end{mat}\right]\hspace{-3pt}\Big|\dot W_p^0 =
R_p(1,p)\,\dot W_p^0,\\ W_p^\flat&=\dot
W_p^0+\ov{\veps_p}\,\ddot W_p^0,\hspace{20pt}
\veps_p^2=\chi(p).
\end{align*}

For $p\mid M^\sharp$, we define
\begin{align*}
W_p^\sharp\left(\left[\begin{mat}{cc}\dot a &\\&\ddot a
\end{mat}\right]\right)&=\phantom{-}
|\dot a\ddot a|_p^{\ic t_p}\,\big|\tfrac{\dot a}{\ddot
a}\big|_p\,\ind_{\Z_p}\left(\tfrac{\dot a}{\ddot a}\right),\\
W_p^\sharp\left(\left[\begin{mat}{cc} &\dot a\\ \ddot a&
\end{mat}\right]\right)&=
-|\dot a\ddot a|_p^{\ic t_p}\,\big|\tfrac{\dot ap}{\ddot
a}\big|_p\,\ind_{\Z_p}\big(\tfrac{\dot ap}{\ddot a}\big),
\end{align*}
and check that
\begin{gather*}
\begin{aligned}
R_p(1,p)[W_p^\sharp]\left(\left[\begin{mat}{cc}a &\\&1
\end{mat}\right]\right)&=
W_p^\sharp\left(\left[\begin{mat}{cc}&a\\ p&
\end{mat}\right]\right)\\
&=-p^{-\ic t_p}\,W_p^\sharp\left(\left[\begin{mat}{cc}a &\\&1
\end{mat}\right]\right),
\end{aligned}\displaybreak[0]\\[8pt]
\begin{aligned}
R_p(1,p)[W_p^\sharp]\left(\left[\begin{mat}{cc}&a\\ p&
\end{mat}\right]\right)&=p^{-2\ic t_p}\,
W_p^\sharp\left(\left[\begin{mat}{cc}a &\\&1
\end{mat}\right]\right)\\&= -p^{-\ic t_p}\,
W_p^\sharp\left(\left[\begin{mat}{cc}&a\\ p&
\end{mat}\right]\right),
\end{aligned}\displaybreak[0]\\[8pt]
\begin{aligned}
T_p(p,1)[W_p^\sharp]\left(\left[\begin{mat}{cc}a &\\&1
\end{mat}\right]\right)
&=\sum_{x\,{\rm mod}\,p}W_p^\sharp\left(\left[
\begin{mat}{cc}ap&ax\\&1\end{mat}\right]\right)\\
&=\sum_{x\,{\rm mod}\,p} \e_p(ax)\,|ap|_p^{1+\ic
t_p}\,\ind_{\Z_p}(ap)\\ &= |ap|_p^{1+\ic
t_p}\,p\,\ind_{\Z_p}(a)\ \,=\ \,p^{-\ic t_p}\,
W_p^\sharp\left(\left[\begin{mat}{cc}a &\\&1
\end{mat}\right]\right),
\end{aligned}\displaybreak[0]\\[8pt]
\begin{aligned}
\lefteqn{T_p(p,1)[W_p^\sharp]\left(\left[\begin{mat}{cc}&a\\1&
\end{mat}\right]\right)=\sum_{x\,{\rm
mod}\,p} W_p^\sharp\left(\left[\begin{mat}{cc}&a\\p&x
\end{mat}\right]\right)}\hspace{30pt}\\
&=W_p^\sharp\left(\left[\begin{mat}{cc}&a\\
p&\end{mat}\right]\right)+\sum_{x\not\equiv0}
W_p^\sharp\left(\left[\begin{mat}{cc}1&\tfrac
ax\\&1\end{mat}\right]
\left[\begin{mat}{cc}ap&\\&1\end{mat}\right]
\left[\begin{mat}{rc}-\tfrac{1}{x}&\\p&x
\end{mat}\right]\right)\\
&= -p\,|ap|_p^{1+\ic t_p}\,\ind_{\Z_p}(a)+
\sum_{x\not\equiv0}\e_p(\tfrac{a}{x}) \,|ap|_p^{1+\ic
t_p}\,\ind_{\Z_p}(ap)\\&= -p\,|ap|_p^{1+\ic
t_p}\,\ind_{\Z_p}(a)+ |ap|_p^{1+\ic t_p}\left(p\,\ind_{\Z_p}(a)
-\ind_{\Z_p}(ap)\right)\\&=-|ap|_p^{1+\ic t_p}\,
\ind_{\Z_p}(ap)\ \,=\ \,p^{-\ic t_p}\,W_p^\sharp
\left(\left[\begin{mat}{cc}&a\\1&
\end{mat}\right]\right).
\end{aligned}
\end{gather*}

For $p\mid M^\chi$, we define
\begin{align*}
W_p^\chi\left(\left[\begin{mat}{cc}\dot a &\\&\ddot a
\end{mat}\right]\right)&=\tilde\chi_p(\dot a)\,
\big|\tfrac{\dot a}{\ddot a}\big|_p^{\frac12+\ic t_p}\,
\ind_{\Z_p}\big(\tfrac{\dot a}{\ddot a}\big),\\
W_p^\chi\left(\left[\begin{mat}{cc}&\dot a\\\ddot a&
\end{mat}\right]\right)&=\ov{\gs_p}\,\tilde\chi_p(\ddot a)\,
\big|\tfrac{\dot ap^2}{\ddot a}\big|_p^{\frac12-\ic t_p}\,
\ind_{\Z_p}\big(\tfrac{\dot ap}{\ddot a}\big),\\
\text{for \,}\gs_p&=\int_{\frac1p\Z_p\ti} \ov{\tilde\chi_p(x)}\,\e_p(x)\,dx,
\end{align*}
and check that
\begin{align*}
\lefteqn{T_p^\chi(p,1)[W_p^\chi]\left(\left[\begin{mat}{cc}a &\\&1
\end{mat}\right]\right) = \sum_{x\,{\rm mod}\,p}
W_p^\chi\left(\left[\begin{mat}{cc}ap&ax\\&1
\end{mat}\right]\right)\,\ov{\tilde\chi_p(p)}}\hspace{30pt}\\
&=\sum_{x\,{\rm mod}\,p} \e_p(ax)\,\tilde\chi_p(a)\,
|ap|_p^{\frac12+\ic t_p}\,\ind_{\Z_p}(ap)\\ &=\tilde\chi_p(a)\,
|ap|_p^{\frac12+\ic t_p}\,p\,\ind_{\Z_p}(a)\ \,=\ \,
p^{\frac12-\ic t_p}\, W_p^\chi\left(\left[\begin{mat}{cc}a &\\&1
\end{mat}\right]\right),\displaybreak[0]\\[8pt]
\lefteqn{T_p^\chi(p,1)[W_p^\chi]\left(\left[\begin{mat}{cc}&a\\
1&\end{mat}\right]\right) =  \sum_{x\,{\rm mod}\,p}
W_p^\chi\left(\left[\begin{mat}{cc}&a\\
p&x\end{mat}\right]\right)\ov{\tilde\chi_p(p)}}\hspace{30pt}\\
&=W_p^\chi\left(\left[\begin{mat}{cc}&
a\\p&\end{mat}\right]\tfrac1p\right)+
\sum_{x\not\equiv0}W_p^\chi\left(\left[\begin{mat}{cc}1&\tfrac
ax\\&1\end{mat}\right]
\left[\begin{mat}{cc}ap&\\&1\end{mat}\right]
\left[\begin{mat}{rc}-\tfrac{1}{x}&\\p&x\end{mat}\right]
\tfrac1p\right)\\
&=\ov{\gs_p}\,|ap|_p^{\frac12-\ic t_p}\,
\ind_{\Z_p}(a)+\sum_{x\not\equiv0}
\e_p(\tfrac{a}{x})\,\tilde\chi_p(-\tfrac
ax)\,|ap|_p^{\frac12+\ic t_p}\, \ind_{\Z_p}(ap)\\
&=\ov{\gs_p}\,|ap|_p^{\frac12-\ic t_p}\,
\ind_{\Z_p}(a)+\ov{\gs_p}\,\ind_{\Z_p\ti}(ap)\ \,=\ \,
p^{\frac12-\ic t_p}\,W_p^\chi\left(\left[\begin{mat}{cc}&a\\
1&\end{mat}\right]\right).
\end{align*}

For $v=\infty$, recall the classical Whittaker function
$w_{\kappa,\mu}$ which uniquely solves
\begin{gather*}
w^{\prime\prime}_{\kappa,\mu}(y)+\left(-\frac14+\frac\kappa y
+\frac{\tfrac14-\mu^2}{y^2}\right)w_{\kappa,\mu}(y)=0\\[6pt]
\text{under}\hspace{10pt}w_{\kappa,\mu}(y)\sim y^\kappa
e^{-\frac y2}\hspace{10pt}\text{as }y\to\infty.
\end{gather*}
In particular,
\begin{align*}
w_{\kappa,\mu}(y)&=\frac{y^{\mu+\frac12}\, e^{-\frac y2}}
{\Gamma(\mu-\kappa+\frac12)}\int_0^\infty e^{-yt}\,
t^{\mu-\kappa-\frac12}\, (1+t)^{\mu+\kappa-\frac12}\,dt,\\
w_{0,\mu}(y)&=\big(\tfrac y\pi\big)^{\frac12}\,K_\mu(\tfrac y2),
\hspace{30pt}w_{\frac k2,\frac{k-1}2}(y)=y^{\frac k2}\,e^{-\frac y2}.
\end{align*}

If $k=0$, define $W_\infty^0\in\W_\infty$ by
\begin{align*}
W_\infty^0\left(\left[\begin{mat}{cc}\dot a &\\&\ddot a
\end{mat}\right]\right)=\sgn^{\delta_\infty}\big(\tfrac{\dot a}{\ddot
a}\big)\,w_{0,s_\infty}\big(\tfrac{\dot a}{\ddot a}\big).
\end{align*}
This has moderate growth and satisfies
\begin{align*}
T_\infty^-[W_\infty^0]&=(-1)^{\delta_\infty}\,W_\infty^0\,,\\
T_\infty^\Delta[W_\infty^0]&=(\tfrac14-s_\infty^2)\,W_\infty^0\,.
\end{align*}
If $|k|\ge 2$, define $W_\infty^k\in\W_\infty$ by
\begin{align*}
W_\infty^k\left(\left[\begin{mat}{cc}\dot a &\\&\ddot a
\end{mat}\right]\right)=\ind_{k\R^+}\big(\tfrac{\dot a}{\ddot a}\big)\,
w_{\frac{|k|}2,\frac{|k|-1}2}\big(\tfrac{\dot a}{\ddot a}\big)\,,
\end{align*}
which has moderate growth and satisfies, for $\pm k\ge2$,
\begin{align*}
R_k^\mp[W_\infty^k]&=0,\\
T_\infty^\Delta[W_\infty^k]&=\tfrac{|k|(2-|k|)}4\,W_\infty^k\,.
\end{align*}

\chapter{Theta Lifting}

\section{Basic Setup}

\subsection{Groups}

Following \cite{hk1}, let $G=\GL_2$\,, $H=\GSp_6$\,, $\Gb=G^3\cap
H$ as linear algebraic groups over $\Q$.  The similitude character
of $H$ and its restriction to $\Gb$ will be denoted $\nu$ (and
distinguished by context from the reduced norm of $B$).  Also
consider $B$,\, $B\ti$,\, $H'=\GO(B,\nu)$,\,
$\Gbp=H^{\prime3}\cap\GO(B^3,\oplus^3\nu)$ as linear algebraic
groups over $\Q$, again denoting the similitude characters by
$\nu$. Write $\breve H'$ for the connected component of the
identity in $H'$, and define $\rho:B\ti\times B\ti
\rightarrow\breve H'$\, by $\rho(\dot\alpha ,\ddot\alpha)(\beta)=
\dot\alpha\beta \ddot\alpha ^{-1}.$ We see there is an exact
sequence
\begin{align*}
1\longrightarrow {\rm Z}(B\ti)\stackrel{\diag}{\longrightarrow}
(B\ti\times B\ti) \stackrel{\rho}{\longrightarrow} \breve
H'\longrightarrow 1.
\end{align*}
We may view the anti-involution $\iota$ as an element of order
two in $H'_\Q$, $\iota(\beta)=\beta^\iota$. It and $\breve H'$
generate $H'=\breve H'\rtimes\langle\iota\rangle$ as a
semi-direct product, $\rho(\dot\alpha ,\ddot\alpha
)\,\iota=\iota\,\rho\left(\tfrac{\ddot\alpha}{\nu(\ddot\alpha)},
\tfrac{\dot\alpha}{\nu(\dot\alpha)}\right)$.  Note that
\begin{gather*}
\textstyle{\prod_v(\Z/2\Z)}\stackrel\sim\longrightarrow
\langle\iota\rangle_\A,\\
(\delta_v)\mapsto(\iota_v^{\delta_v}).
\end{gather*}
To compute inner-products on $H'$, we will use the
parametrization
\begin{align*}
\rho:PB\ti\times PB\ti\stackrel\sim\longrightarrow P\breve H'.
\end{align*}
We will also need Shimizu's parametrization,
\begin{gather*}
\begin{aligned}
B\ti\rtimes PB\ti &{}\stackrel\sim\longrightarrow\breve H',\\
B^{(\delta)}\times PB\ti &{}\stackrel{\text{1-1}}
\longrightarrow\breve H^{\prime(\delta)},
\end{aligned}\\
(\dot\beta,\ddot\beta)\longmapsto
\rho(\dot\beta\ddot\beta,\ddot\beta).
\end{gather*}

\subsection{Measures}
\label{tmeas}

We now describe the canonical Tamagawa measures on each of
these groups, as in \cite{weil2,vig}.  Let $dx=dx_v$ be the
Haar measure on $\Q_v$ such that the Fourier transform is
self-dual \big(${\rm vol}(\R/\Z)={\rm vol}(\Z_p)=1={\rm
vol}(\A/\Q)$\big): For $\vp\in\Sz(\Q_v)$,
\begin{align*}
\F\vp(y)&=\int_{\Q_v}\vp(y)\,\e_v(xy)\,dx=\F^{-1}\vp(-y).
\end{align*}
Then define $d\ti\!x_v=\tfrac{\zeta_v(1)}{|x|_v}dx_v$, so that
${\rm vol}\ti\!(\R^+/e^\Z)={\rm vol}\ti\!(\Z_p\ti)=1$.

Identify $B_v$ with its algebraic dual $B_v^*$ using the
non-degenerate symmetric bilinear form $\langle
\alpha,\beta\rangle= \tr(\alpha\beta^\iota)=
\alpha\beta^\iota+\beta\alpha^\iota$.  Now choose the Haar
measure $d\alpha=d\alpha_v$ on $B_v$ so that the
`$\iota$-twisted' Fourier transform is self-dual: For
$\vp\in\Sz(B_v)$,
\begin{align*}
\F\vp(\beta)&=\int_{B_v}\vp(\alpha)\,\e_v(\langle\alpha,\beta\rangle)
\,d\alpha=\F^{-1}\vp(-\beta).
\end{align*}
If $v\nmid d_B$, $d\alpha=\prod_{ij}d\alpha_{ij}$, while for
$p\mid d_B$, ${\rm vol}(\valr_p)=\tfrac1p$\,. It follows from
calculating ${\rm vol}(B_\infty/\ord(d_B,\vec N))=d_BN$\,
that\, ${\rm vol}(B_\A/B_\Q)=1$.

As before, define
$d\ti\!\alpha_v=\tfrac{\zeta_v(1)}{|\alpha|_v}d\alpha_v$. It is
well known and follows from the calculations in \S\ref{hops} that
\begin{align*}
\volx(\ordx_p(d_B,\vec N))=\zeta_p(2)^{-1}
\begin{cases}
(p-1)^{-1} & \text{if }p\mid d_B,\\ (p+1)^{-1} & \text{if }p\mid N, \\
\phantom{(}1 & \text{if }p\nmid d_BN.
\end{cases}
\end{align*}
On $B_\infty$, $\dti\alpha_\infty=\dti a_1\,\dti a_2\, dx\,
d\theta$ in terms of the coordinates
\begin{align*}
\alpha_\infty= \left[\begin{mat}{cc}a_1&\\&a_2\end{mat}\right]
\left[\begin{mat}{cc}1&x\\&1\end{mat}\right]\kappa_\theta,
\hspace{20pt}\text{for \,}0\le\theta<\pi.
\end{align*}
We will also write $\dti\alpha_v$ for the measure on $PB\ti_v$
compatible with the exact sequence
\begin{align*}
1\longrightarrow \Q_v\ti\longrightarrow B\ti_v\longrightarrow PB\ti_v
\longrightarrow 1,
\end{align*}
and $\don\alpha_v$ for the measure on $B^{(1)}_v$ compatible
with
\begin{align*}
1\longrightarrow B^{(1)}_v \longrightarrow B\ti_v
\stackrel\nu\longrightarrow \Q_v\ti \longrightarrow 1.
\end{align*}
To be explicit,
\begin{align*}
\volx(P\ordx_p(d_B,\vec N))=\volx(\ordx_p(d_B,\vec N)),
\end{align*}
while $\dti\alpha_\infty=\dti a\, dx\, d\theta$ in the
$PB\ti_\infty$ coordinates
\begin{align*}
\alpha_\infty=\left[\begin{mat}{cc}a&\\&1\end{mat}\right]
\left[\begin{mat}{cc}1&x\\&1\end{mat}\right]\kappa_\theta,
\hspace{20pt}\text{for \,}0\le\theta<\pi.
\end{align*}
Also
\begin{align*}
\volo(\ord^{(1)}_p(d_B,\vec N))=\volx(\ordx_p(d_B,\vec N)),
\end{align*}
while $\don\alpha_\infty=\dti a\, dx\, d\theta$ in the
$B^{(1)}_\infty$ coordinates
\begin{align*}
\alpha_\infty= \left[\begin{mat}{cc}a&\\&a^{-1}\end{mat}\right]
\left[\begin{mat}{cc}1&x\\&1\end{mat}\right]\kappa_\theta,
\hspace{20pt}\text{for \,}0\le\theta<\pi.
\end{align*}
It is well known (see \cite{vig}) that $\volx(PB\ti_\Q\lmod
PB\ti_\A)=2$ and $\volo(B^{(1)}_\Q\lmod B^{(1)}_\A)=1$.  The
difference between $\dti \alpha_\infty$ and $\don\alpha_\infty$
may be seen as follows: Lift $\dti\alpha_\infty$ to the double
cover $\R^+\lmod B_\infty\ti$ of $PB_\infty\ti$.  Then
$\tfrac12\,\dti\alpha_\infty$ is compatible with
\begin{align*}1\longrightarrow B_\infty^{(1)}\longrightarrow \R^+\lmod B_\infty\ti
\stackrel\nu\longrightarrow \R^+\lmod\R\ti \longrightarrow1.\end{align*}

The Tamagawa measures on $G=\GL_2$, $PG=\PGL_2$, $G^{(1)}=\SL_2$
are special cases of those above, for $d_B=1$.  Now define
$d\gb_v$ on $\Gb_v$ to be compatible with
\begin{align*}
1\longrightarrow\big(G^{(1)}_v\big)^3\longrightarrow \Gb_v
\stackrel\nu\longrightarrow \Q\ti_v \longrightarrow1.
\end{align*}
Explicitly,
\begin{align*}
\vol(\ordx_p(1,\vec N)^3\cap\Gb_p)= \volx(\ordx_p(1,\vec N))^3,
\end{align*}
and $d\gb_\infty=\dti a\prod_j(\dti a_j\, dx_j\, d\theta_j)$ in
the $\Gb_\infty$ coordinates
\begin{align*}
\gb_{\infty j}=\left[\begin{mat}{cc}a&\\&1\end{mat}\right]
\left[\begin{mat}{cc}a_j&\\&a_j^{-1}\end{mat}\right]
\left[\begin{mat}{cc}1&x_j\\&1\end{mat}\right]\kappa_{\theta_j},
\hspace{20pt}\text{for \,}0\le\theta_j<\pi.
\end{align*}
Also define $d\gb_v$ on $P\Gb_v$ to be compatible with
\begin{align*}
1\longrightarrow \Q_v\ti\longrightarrow \Gb_v
\longrightarrow P\Gb_v \longrightarrow1.
\end{align*}
Then
\begin{align*}
\vol(P(\ordx_p(1,\vec N)^3\cap\Gb_p))=\volx(\ordx_p(1,\vec N))^3,
\end{align*}
and $d\gb_\infty=2\prod_j(\dti a_j\, dx_j\, d\theta_j)$ in the
$P\Gb_\infty$ coordinates
\begin{gather*}
\gb_{\infty j}=\left[\begin{mat}{cc}\eps&\\&1\end{mat}\right]
\left[\begin{mat}{cc}a_j&\\&a_j^{-1}\end{mat}\right]
\left[\begin{mat}{cc}1&x_j\\&1\end{mat}\right]\kappa_{\theta_j},\\[8pt]
\text{for \,}\eps^2=1,\hspace{10pt}a_1a_2a_3>0,\hspace{10pt}
0\le\theta_j<\pi.
\end{gather*}
It is easy to check that $\vol(P\Gb_\Q\lmod P\Gb_\A)=2$.

\subsection{Dual-Pairs}
\label{wrep}

Extend the definition of $\langle\cdot\,,\cdot\rangle$ to $B^n_v$
as an orthogonal direct sum, and define $\F$, $d\alpha_v$ on
$B^n_v$ as before. We may write $\F_j$ for the Fourier transform
in $\beta_j$ alone, $\beta=(\beta_j)\in B^n_v$. The Weil
representation $\omega$ of ${\rm Sp}_{2n}(\Q_v)$ on $\Sz(B^n_v)$
is uniquely determined by: \cite{jl,sz}
\begin{alignat*}{2}
\omega\left(\left[\begin{mat}{cc} 1_n & U \\  & 1_n
\end{mat}\right]\right) \vp(\beta)&=\e_v(\hf\langle
U\beta,\beta\rangle)\,\vp(\beta)&\hspace{20pt}
&\text{for \,}U=U^t\in\M_n(\Q_v),\\
\omega\left(\left[\begin{mat}{cc} A^t & \\ & A^{-1}
\end{mat}\right]\right) \vp(\beta)&=|\det A|_v^2\, \vp(A
\beta)&&\text{for \,}A\in \GL_n(\Q_v),\\
\omega\left(\left[\begin{mat}{cc} & 1_n \\ -1_n &
\end{mat}\right]\right)\vp(\beta)&= (-1)^{n\ind_{v\mid d_B}}\,
\F\vp(\beta).
\end{alignat*}

Now consider the similitude dual-pairs
\begin{align*}
R(X,Y)&=\{(x,y)\in X\times Y\,;\,\nu(x)=\nu(y)\},\\
\text{for \,}(n,X,Y)&=(1,G,H'),\,(3,\Gb,\Gb'),\,(3,H,H').
\end{align*}
We extend the definition of $\omega$ to each of these as follows:
Let $L$ denote the unitary left regular representation of $Y_v$ on
$\Sz(B^n_v)$,
\begin{align*}
L(h')\,\vp(\beta)&=
|\nu(h')|^{-1}_v\,\vp\left((h')^{-1}\beta\right),\\
L(\gb')\,\vpb(\betab)&=
|\nu(\gb')|^{-3}_v\,\vpb\left((h'_j)^{-1}\beta_j\right),\\
L(h')\,\vpb(\betab)&=
|\nu(h')|^{-3}_v\,\vpb\left((h')^{-1}\beta_j\right).
\end{align*}
If $\nu(x,y)=\delta$, set $\alpha=\left[\begin{mat}{cc}1_n&\\
&\delta\,1_n\end{mat}\right]\in X_v$, $x^{(1)}=x\alpha^{-1}$,
${}^{(1)}\!x=\alpha^{-1}x$. Then
\begin{align*}
\omega(x,y)=\omega(x^{(1)})\,L(y)=L(y)\,\omega({}^{(1)}\!x)
\end{align*}
defines a representation of $R(X,Y)_v$ on $\Sz(B^n_v)$.

Note $\omega(z,z)=1$ for $z\in\Q_v\ti$, since
$\omega\left(\left[\begin{mat}{cc}z\,1_n&\\&z^{-1}\,1_n\end{mat}\right]\right)
= L(z)^{-1}.$

\section{Jacquet--Langlands Correspondence}
\label{jls} In this section we explicitly compute the Shimizu
theta lift and its adjoint, which realize the
Jacquet--Langlands correspondence.

\subsection{Shimizu's Theta Lift}

For $\dot\Psi,\ddot\Psi$ as in \S\ref{autf} and
$\eta\in\prod_v(\Z/2\Z)$, define $\breve F'\in L^2_0(\breve
H'_\Q\lmod\breve H'_\A,\tilde\chi)$ by
\begin{align*}
\breve F'(\rho(\dot\beta,\ddot\beta))=\dot\Psi^\eta(\dot\beta)\,
\ov{\mbox{$\displaystyle\ddot\Psi^\eta$}(\ddot\beta)},
\end{align*}
with matching local data $\breve\vp\in\Sz(B_\A)$ as in
\S\ref{locdata}. Shimizu's theta lift is defined in \cite{sz}:
For $g\in G^{(\delta)}_\A$,
\begin{align*}
\breve\Theta'_{\breve\vp}(\breve F')(g)&=\tfrac12\int_{\breve
H_\Q^{\prime(1)}\lmod \breve H_\A^{\prime(\delta)}}\
\sum_{\alpha\in B_\Q} \omega(g,h')\breve\vp(\alpha)\,\breve
F'(h')\,dh'\\&=\tfrac12\int_{PB_\Q\ti\lmod PB_\A\ti
}\int_{B_\Q^{(1)}\lmod B_\A^{(\delta)}} \sum_{\alpha\in B_\Q}
\omega(g,\rho(\dot\beta\ddot\beta,\ddot\beta)) \breve\vp(\alpha)
\\&\hspace{30pt}
\cdot\,\dot\Psi^\eta(\dot\beta\ddot\beta)
\ov{\mbox{$\displaystyle\ddot\Psi^\eta$}(\ddot\beta)}
\,\don\dot\beta\, \dti\ddot\beta\\&=\tfrac12\int_{PB_\Q\ti\lmod
PB_\A\ti }\int_{B_\Q^{(1)}\lmod B_\A^{(\delta)}} \sum_{\alpha\in
B_\Q} |\delta|_\A^{-1}\omega({}^{(1)}\!g)
\breve\vp((\dot\beta\ddot\beta)^{-1}\alpha\ddot\beta)
\\&\hspace{30pt}
\cdot\,\dot\Psi^\eta(\dot\beta\ddot\beta)
\ov{\mbox{$\displaystyle\ddot\Psi^\eta$}(\ddot\beta)}
\,\don\dot\beta\,\dti\ddot\beta.
\end{align*}
This expression is readily converted into a Fourier series,
following \cite{sz}. First, decompose the sum over $B_\Q$ into
left $B_\Q^{(1)}$ orbits. Note that $B_\Q^{(1)}\lmod
B_\Q\ti\simeq\Q\ti$ by $\nu$, and if $B_\Q$ is a division algebra,
$B_\Q^{(0)}=\{0\}$. In the unramified case $B=M$\,, the remaining
elements are those $\alpha\in M_\Q$ of rank 1, and every such can
be written as $\alpha=\gamma^{-1}\tilde\alpha,$ for
representatives of unique classes $\tilde\alpha\in\Prj
\left[\begin{mat}{cc}\Q&\Q\\0&0\end{mat}\right]$, $\gamma\in
N_\Q\lmod{\rm SL}_2(\Q)$.

The contribution to $\breve\Theta'_{\breve\vp}(\breve F')$ from
$\alpha=0$ is proportional to
\begin{align*}
\int_{PB_\Q\ti\lmod PB_\A\ti} \left(\int_{B_\Q^{(1)}\lmod
B_\A^{(\delta)}} \dot\Psi^\eta(\dot\beta\ddot\beta)
\,\don\dot\beta\right)\ov{\mbox{$\displaystyle\ddot\Psi^\eta$}(\ddot\beta)}
\,\dti\ddot\beta\quad=\quad0,
\end{align*}
since $\dot\pi$ is not one-dimensional, by the argument in
\cite{sz}. For \hbox{$B=M$} (which is not considered in
\cite{sz}, but see \cite{gs}), the contribution from all
$\alpha\in M_{\Q}^{(0)}\setminus\{0\}$ equals
\begin{align*}
\lefteqn{\tfrac12\int_{PG_\Q\lmod PG_\A} \int_{G_\Q^{(1)}\lmod
G_\A^{(\delta)}} \ \sum_{\tilde\alpha}\ \sum_{\gamma\in N_\Q\lmod
G_\Q^{(1)}}}\hspace{20pt}&\\
&\hspace{40pt} |\delta|_\A^{-1}\omega({}^{(1)}\!g)
\breve\vp((\gamma\dot\beta
\ddot\beta)^{-1}\tilde\alpha\ddot\beta)\, \dot\Psi^\eta(\dot\beta
\ddot\beta)\ov{\mbox{$\displaystyle\ddot\Psi^\eta$}(\ddot\beta)}
\,\don\dot\beta\,\dti\ddot\beta\displaybreak[0]\\
&=\sum_{\tilde\alpha}\, \tfrac12\int_{PG_\Q\lmod PG_\A}
\int_{N_\Q\lmod G_\A^{(\delta)}}\\
&\hspace{40pt}|\delta|_\A^{-1}\omega({}^{(1)}\!g) \breve\vp((\dot\beta
\ddot\beta)^{-1}\tilde\alpha\ddot\beta)\, \dot\Psi^\eta(\dot\beta
\ddot\beta)\ov{\mbox{$\displaystyle\ddot\Psi^\eta$}(\ddot\beta)}
\,\don\dot\beta\,\dti\ddot\beta\\
&=\sum_{\tilde\alpha}\,
\tfrac12\int_{PG_\Q\lmod PG_\A} \int_{N_\A\lmod
G_\A^{(\delta)}}|\delta|_\A^{-1}\omega({}^{(1)}\!g)
\breve\vp((\dot\beta \ddot\beta)^{-1} \tilde\alpha\ddot\beta)\\
&\hspace{40pt}\cdot\left(\int_{N_\Q\lmod N_\A}
\dot\Psi^\eta(n\dot\beta\ddot\beta)
\,dn\right)\ov{\mbox{$\displaystyle\ddot\Psi^\eta$}(\ddot\beta)}
\,\don\dot\beta\,\dti\ddot\beta\quad=\quad0,
\end{align*}
since $\dot\pi$ is cuspidal. Therefore in all cases under
consideration,
\begin{align*}
\breve\Theta'_{\breve\vp}(\breve F')(g)={}&
\tfrac12\int_{PB_\Q\ti\lmod PB_\A\ti }\int_{B_\Q^{(1)}\lmod
B_\A^{(\delta)}}\sum_{\tilde\alpha\in B_\Q^{(1)}\lmod B_\Q\ti}\
\sum_{\gamma\in B_\Q^{(1)}}\\
&\hspace{30pt}\omega(g,\rho(\gamma\dot\beta\ddot\beta,\ddot\beta))
\breve\vp(\tilde\alpha)\,\dot\Psi^\eta(\dot\beta
\ddot\beta)\ov{\mbox{$\displaystyle\ddot\Psi^\eta$}(\ddot\beta)}
\,\don\dot\beta\,\dti\ddot\beta\\
={}&\sum_{\tilde\alpha\in B_\Q^{(1)}\lmod B_\Q\ti}
\tfrac12\int_{PB_\Q\ti\lmod PB_\A\ti} \int_{B_\A^{(\delta)}}\\
&\hspace{30pt} \omega(g,\rho(\dot\beta
\ddot\beta,\ddot\beta)) \breve\vp(\tilde\alpha)\,
\dot\Psi^\eta(\dot\beta \ddot\beta)
\ov{\mbox{$\displaystyle\ddot\Psi^\eta$}(\ddot\beta)}
\,\don\dot\beta\,\dti\ddot\beta\\
={}&\sum_{\tilde\alpha\in B_\Q^{(1)}\lmod B_\Q\ti}
\tfrac12\int_{PB_\Q\ti\lmod PB_\A\ti} \int_{B_\A^{(\delta)}}
\omega\left(\left[\begin{mat}{cc}\nu(\tilde\alpha)&\\&1\end{mat}\right],
\,\rho(\tilde\alpha^\iota,1)\right)\circ\\
&\hspace{30pt}\omega(g,\rho(\dot\beta \ddot\beta,\ddot\beta))
\breve\vp(1)\, \dot\Psi^\eta(\dot\beta
\ddot\beta)\ov{\mbox{$\displaystyle\ddot\Psi^\eta$}(\ddot\beta)}
\,\don\dot\beta\,\dti\ddot\beta\\
={}&\sum_{\xi\in\Q\ti}\widehat{\breve\Theta'_{\breve\vp}(\breve
F')}\left( \left[\begin{mat}{cc}\xi&\\&1\end{mat}\right]
g\right),\quad\text{where}\\[8pt]
\widehat{\breve\Theta'_{\breve\vp}(\breve F')}(g)
:={}&\tfrac12\int_{PB_\Q\ti\lmod PB_\A\ti} \int_{B_\A^{(\delta)}}
\omega(g,\rho(\dot\beta \ddot\beta,\ddot\beta))\breve\vp(1)\\
&\hspace{30pt}\cdot\dot\Psi^\eta(\dot\beta\ddot\beta)
\ov{\mbox{$\displaystyle\ddot\Psi^\eta$}(\ddot\beta)}
\,\don\dot\beta\,\dti\ddot\beta\\
={}&\tfrac12\int_{PB_\Q\ti\lmod PB_\A\ti} \int_{B_\A^{(\delta)}}
\omega(g,\rho(\ddot\beta\dot\beta ,\ddot\beta))\breve\vp(1)\\
&\hspace{30pt}\cdot\dot\Psi^\eta(\ddot\beta\dot\beta)
\ov{\mbox{$\displaystyle\ddot\Psi^\eta$}(\ddot\beta)}
\,\don\dot\beta\,\dti\ddot\beta\\
={}&\tfrac12\int_{PB_\Q\ti\lmod PB_\A\ti}\tilde\Psi^\eta(g,\beta)
\ov{\mbox{$\displaystyle\ddot\Psi^\eta$}(\beta)}\,\dti\beta,\\[8pt]
\tilde\Psi^\eta(g,\beta):={}&\int_{B_\A^{(\delta)}}\omega(g,\rho
(\alpha,1))\breve\vp(1)\,\dot\Psi^\eta(\beta\alpha)\,\don\alpha.
\end{align*}

At this point our proof diverges from \cite{sz}.  Using the
calculations in \S\ref{locdata}, we realize the last integral as
the action of Hecke operators on $\dot\Psi^\eta$ and show
\begin{align*}
\tilde\Psi^\eta(g,\beta)=\widehat{F}(g)\,\dot\Psi^\eta(\beta),
\end{align*}
where $F$ is the unique Hecke normalized $\Hal^\star$ eigenform
\begin{gather*}
F\in\tilde S_{|k|}(1,(1,d_B N),\tilde\chi,N^\flat)\cap\pi^{\rm
JL}(\dot\pi)
\hspace{20pt}\text{s.t.}\\[2pt]
R_p(1,d_BN)\,F=\dot\veps_p\,F\hspace{20pt}\text{for all \,}
p\mid N^\flat,\\[-6pt]
\intertext{and so\vspace{-6pt}}
\breve\Theta'_{\breve\vp}(\breve F')= \langle
\dot\Psi^\eta,\ddot\Psi^\eta\rangle\,F = \langle
\dot\Psi,\ddot\Psi\rangle\,F.
\end{gather*}

It is clear from the definition of $\omega$ that
\begin{align*}
\tilde\Psi^\eta\left(\left[\begin{mat}{cc}1&x\\&1\end{mat}\right]
g,\beta\right)=\e_\A(x)\,\tilde\Psi^\eta(g,\beta)\hspace{20pt}
\text{for \,}x\in\A,
\end{align*}
and since $\omega(z,\rho(z,1))=1$,
\begin{align*}
\tilde\Psi^\eta(zg,\beta)=\tilde\chi(z)\,\tilde\Psi^\eta(g,\beta)
\hspace{20pt}\text{for \,}z\in\A\ti.
\end{align*}
Furthermore, the right $K_v$-types of $\tilde\Psi^\eta(g,\beta)$
are determined by Lemma \ref{ktype} in \S\ref{locdata}. The
remainder of our calculation consists identifying the Whittaker
functions from \S\ref{whit} (with $M=d_BN$, $M^\flat=N^\flat$,
$M^\sharp=d_BN^\sharp$, $M^\chi=N^\chi$).

If $p\nmid d_BN$, it suffices to consider
$g=\left[\begin{mat}{cc}a&\\&1\end{mat}\right]$\,:
\begin{gather*}
\omega\left(g,\rho(\alpha,1)\right)\breve\vp_p(1)=
|a|_p\,\breve\vp_p(\alpha^\iota)= |a|_p\,\vp_p(\alpha),\\[8pt]
\begin{aligned}
\tilde\Psi_p^\eta\left(g,\beta\right)&=
|a|_p\int_{B_p^{(a)}}\vp_p(\alpha)\,
\dot\Psi_p^\eta(\beta\alpha)\,d_p^{(1)}\alpha\\&=
|a|_p\int_{B_p^{[a]}}\vp_p(\alpha)\,
\dot\Psi_p^\eta(\beta\alpha)\,\dpt\alpha\\
&=|a|_p\,T_p^{[a]}\dot\Psi_p^\eta(\beta)\\
&=|a|_p^{\frac12}\left( \tfrac{|ap|_p^{\dot s_p}-|ap|_p^{\ddot
s_p}} {|p|_p^{\dot s_p}-|p|_p^{\ddot
s_p}}\right)\ind_{\Z_p}(a)\,
\dot\Psi_p^\eta(\beta)\\
&=W_p^0(g)\,\dot\Psi_p^\eta(\beta).
\end{aligned}
\end{gather*}

If $p\mid N^\flat$, we consider $g$ as before and also
$g'=\left[\begin{mat}{cc}&a\\-1&\end{mat}\right]$\,:
\begin{gather*}
\omega(g,\rho(\alpha,1))\breve\vp_p(1)=
|a|_p\,\breve\vp_p(\alpha^\iota)=|a|_p\,\vp_p(\alpha),
\displaybreak[0]\\[8pt]
\begin{aligned}
\tilde\Psi_p^\eta(g,\beta)&= |a|_p\,
T_p^{[a]}\,\dot\Psi_p^\eta(\beta)\\
&=|a|_p^{\frac12}\left( \tfrac{|ap|_p^{\dot s_p}-|ap|_p^{\ddot
s_p}} {|p|_p^{\dot s_p}-|p|_p^{\ddot s_p}}\right)\ind_{\Z_p}(a)\,
\dot\Psi_p^\eta(\beta)\\
&\hspace{40pt}+\big|\tfrac{a}p\big|_p^{\frac12}
\left(\tfrac{|a|_p^{\dot s_p}-|a|_p^{\ddot s_p}}
{|p|_p^{\dot s_p}-|p|_p^{\ddot s_p}}\right)\ind_{\Z_p}
\big(\tfrac ap\big)\,R_p(\vec N)\dot\Psi_p^\eta(\beta)\\
&=\big(\dot W_p^0(g)+\ov{\dot\veps_p}\,\ddot W_p^0(g)\big)
\dot\Psi_p^\eta(\beta),
\end{aligned}
\displaybreak[0]\\[20pt]
\omega(g',\rho(\alpha,1))\breve\vp_p(1)=
|a|_p\,\F\breve\vp_p(\alpha^\iota) =|a|_p\,\F\vp_p(\alpha),
\displaybreak[0]\\[8pt]
\begin{aligned}
\tilde\Psi_p^\eta(g',\beta)&=|a|_p\int_{B_p^{[a]}}\F\vp_p(\alpha)\,
\dot\Psi_p^\eta(\beta\alpha)\,\dpt\alpha\\
&=|ap|_p\,R_p(\vec N)^{-1}\,T_p^{[ap]}\, \dot\Psi_p^\eta(\beta)\\
&=|ap|_p^{\frac12}\left( \tfrac{|ap^2|_p^{\dot
s_p}-|ap^2|_p^{\ddot s_p}} {|p|_p^{\dot s_p}-|p|_p^{\ddot
s_p}}\right)\ind_{\Z_p}(ap)\, R_p(\vec N)^{-1}
\dot\Psi_p^\eta(\beta)\\
&\hspace{40pt}+ |a|_p^{\frac12}\left( \tfrac{|ap|_p^{\dot
s_p}-|ap|_p^{\ddot s_p}} {|p|_p^{\dot s_p}-|p|_p^{\ddot
s_p}}\right)\ind_{\Z_p}(a)\, \dot\Psi_p^\eta(\beta)\\ &=
\big(\dot W_p^0(g')+\ov{\dot\veps_p}\,\ddot W_p^0(g')\big)
\dot\Psi_p^\eta(\beta).
\end{aligned}
\end{gather*}

If $p\mid N^\sharp$,
\begin{gather*}
\omega(g,\rho(\alpha,1))\breve\vp_p(1)=
|a|_p\,\breve\vp_p(\alpha^\iota)=|a|_p\,\vp_p(\alpha),
\displaybreak[0]\\[8pt]
\begin{aligned}
\tilde\Psi_p^\eta(g,\beta)&= |a|_p\,
T_p^{[a]}\,\dot\Psi_p^\eta(\beta)\\ &= |a|_p^{1+\ic
t_p}\,\ind_{\Z_p}(a)\, \dot\Psi_p^\eta(\beta)\\
&=W_p^\sharp(g)\,\dot\Psi_p^\eta(\beta),
\end{aligned}
\displaybreak[0]\\[20pt]
\omega(g',\rho(\alpha,1))\breve\vp_p(1)=
|a|_p\,\F\breve\vp_p(\alpha^\iota)=|a|_p\,\F\vp_p(\alpha),
\displaybreak[0]\\[8pt]
\begin{aligned}
\tilde\Psi_p^\eta(g',\beta)&=|ap|_p\,R_p(\vec
N)^{-1}\,T_p^{[ap]}\,\dot\Psi_p^\eta(\beta)\\
&=-|a|_p^{\ic t_p}\,|ap|_p\,\ind_{\Z_p}(ap)\,
\dot\Psi_p^\eta(\beta)\\
&=W_p^\sharp(g')\,\dot\Psi_p^\eta(\beta).
\end{aligned}
\end{gather*}

If $p\mid N^\chi$,
\begin{gather*}
\omega(g,\rho(\alpha,1))\breve\vp_p(1)=
|a|_p\,\breve\vp_p(\alpha^\iota)=
|a|_p\,\vp_p(\iota^{1-\eta_p}\alpha),
\displaybreak[0]\\[8pt]
\begin{aligned}
\tilde\Psi_p^\eta(g,\beta)&= |a|_p\,\tilde\chi_p(a)\,
\ind_{\Z_p}(a)\, \big(T_p^\chi(a,1)\,
\dot\Psi_p\big)^{\eta_p}(\beta)\\
&=\tilde\chi_p(a)\,|a|_p^{\frac12+\ic t_p}\,
\ind_{\Z_p}(a)\,\dot\Psi_p^\eta(\beta)\\
&=W_p^\chi(g)\,\dot\Psi_p^\eta(\beta),
\end{aligned}
\displaybreak[0]\\[20pt]
\omega(g',\rho(\alpha,1))\breve\vp_p(1)=
|a|_p\,\F\breve\vp_p(\alpha^\iota)=
|a|_p\,\F\vp_p(\iota^{1-\eta_p}\alpha),
\displaybreak[0]\\[8pt]
\begin{aligned}
\tilde\Psi_p^\eta(g',\beta)&=|ap^2|_p\,\ov{\gs_p}\,
\tilde\chi_p(-1)\,\ind_{\Z_p}(ap)\,
\big(T_p^\chi(p^{-1},ap)\,\dot\Psi_p\big)^{\eta_p}(\beta),\\
&=\ov{\gs_p}\,\tilde\chi_p(-1)\, |ap^2|_p^{\frac12-\ic
t_p}\,\ind_{\Z_p}(ap)\, \dot\Psi_p^\eta(\beta)\\
&=W_p^\chi(g')\,\dot\Psi_p^\eta(\beta).
\end{aligned}
\end{gather*}

If $p\mid d_B$,
\begin{gather*}
\omega(g,\rho(\alpha,1))\breve\vp_p(1)=
|a|_p\,\breve\vp_p(\alpha^\iota)=|a|_p\,\vp_p(\alpha),
\displaybreak[0]\\[8pt]
\begin{aligned}
\tilde\Psi_p^\eta(g,\beta)&=|a|_p
\,T_p^{[a]}\,\dot\Psi_p^\eta(\beta)\\
&=|a|_p^{1+\ic t_p}\,\ind_{\Z_p}(a)\,
\dot\Psi_p^\eta(\beta)\\
&=W_p^\sharp(g)\,\dot\Psi_p^\eta(\beta),
\end{aligned}
\displaybreak[0]\\[20pt]
\omega(g',\rho(\alpha,1))\breve\vp_p(1)=
-|a|_p\,\F\breve\vp_p(\alpha^\iota)=
-|a|_p\,\F\vp_p(\alpha),
\displaybreak[0]\\[8pt]
\begin{aligned}
\tilde\Psi_p^\eta(g',\beta)&=-|ap|_p
\,T_p(\varpi)^{-1}\,T_p^{[ap]}\,\dot\Psi_p^\eta(\beta)\\
&=-|a|_p^{\ic t_p}\,|ap|_p\,
\ind_{\Z_p}(ap)\,\dot\Psi_p^\eta(\beta)\\
&= W_p^\sharp(g')\,\dot\Psi_p^\eta(\beta).
\end{aligned}
\end{gather*}

If $v=\infty$, we consider $g^\pm= \left[\begin{mat}{cc}\pm
a&\\&1\end{mat}\right]$, $a\in\R^+$\,:
\begin{gather*}
\omega(g^\pm,\rho(\alpha,1))\breve\vp_\infty(1)=
a\,\breve\vp_\infty(\alpha^\iota),
\displaybreak[0]\\[8pt]
\begin{aligned}
\tilde\Psi_\infty^\eta(g^+,\beta)&=\int_{B_\infty^{(a)}}
a^{1+\frac{|k|}2}\,\phi_{X_\infty}^{-\breve k,a}
(\alpha^\iota)\,\Psi_\infty^\eta(\beta\alpha)\,d_\infty^{(1)}\alpha\\
&=\int_{\R^+\lmod B_\infty^+} a^{1+\frac{|k|}2}\,
\phi_{X_\infty}^{\breve k,a}(\alpha)\,
\Psi_\infty^\eta(\beta\alpha)\,\tfrac12\dit\alpha\\
&=\tfrac12\,a^{1+\frac{|k|}2}\,T_{X_\infty}^{\breve
k,a}\,\dot\Psi_\infty^\eta(\beta)\\
&=W_\infty^{|k|}(g^+)\,\dot\Psi_\infty^\eta(\beta),\\[8pt]
\tilde\Psi_\infty^\eta(g^-,\beta)&=\tfrac12\,
a^{1+\frac{|k|}2}\,T_\infty^-\,T_{Y_\infty}^{\breve
k,a}\,\Psi_\infty^\eta(\beta)\\
&=W_\infty^{|k|}(g^-)\, \dot\Psi_\infty^\eta(\beta).
\end{aligned}
\end{gather*}

\subsection{Adjoint of Shimizu's Lift}
\label{adj}

Let $\Psi=\dot\Psi=\ddot\Psi$, $F$ be as in \S\ref{jls}, and for
$h'\in H^{\prime(\delta)}_\A$ define
\begin{align*}
\Theta_\vp\big(\ov{F}\big)(h')=
\int_{G^{(1)}_\Q\lmod G^{(\delta)}_\A} \sum_{\alpha\in B_\Q}
\omega(g,h')\vp(\alpha)\,\ov{F(g)}\,dg.
\end{align*}
\begin{thm}\label{adjshim}
\begin{gather*}
\Theta_\vp\big(\ov{F}\big)=\tfrac{\|F\|^2}{\|\Psi\|^2}\,\ov{F'}
\in L^2_0(H'_\Q\lmod H'_\A,\ov{\tilde\chi}),\\[2pt]
F'(\rho(\dot\beta,\ddot\beta)\iota^\eta)=\Psi^\eta(\dot\beta)\,
\ov{\mbox{$\displaystyle\Psi^\eta$}(\ddot\beta)}.
\end{gather*}
\end{thm}
\begin{proof}
Since the theta kernel is smooth, automorphic, and has moderate
growth, it follows from Lemma \ref{ktype} that
\begin{align*}
\Theta_\vp\big(\ov{F}\big)(\rho(\dot\beta,\ddot\beta)\iota^\eta)
&=\Theta_{\breve\vp}\big(\ov{F}\big)(\rho(\dot\beta,\ddot\beta))\\
&\in\ov{\tilde A_k^\eta(d_B,\vec N,\tilde\chi)}\times
\tilde A_k^\eta(d_B,\vec N,\tilde\chi).
\end{align*}
Furthermore by Lemma \ref{weilhop}, for $p\nmid d_BN$,
\begin{align*}
\Theta_{\breve\vp}\big(T_p^{[p]}\,\ov{F}\big)=
\dot T_p^{[p]}\,\Theta_{\breve\vp}\big(\ov{F}\big)=
\big(\ddot T_p^{[p]}\big)^\vee\,\Theta_{\breve\vp}\big(\ov{F}\big),
\end{align*}
and hence by strong multiplicity-one on $B_\A\ti$,
\begin{align*}
\Theta_{\breve\vp}\big(\ov{F}\big)\circ\rho
\in \ov\pi\times\ov\pi^\vee = \ov\pi\times\pi.
\end{align*} Thus
by Lemma \ref{basis}, $\Theta_{\breve\vp}\big(\ov{F}\big)$ is
determined up to scaling.  We compute its normalization using
the adjoint identity below (compare with the similitude see-saw
identity in \cite{hk2}).\linebreak
\end{proof}

\begin{lem}
\begin{align*}
\Big\langle \breve\Theta'_{\breve\vp}\big(\breve F'\big),
F\Big\rangle_{PG}= \Big\langle \breve F',
\ov{\Theta_{\breve\vp}\big(\ov{F}\big)} \Big\rangle_{P\breve H'}.
\end{align*}
\end{lem}
\begin{proof}
\begin{alignat*}{2}
\text{LHS\,}&{}={}&\tfrac12&\int_{G_\Q\A\ti\lmod G_\A}
\tfrac12\int_{\breve H_\Q^{\prime(1)}\lmod \breve
H_\A^{\prime(\delta)}}\ \vartheta(g,h',\breve\vp)
\breve F'(h')\ov{F(g)}\,\don h'\,\dti g\\
&{}={}&\tfrac12&\int_{G_\Q\R^+\lmod G_\A} \tfrac12\int_{\breve
H_\Q^{\prime(1)}\lmod \breve H_\A^{\prime(\delta)}}\
\vartheta(g,h',\breve\vp)\breve F'(h')\ov{F(g)}\,\don h'\,\dti g\\
&{}={}&\int_{\Q\ti\R^+\lmod\A\ti}&\int_{G_\Q^{(1)}\lmod
G_\A^{(\delta)}} \tfrac12\int_{\breve H_\Q^{\prime(1)}\lmod
\breve H_\A^{\prime(\delta)}}\ \vartheta(g,h',\breve\vp)\breve
F'(h')\ov{F(g)}\,\don h'\,\don g\,\dti\delta,\\[20pt]
&=\,\lefteqn{\text{RHS by symmetry.}}
\end{alignat*}
\end{proof}

\section{Local Data}
\label{locdata}

With the notation of \S\ref{jls}, define
$\vp,\breve\vp\in\Sz(B_\A)$ by
\begin{align*}
\vp_p(\beta)&=\tfrac1{\volx(\ordpx)}\,\ind_{\ordp}(\beta)
\begin{cases}\ 1 & \text{if }p\nmid N^\chi,\\ \tilde\chi_p(\beta_{11})\,
\ind_{\Z_p\ti}(\beta_{11}) & \text{if }p\mid N^\chi,
\end{cases}\\
\vp_\infty(\beta)&=\tfrac1\pi X^{\und k}(\beta)\,\e(iP(\beta))\,,\\[12pt]
\breve\vp_v(\beta)&=\omega(1,\iota_v^{\eta_v})\vp_v(\beta)
=\vp_v(\iota_v^{\eta_v}\beta),\hspace{20pt}\text{i.e.}\\
\breve\vp_p(\beta)&=\tfrac1{\volx(\ordpx)}\,\ind_{\ordp}(\beta)
\begin{cases}\ 1 & \text{if }p\nmid N^\chi,\\
\tilde\chi_p(\beta_{11})\, \ind_{\Z_p\ti}(\beta_{11}) & \text{if }
p\mid N^\chi,\,\eta_p=0,\\ \tilde\chi_p(\beta_{22})\,
\ind_{\Z_p\ti}(\beta_{22}) & \text{if }p\mid N^\chi,\,\eta_p=1,
\end{cases}\\
\breve\vp_\infty(\beta)&=\tfrac1\pi X^{\und{\breve k}}
(\beta)\,\e(iP(\beta)),\quad \breve k=(-1)^{\eta_\infty}k.
\end{align*}
Now we compute the Fourier transforms of these functions in each
case, and relate them to Hecke operators for $B_v\ti$.

If $p\nmid d_BN^\chi$, writing $\beta=\left[\begin{mat}{cc}
a&b\\c&d\end{mat}\right]$, $\alpha_{\vec N}=\left[\begin{mat}{cc}
&\dot N\\\ddot N&\end{mat}\right]$,
\begin{align*}
\breve\vp_p(\beta)&=\tfrac1{\volx(\ordpx)}\,\ind_{\Z_p}(a)\,\ind_{\dot
N\Z_p}(b)\, \ind_{\ddot N\Z_p}(c)\,\ind_{\Z_p}(d),\\
\breve\vp_p\big|_{B_p\ti}&=\textstyle \sum_{n\ge0}T_p^{[p^n]},\\[12pt]
\F\breve\vp_p(\beta)&=
\tfrac1{\volx(\ordpx)}\,\ind_{\Z_p}(d)\,|\dot
N|_p\ind_{\frac1{\dot N}\Z_p}(c) \,|\ddot N|_p\ind_{\frac1{\ddot
N}\Z_p}(b)\,\ind_{\Z_p}(a),\\
&=\tfrac1{\volx(\ordpx)}\,|N|_p\,\ind_{\Z_p}(a)\,
\ind_{\Z_p}(\ddot Nb) \,\ind_{\Z_p}(\dot Nc)\,\ind_{\Z_p}(d),\\
&=\begin{cases}
\breve\vp_p(\beta) & \text{if }p\nmid d_BN,\\
\tfrac{1}p\,\breve\vp_p(\alpha_{\vec N}\beta) & \text{if }p\mid
N^\flat N^\sharp,
\end{cases}\\
\F\breve\vp_p\big|_{B_p\ti}&=
\begin{cases}
\breve\vp_p\big|_{B_p\ti} & \text{if }p\nmid d_BN,\\
\tfrac{1}p\,R_p(\vec N)^{-1}\, \breve\vp_p\big|_{B_p\ti}
& \text{if }p\mid N^\flat N^\sharp.
\end{cases}
\end{align*}

If $p\mid N^\chi$, (note $\iota_p$ and $\F$ commute)
\begin{align*}
\vp_p(\beta)&=\tfrac1{\volx(\ordpx)}\,
\tilde\chi_p(a)\ind_{\Z_p\ti}(a)\,\ind_{\dot N\Z_p}(b)\,
\ind_{\ddot N\Z_p}(c)\,\ind_{\Z_p}(d)\,,\\
\vp_p\big|_{B_p\ti}&=\textstyle\sum_{n\ge0}
(\tilde\chi_p\circ\nu)\cdot\phi^{\chi\,2,2}_{T_p(1,p^n)}\,,\\[12pt]
\F\vp_p(\beta)&=\tfrac1{\volx(\ordpx)}\, \tfrac1{p^2}\,
\ov{\gs_p\,\tilde\chi_p(-d)}\,\ind_{\frac1p \Z_p\ti}(d)\,
\ind_{\frac1p\ddot N \Z_p}(c)\, \ind_{\frac1p\dot N \Z_p}(b)
\,\ind_{\Z_p}(a)\,,\\
\F\vp_p\big|_{B_p\ti}&=\tfrac1{p^2}\,\ov{\gs_p}\,
\tilde\chi_p(-1) \textstyle\sum_{n\ge0}
\phi^{\chi\,2,2}_{T_p(p^n,p^{-1})}\,.
\end{align*}

If $p\mid d_B$,\hspace{10pt} $\valr^*_p:=
\{\alpha\in\diva_p\,;\,\forall\beta\in\valr_p\,,\
\tr(\alpha\beta)\in\Z_p\}=\varpi_p^{-1}\valr_p\,,$
\begin{align*}
\breve\vp_{p}&=\tfrac1{\volx(\valr_p)}\, \ind_{\valr_p}\,,\\
\breve\vp_p\big|_{B_p\ti}&=\textstyle\sum_{n\ge0}T_p^{[p^n]}\,,\\[12pt]
-\omega(J)\breve\vp_p=
\F\breve\vp_{p}&=\tfrac1{\volx(\valr_p)}\,
\tfrac1p\,\ind_{(\valr^*_p)^\iota}
=\tfrac1{\volx(\valr_p)}\,
\tfrac1p\,\ind_{\varpi_p^{-1}\valr_p}\,,\\
\F\breve\vp_p\big|_{B_p\ti}&= \tfrac1p\,
T_p(\varpi_p)^{-1}\,\breve\vp_p\big|_{B_p\ti}\,.
\end{align*}

If $v=\infty$,
\begin{align*}
\breve\vp_\infty\big|_{B_\infty^+}&=\nu^{\frac{|k|}2}\cdot
\phi_{X_\infty}^{-\breve k,\nu},\\
T_\infty^-\,\breve\vp_\infty\big|_{B_\infty^-}&=\nu^{\frac{|k|}2}\cdot
\phi_{Y_\infty}^{-\breve k,\nu}.
\intertext{Then}
D\omega\left(\left[\begin{mat}{cc}0&1\\0&0\end{mat}\right]\right)
&=2\pi\ic\,(X\ov{X}-Y\ov{Y})\cdot\,,\\
-D\omega\left(\left[\begin{mat}{cc}0&0\\1&0\end{mat}\right]\right)
&=\tfrac{1}{2\pi\ic}\left(\ddx{X}\ddx{\ov{X}}-
\ddx{Y}\ddx{\ov{Y}}\right),\\[12pt]
D\omega(J)\,\breve\vp_\infty&= \ic|k|\,\breve\vp_\infty\\
\Longrightarrow\hspace{10pt}
\omega(\kappa_\theta)\,\breve\vp_\infty&=
e^{\ic|k|\theta}\breve\vp_\infty.
\end{align*}

\begin{lem}\label{ktype}
For $\kappa\in K_p(1,(1,d_BN))$ and $\dot\kappa,\ddot\kappa\in
K_p(d_B,\vec N)$ s.t.\
$\det\kappa=\nu(\dot\kappa)/\nu(\ddot\kappa)$\,,\linebreak
(Note $\tilde\chi_p$ has different definitions on these two
groups.)
\begin{align*}
\omega(\kappa,\rho(\dot\kappa,\ddot\kappa)) \breve\vp_p &=
\tilde\chi_p(\kappa)\,\ov{\tilde\chi_p(\iota_p^{\eta_p}\dot\kappa)}
\,\tilde\chi_p(\iota_p^{\eta_p}\ddot\kappa)\, \breve\vp_p\,,\\
\omega(\kappa_\theta,\rho(\kappa_{\dot\theta},\kappa_{\ddot\theta}))
\breve\vp_\infty &= e^{\ic(|k|\theta-\breve k\dot\theta+
\breve k\ddot\theta)}\,\breve\vp_\infty\,.
\end{align*}
\end{lem}
\begin{proof}
First consider diagonal $\kappa\in K_p(1,(1,d_BN))$\,:
\begin{align*}
\omega(\kappa,\rho(\dot\kappa,\ddot\kappa)) \breve\vp_p(\beta)&=
\breve\vp_p(\kappa_{11}\dot\kappa^{-1}\beta\ddot\kappa)\\
&=\tilde\chi_p(\kappa)\,\ov{\tilde\chi_p(\iota_p^{\eta_p}\dot\kappa)}
\,\tilde\chi_p(\iota_p^{\eta_p}\ddot\kappa)\, \breve\vp_p\,.
\end{align*}
Then it suffices to check
\begin{alignat*}{3}
\omega\left(\left[\begin{mat}{cc}1&x\\&1\end{mat}\right]\right)&
\breve\vp_p&&=\breve\vp_p&\hspace{20pt}&\text{for \,}x\in\Z_p\,,\\
\omega\left(\left[\begin{mat}{cc}1&x\\&1\end{mat}\right]\right)&
\F\breve\vp_p&&=\F\breve\vp_p&&\text{for \,}x\in d_BN\Z_p\,.
\end{alignat*}
The archimedean statement follows from previous calculations.
\end{proof}

\begin{lem}\label{weilhop}
For $p\nmid d_B N$,
$\alpha=\left[\begin{mat}{cc}1&\\&p\end{mat}\right]$,
\begin{align*}
\int_{\SL_2(\Z_p)}\omega(\alpha\kappa\alpha^{-1})\vp_p(\beta)
\,d_p^{(1)}\kappa &=\int_{K_p^{(1)}(d_B,\vec N)}
\omega(\alpha,\rho(\dot\kappa^{-1}\alpha,1))
\vp_p(\beta)\,d_p^{(1)}\dot\kappa\\
&=\int_{K_p^{(1)}(d_B,\vec N)}
\omega(\alpha,\rho(1,\ddot\kappa\alpha^{-1}))\vp_p(\beta)
\,d_p^{(1)}\ddot\kappa\,.
\end{align*}
\end{lem}
\begin{proof}
Since $\omega(\alpha\kappa\alpha^{-1})\,\vp_p=\vp_p$ for
$\kappa\in K_p^{(1)}(1,(p,\tfrac1p))$, the $\kappa$ integral
reduces to an average over the
$K_p^{(1)}(1,(1,1))\big/K_p^{(1)}(1,(p,1))$ coset
representatives,
\begin{align*}
\left[\begin{mat}{cc}1&x\\&1\end{mat}\right]\hspace{10pt}(x\
{\rm mod}\ p),\hspace{20pt}\left[\begin{mat}{cc}&1\\-1&\end{mat}\right],
\end{align*}
and hence evaluates to
\begin{align*}
\vp_p'(\beta)=\tfrac1{p+1}\big(p\,\ind_{\ord_p^{(p\Z_p)}(d_B,\vec N)}
(\beta)+p^2\,\ind_{p\,\ord_p(d_B,\vec N)}(\beta)\big).
\end{align*}
Now the $\dot\kappa$ and $\ddot\kappa$ integrals evaluate to
\begin{align*}
\int_{K_p^{(1)}(d_B,\vec N)}p\,\vp_p(\alpha^{-1}\dot\kappa\beta)
\,d_p^{(1)}\dot\kappa
&=\tfrac{p}{p+1}\,\big(T_p^{[p]}*\ind_{\ordp}\big)(\beta),\\
\int_{K_p^{(1)}(d_B,\vec N)}p\,\vp_p(\beta\ddot\kappa\alpha^{-1})
\,d_p^{(1)}\ddot\kappa
&=\tfrac{p}{p+1}\,\big(\ind_{\ordp}*T_p^{[p]}\big)(\beta).
\end{align*}
It follows from the calculations in \S\ref{hring} (extending by
continuity to $B_v$) that
\begin{align*}
\ind_{\ordpx}&=\ind_{\ord_p}*\big(T_p(1)-T_p^{[p]}+p\,T_p(p)\big),\\
\ind_{\ord_p}*T_p^{[p]}&= \ind_{\ord_p}-\ind_{\ordpx}
+p\,\ind_{p\ord_p} = T_p^{[p]}*\ind_{\ord_p}\\
&=\ind_{\ord_p^{(p\Z_p)}}+p\,\ind_{p\ord_p}=\tfrac{p+1}p\,\vp'_p\,.
\end{align*}
\end{proof}

\chapter{L-functions}

In this chapter we compute special values of adjoint and triple
product L-functions. First we describe the relevant Langlands
parameters (admissible representations of Weil or Weil-Deligne
groups), which determine canonical $L$ and $\veps$ factors.
Then we compute special values of the global $L$-functions
using the Rankin--Selberg and Garrett/Piatetski-Shapiro--Rallis
integral representations.

\section{Local Langlands Correspondence}

Let $\We_{F}$ denote the Weil group of the local field $F$ (as in
\cite{tate,kudla,shah}) and ${\mathsf r}_F$ the local class-field
theory isomorphism ${\mathsf r}_F:F\ti\stackrel\sim\longrightarrow
\We_{F}^{\rm ab}$. If $\chi$ is a character of $F\ti$, denote by
$\chi^{\mathsf r}$ the corresponding character of $\We_{F}$. In
particular, write $\|\,\|=|\,|_F^{\mathsf r}$.

Recall that $\We_\C=\C\ti$ is abelian and ${\mathsf r}_\C$ is the
identity (so $\|z\|=|z|_\C$), and
\begin{align*}
\We_\R=\C\ti\cup\jmath\,\C\ti,\hspace{30pt}
\jmath^2=-1,\hspace{20pt}\jmath z\jmath^{-1}=\ov{z},
\end{align*}
has closed commutator subgroup $\We_\R^{\rm c}=\C^{(1)}$, with
${\mathsf r}_\R$ being defined by
\begin{align*}
\R^+\ni r\mapsto r^{\frac12}\,\C^{(1)},\hspace{20pt}
-1\mapsto\jmath\,\C^{(1)}\hspace{30pt}\big(\text{so }
\|z\|=\|\jmath z\|=|z|_\C\big).
\end{align*}

As described in \cite{knapp} (with a normalization error that we
have corrected), all one-dimensional representations of $\We_\R$
are isomorphic to some $\varrho^1=(s,\delta)_\R^1$\,,
\begin{align*}
\varrho^1(z)=\|z\|^s,\hspace{20pt}\varrho^1(\jmath)=(-1)^\delta,
\hspace{20pt}\text{for \,}s\in\C,\ \delta\in\{0,1\},
\end{align*}
and all irreducible two-dimensional semi-simple representations
to some $\varrho^2=(s,l)_\R^2$\,,
\begin{align*}
\varrho^2(re^{i\theta})=\left[\begin{mat}{cc}r^{2s}e^{il\theta} &\\
&r^{2s}e^{-il\theta}\end{mat}\right],
\hspace{20pt}\varrho^2(\jmath)=\left[\begin{mat}{cc}&(-1)^l
\\1&\end{mat}\right],\hspace{20pt}\text{for \,}s\in\C,\
l\in\Z^+.
\end{align*}
If we allow $l\in\Z$, then
\begin{gather*}
(s,0)_\R^2\simeq(s,0)^1_\R\oplus(s,1)^1_\R\,,\hspace{30pt}
(s,l)_\R^2\simeq(s,-l)_\R^2\,,\\[6pt]
\begin{aligned}
(s_1,\delta_1)_\R^1\otimes(s_2,\delta_2)_\R^1&\simeq
(s_1+s_2,\delta)_\R^1\,,
\hspace{15pt}\delta\equiv\delta_1+\delta_2\ \,{\rm mod}\,\,2,\\
(s_1,l_1)_\R^2\otimes(s_2,l_2)_\R^2&\simeq(s_1+s_2,l_1+l_2)_\R^2
\oplus(s_1+s_2,l_1-l_2)_\R^2\,,\\
(s_1,\delta_1)_\R^1\otimes(s_2,l_2)_\R^2&\simeq(s_1+s_2,l_2)_\R^2\,.
\end{aligned}
\end{gather*}
The standard local factors are defined as follows:
\begin{alignat*}{2}
\zeta_\R(s&)=\pi^{-\frac{s}{2}}\Gamma(\tfrac{s}{2}),&
\zeta_\C(s&)=\zeta_\R(s)\,\zeta_\R(s+1)\\
&&&\phantom{)}=(2\pi)^{-s}\,\Gamma(s),\\[6pt]
L(s,\varrho^1&)=\zeta_{\R}(s+s_\infty+\delta_\infty),&\hspace{30pt}
L(s,\varrho^2&)=\zeta_{\C}(s+s_\infty+\tfrac{l_\infty}2),\\
C(t,\varrho^1&)=(1+|\ic t+s_\infty|),&
C(t,\varrho^2&)=(1+|\ic t+s_\infty+\tfrac{l_\infty}2|)^2,\\
\veps(s,\varrho^1&{},\e_\infty)=\ic^{\delta_\infty},&
\veps(s,\varrho^2&{},\e_\infty)=\ic^{l_\infty+1}.
\end{alignat*}

Now let $\We'_{\Q_p}$ denote the Weil-Deligne group of $\Q_p$. We
implicitly identify the characters of $\We_{\Q_p}$ and
$\We'_{\Q_p}$. Recall from \cite{tate,kudla} the indecomposable
admissible representation $\spp^n$ of $\We'_{\Q_p}$ on $\C\{{\rm
e}_0,\ldots,{\rm e}_{n-1}\}$ defined by
\begin{align*}
\spp^n(w)\,{\rm e}_i=\|w\|^{i}\,{\rm e}_i,\hspace{30pt}
\spp^n(\Nil)\,{\rm e}_{i}=
\begin{cases}{\rm e}_{i+1} & \text{if }i<n-1,\\\ 0 & \text{if }i=n-1.
\end{cases}
\end{align*}
Note $(\spp^n)^\vee\simeq\|\,\|^{1-n}\otimes\spp^n$.  Since
$\spp^n\simeq {\rm Sym}^{n-1}\spp^2$, it follows that
\begin{align*}
\spp^m\otimes\spp^n&\simeq\bigoplus_{i=0}^{n-1}
\|\,\|^{i}\otimes\spp^{m+n-2i-1}\hspace{10pt}\forall\ m\ge n,\\
\otimes^2\spp^2&\simeq\|\,\|^1\oplus\spp^3,\\
\otimes^3\spp^2&\simeq\oplus^2\big(\|\,\|^1\otimes\spp^2\big)\oplus\spp^4.
\end{align*}
Also consider a character $\tilde\chi_p$ of $\Q_p\ti$ with
conductor $p^k\Z_p$ for $k>0$. The standard local factors we
need are then defined as follows: $\zeta_p(s)=(1-p^{-s})^{-1}$,
\begin{alignat*}{2}
\varrho&=\|\,\|^{s_p}\otimes\spp^n,&
\varrho'&=\tilde\chi_p^{\mathsf r}\,,\\
L(s,\varrho)&=\zeta_p(s+s_p+n-1),&L(s,\varrho')&=1,\\
C(\varrho)&=p^{n-1},&C(\varrho')&=p^k,\\
\veps(s,\varrho,\e_p)&=\prod_{i=0}^{n-2}
\frac{\zeta_p(-s-s_p-i)}{\zeta_p(s+s_p+i)}&\hspace{30pt}
\veps(s,\varrho',\e_p)&=p^{-ks}\int_{p^{-k}\Z_p\ti}
\tilde\chi_p^{-1}(x)\,\e_p(x)\,dx.\\
&=\big(-p^{-s-s_p-\frac{n-2}2}\big)^{n-1},
\end{alignat*}

\subsection{Adjoint Parameters}

Consider a representation $\pi_v$ of $\GL_2(\Q_v)$ as in
\S\ref{autr}\,. Below we list the corresponding Langlands
parameters $\varrho$ and their local factors, plus those of
compositions with the adjoint representation ${\rm
Ad}:\GL_2(\C)\rightarrow\GL_3(\C)$, \hbox{${\rm Ad}\varrho
={\rm Ad}\circ\varrho
\simeq\varrho\otimes\varrho^\vee\ominus1$}.\vspace{-4pt}

\subsubsection{Archimedean}

If $k=0$, then
$\pi_\infty\simeq\pi(\sgn^{\delta_\infty}|\,|_\infty^{s_\infty},
\sgn^{\delta_\infty}|\,|_\infty^{-s_\infty})$,
$\delta_\infty\in\{0,1\}$, corresponds to\vspace{-2pt}
\begin{align*}
\varrho^0&\simeq(s_\infty,\delta_\infty)_\R^1\oplus
(-s_\infty,\delta_\infty)_\R^1\,,\text{ with}\\
L(s,\varrho^0)&=\zeta_\R(s+s_{\infty}+\delta_{\infty})\,
\zeta_\R(s-s_{\infty}+\delta_{\infty}),\\
C(t,\varrho^0)&=(1+|\ic t+s_\infty|)\, (1+|\ic t-s_\infty|),\\
\veps(s,\varrho^0,\e_\infty)&=(-1)^{\delta_\infty},\text{ and}\\[4pt]
{\rm Ad} \varrho^0&\simeq (2s_\infty,0)_\R^1\oplus
(0,0)_\R^1\oplus (-2s_\infty,0)_\R^1\,,\text{ with}\\
L(s,{\rm Ad} \varrho^0)&=\zeta_\R(s+2s_{\infty})\,\zeta_\R(s)\,
\zeta_\R(s-2s_{\infty}),\\ C(t,{\rm Ad} \varrho^0)&=(1+|\ic
t+2s_\infty|)\,(1+|\ic t|)\,(1+|\ic t-2s_\infty|),\\
\veps(s,{\rm Ad} \varrho^0,\e_\infty)&=1.\\[-22pt]
\end{align*}
If $k\ge2$, then
$\pi_\infty\simeq\sigma(\sgn^k|\,|_\infty^{s_\infty},|\,
|_\infty^{-s_\infty})$, $s_\infty=\tfrac{k-1}2$, corresponds
to\vspace{-2pt}
\begin{align*}
\varrho^k&\simeq(0,k-1)_\R^2\,,\text{ with}\\
L(s,\varrho^k)&=\zeta_\C(s+s_{\infty}),\\
C(t,\varrho^k)&=(1+|\ic t+s_\infty|)^2,\\
\veps(s,\varrho^k,\e_\infty)&=\ic^{k},\text{ and}\\[4pt]
{\rm Ad}\varrho^k&\simeq(0,2k-2)_\R^2\oplus(0,1)_\R^1\,,\text{ with}\\
L(s,{\rm Ad} \varrho^k)&= \zeta_\C(s+2s_{\infty})\,\zeta_\R(s+1),\\
C(t,{\rm Ad} \varrho^k)&=(1+|\ic t+2s_\infty|)^2\,(1+|\ic t|),\\
\veps(s,{\rm Ad} \varrho^k,\e_\infty)&=(-1)^k\,.\\[-22pt]
\end{align*}

\subsubsection{Non-Archimedean}

If $p\nmid M^\sharp M^\chi$, then $\pi_p\simeq\pi(|\,|_p^{\dot
s_p}, |\,|_p^{\ddot s_p})$ corresponds to \vspace{-2pt}
\begin{align*}
\varrho^0&\simeq\|\,\|^{\dot s_p} \oplus\|\,\|^{\ddot s_p},\text{ with}\\
L(s,\varrho^0)&=\zeta_p(s+\dot s_p)\, \zeta_p(s+\ddot s_p),\\
\veps(s,\varrho^0,\e_p)&=1,\\ C(\varrho^0)&=1,\text{ and}
\displaybreak[0]\\[4pt]
{\rm Ad}\varrho^0&\simeq \|\,\|^{\dot s_p-\ddot s_p}\oplus\|\,
\|^0 \oplus\|\,\|^{\ddot s_p-\dot s_p},\text{ with}\\
L(s,{\rm Ad} \varrho^0)&=\zeta_p(s+\dot
s_p-\ddot s_p)\,\zeta_p(s) \,\zeta_p(s+\ddot s_p-\dot s_p),\\
\veps(s,{\rm Ad} \varrho^0,\e_p)&=1,\\
C({\rm Ad} \varrho^0)&=1.\\[-22pt]
\end{align*}
If $p\mid M^\sharp$, then
$\pi_p\simeq\sigma(|\,|_p^{-\frac12+\ic t_p},
|\,|_p^{\frac12+\ic t_p})$ corresponds to \vspace{-2pt}
\begin{alignat*}{2}
\varrho^\sharp&\simeq\|\,\|^{-\frac12+\ic t_p}\otimes\spp^2
&\text{and}\hspace{10pt}{\rm Ad}\varrho^\sharp&\simeq\|\,\|^{-1}
\otimes\spp^3,\text{ with}\\
L_p(s,\varrho^\sharp)&=\zeta_p(s+\tfrac12+\ic t_p), & L(s,{\rm
Ad} \varrho^\sharp)&=\zeta_p(s+1),\\
\veps_p(s,\varrho^\sharp,\e_p)&=(-p^{-\ic
t_p})p^{-s+\frac12},&\hspace{30pt}\veps(s,{\rm Ad}
\varrho^\sharp,\e_p)&=p^{-2s+1},\\
C(\varrho^\sharp)&=p\,,&C({\rm Ad}\varrho^\sharp)
&=p^2.\\[-22pt]
\end{alignat*}
If $p\mid M^\chi$, then $\pi_p\simeq\pi(\tilde\chi_p|\,|_p^{\ic
t_p}, |\,|_p^{-\ic t_p})$ corresponds to \vspace{-2pt}
\begin{align*}
\varrho^\chi&\simeq\tilde\chi_p^{\mathsf r}\,\|\,\|^{\ic
t_p}\oplus\|\,\|^{-\ic t_p},\text{ with}\\
L(s,\varrho^\chi)&=\zeta_p(s-\ic t_p),\\
\veps(s,\varrho^\chi,\e_p)&=\Big(p^{-\ic t_p-\frac12}
\displaystyle\int_{\frac1p\Z_p\ti}
\ov{\tilde\chi_p(x)}\,\e_p(x)\,dx\Big)p^{-s+\frac12},\\[-8pt]
C(\varrho^\chi)&=p\,,\text{ and}\\[4pt]
{\rm Ad} \varrho^\chi &\simeq\tilde\chi_p^{\mathsf
r}\,\|\,\|^{2\ic t_p}\oplus\|\,\|^0
\oplus\ov{\tilde\chi_p^{\mathsf r}}\,\|\,\|^{-2\ic t_p},
\text{ with}\\
L(s,{\rm Ad} \varrho^\chi)&=\zeta_p(s),\\
\veps(s,{\rm Ad}\varrho^\chi,\e_p)&=\tilde\chi_p(-1)\,p^{-2s+1},\\
C({\rm Ad} \varrho^\chi)&=p^2.\\[-22pt]
\end{align*}

\subsection{Triple Product Parameters}

\subsubsection{Archimedean}

If $k_1=k_2=k_3=0$, let $\{0,1\}\ni\delta\equiv
\delta_1+\delta_2+\delta_3\ \,{\rm mod}\,\,2$, so that
\vspace{-2pt}
\begin{align*}
\varrho:=\varrho_1^0\otimes\varrho_2^0\otimes\varrho_3^0
&\simeq\bigoplus_{\ \{\pm\}^3}(\pm s_{1} \pm s_{2}
\pm s_{3} ,\delta)_\R^1\,,\text{ with}\\
L(s,\varrho)&=\prod_{\ \{\pm\}^3} \zeta_\R(s\pm s_{1}\pm s_{2}
\pm s_{3}+\delta),\\
C(t,\varrho)&= \displaystyle \prod_{\ \{\pm\}^3}
(1+|\ic t\pm s_{1} \pm s_{2} \pm s_{3}|),\\
\veps(s,\varrho,\e_\infty) &=1.
\end{align*}
If $k=|k_1|=|k_2|\ge2$ and $k_3=0$, \vspace{-2pt}
\begin{align*}
\varrho:=\varrho_1^k\otimes\varrho_2^{k}\otimes\varrho_3^0
&\simeq (s_{3},2k-2)_\R^2\oplus (s_{3},0)_\R^2\\&\hspace{20pt}
\oplus(-s_{3},2k-2)_\R^2\oplus(-s_{3},0)_\R^2\,,\text{ with}\\
L(s,\varrho)&= \zeta_\C(s+s_{3}+k-1)\,\zeta_\C(s+s_{3})\\
&\hspace{20pt}\cdot\zeta_\C(s-s_{3}+k-1)\, \zeta_\C(s-s_{3}),\\
C(t,\varrho)&=(1+|\ic t+s_3+k-1|)^2\,(1+|\ic t+s_3|)^2\\
&\hspace{20pt}\cdot(1+|\ic t-s_3+k-1|)^2\,(1+|\ic t-s_3|)^2,\\
\veps(s,\varrho,\e_\infty) &= 1\,.\\[-22pt]
\end{align*}
If $|k_1|=|k_2|+|k_3|>|k_2|\ge|k_3|\ge 2$, \vspace{-2pt}
\begin{align*}
\varrho:=\varrho_1^{|k_1|}\otimes\varrho_2^{|k_2|}\otimes
\varrho_3^{|k_3|}&\simeq(0,2|k_1|-3)_\R^2\oplus(0,2|k_2|-1)_\R^2\\
&\hspace{20pt}\oplus (0,2|k_3|-1)_\R^2 \oplus (0,1)_\R^2\,,\text{ with}\\
L(s,\varrho)&=\zeta_\C(s+|k_1|-\tfrac32)\,\zeta_\C(s+|k_2|-\tfrac12)\\
&\hspace{20pt}\cdot \zeta_\C(s+|k_3|-\tfrac12)
\,\zeta_\C(s+\tfrac12)\,,\\
C(t,\varrho)&= (1+|\ic t+|k_1|-\tfrac32|)^2\,(1+|\ic t+|k_2|-\tfrac12|)^2\\
&\hspace{20pt}\cdot (1+|\ic t+|k_3|-\tfrac12|)^2\,(1+|\ic
t+\tfrac12|)^2,\\ \veps(s,\varrho,\e_\infty) &= 1\,.
\end{align*}

\subsubsection{Non-Archimedean}
\vspace{-12pt}

\begin{align*}
\varrho^0_1\otimes \varrho^0_2&\simeq
\bigoplus_{\eta\in\{\cdot,\cdot\cdot\}^2}
\|\,\|^{s_1^{\eta_1}+s_2^{\eta_2}},\\
\varrho^0_1\otimes \varrho^0_2\otimes \varrho^0_3&\simeq
\bigoplus_{\eta\in\{\cdot,\cdot\cdot\}^3}
\|\,\|^{s_1^{\eta_1}+s_2^{\eta_2}+s_3^{\eta_3}},\\
L_p(s,\varrho^0_1\otimes \varrho^0_2\otimes \varrho^0_3)&=
\displaystyle \prod_{\eta\in\{\cdot,\cdot\cdot\}^3}
\zeta_p(s+s_1^{\eta_1}+s_2^{\eta_2}+s_3^{\eta_3})\,,\\
\veps_p(s,\varrho^0_1\otimes\varrho^0_2\otimes \varrho^0_3,\e_p)&=1,\\
C(\varrho^0_1\otimes \varrho^0_2\otimes \varrho^0_3)&=1,\\[8pt]
\varrho^0_1\otimes \varrho^0_2\otimes
\varrho^\sharp_3&\simeq\bigoplus_{\eta\in\{\cdot,\cdot\cdot\}^2}
\|\,\|^{s_1^{\eta_1}+s_2^{\eta_2}-\frac12+\ic
t_3}\otimes\spp2,
\end{align*}
\begin{align*}
\varrho^\sharp_1\otimes \varrho^\sharp_2&\simeq \|\,\|^{\ic
t_1+\ic t_2}\oplus\big(\|\,\|^{-1+\ic
t_1+\ic t_2}\otimes\spp^3\big),\\
\varrho^\sharp_1\otimes \varrho^\sharp_2\otimes
\varrho^\sharp_3&\simeq\oplus^2\big(\|\,\|^{-\frac12+\ic
(t_1+t_2+t_3)}\otimes\spp^2\big)
\oplus\big(\|\,\|^{-\frac32+\ic(t_1+t_2+t_3)}\otimes\spp^4\big),\\
L(s,\varrho^\sharp_1\otimes\varrho^\sharp_2\otimes\varrho^\sharp_3)&=
\zeta_p(s+\tfrac12+\ic(t_1+t_2+t_3))^2\,\zeta_p(s+\tfrac32+\ic
(t_1+t_2+t_3))\,,\\
\veps(s,\varrho^\sharp_1\otimes \varrho^\sharp_2\otimes
\varrho^\sharp_3,\e_p)&=(-p^{-5\ic(t_1+t_2+t_3)})p^{-5s+\frac52},\\
C(\varrho^\sharp_1\otimes \varrho^\sharp_2\otimes
\varrho^\sharp_3)&=p^5,\\[8pt]
\varrho^\sharp_1\otimes \varrho^\sharp_2\otimes\varrho^0_3
&\simeq \|\,\|^{\ic t_1+\ic t_2+\dot s_3}\oplus\big(\|\,\|
^{-1+\ic t_1+\ic t_2+\dot s_3}\otimes\spp^3\big)\\
&\hspace{10pt}\oplus\|\,\|^{\ic t_1+\ic t_2+\ddot s_3}
\oplus\big(\|\,\|^{-1+\ic t_1+\ic t_2+\ddot
s_3}\otimes\spp^3\big).
\end{align*}

\section{Zeta Integrals}

\subsection{Rankin--Selberg}
\label{adjsqz}

\begin{lem}\label{ransel} Let $F\in\tilde
S_k(1,(1,M),\tilde\chi,M^\flat)$ be a Hecke normalized
$\Hal^\star$ eigenform, with Langlands parameters
$\varrho=(\varrho_v)$. Then
\begin{gather*}
\langle F,F\rangle=\int_{\A\ti}\left|\widehat{F}\left(\left[
\begin{mat}{cc}a&\\&1\end{mat}\right]\right)\right|^2 |a|_\A^{-1}\,da
=\frac{\prod_vc_v}{\zeta^*(2)}\,L^*(1,\Ad\varrho),\\[8pt]
\begin{aligned}
c_\infty&=
\begin{cases}
1 & \text{if }k=0,\\
2^{-|k|-1} & \text{if }k\neq0,
\end{cases}&\hspace{20pt}
c_p&=
\begin{cases}
1 & \text{if }p\nmid M,\\
\tfrac p{p+1} & \text{if }p\mid M^\sharp M^\chi,\\
2\big(1+\tfrac{\ov{\veps_p}\lambda_p}{p+1}\big) & \text{if }
p\mid M^\flat.
\end{cases}
\end{aligned}
\end{gather*}
In terms of the corresponding classical Hecke form $f\in
S_k(1,(1,M),\chi,M^\flat)$,
\begin{align*}
\int_{\Gamma_0(M)\lmod\h}|f(z)|^2\,\dvh{x}{y}=2c_\infty M
\Big(\sideset{}{_{p\mid M^\flat}}\prod 2\big(1+\tfrac{1+\ov{\veps_p}\lambda_p}p\big)\Big)
\,L^*(1,\Ad\varrho).
\end{align*}
\end{lem}
\begin{proof}
Suppose $M^\flat=1$ for now, and define the $PG_\A$ Eisenstein
series
\begin{gather*}
\begin{aligned}
E(s,g,\xi)&=\sum_{\gamma\in PT_\Q\lmod PG_\Q}\xi(s,\gamma g),\\
\xi_v\left(s,\left[\begin{mat}{cc}a&x\\&1\end{mat}\right]\kappa\right)
&=|a|_v^s\,\tfrac1{{\rm vol}_v}\ind_{PK_v(1,(1,M))}(\kappa),
\end{aligned}\\
\kappa\in PK_v(1,(1,1)),\hspace{10pt} {\rm vol}_v={\rm
vol}_v(PK_v(1,(1,M))).
\end{gather*}
Now consider the Rankin--Selberg zeta integral,
\begin{align*}
Z(s,F\times\ov F,\xi)&=\int_{PG_\Q\lmod PG_\A}|F(g)|^2\,
E(s,g,\xi)\,dg\\
&=\int_{PT_\Q\lmod PG_\A} |F(g)|^2\,\xi(s,g)\,dg\\
&=\int_{\Q\ti\lmod\A\ti}\int_{\Q\lmod\A}\left|F\left(\left[
\begin{mat}{cc}a&x\\&1\end{mat}\right]\right)\right|^2\,
|a|^{s-1}_\A\,\dti a\,dx\\
&=\int_{\Q\ti\lmod\A\ti}\sum_{\xi\in\Q\ti}
\left|\widehat{F}\left( \left[\begin{mat}{cc}\xi
a&\\&1\end{mat}\right]\right)\right|^2 |a|_\A^{s-1}\,\dti a\\
&=\sideset{}{_v}\prod Z_v(s,\widehat{F}_v\times\widehat{\ov F}_v),\\
Z_v(s,\widehat{F}_v\times\widehat{\ov F}_v) &= \int_{\Q_v\ti}
\left|\widehat{F}_v\left(\left[\begin{mat}{cc}a&\\&1\end{mat}
\right]\right)\right|^2|a|_v^{s-1}\,\dvt a.
\end{align*}
These local factors are easy to compute: If $k=0$,
\begin{align*}
Z_\infty(s,W_p^0\times\check W_p^0) &=\zeta_\R(s+2s_\infty)\,
\zeta_\R(s)^2\,\zeta_\R(s-2s_\infty)\,\zeta_\R(2s)^{-1}\\
&=L_\infty(s,\varrho_\infty\otimes\ov\varrho_\infty)\,\zeta_\R(2s)^{-1}.
\intertext{If $|k|\ge2$,}
Z_\infty(s,W_p^k\times\check W_p^k)&=(4\pi)^{-(s+|k|-1)}
\,\Gamma(s+|k|-1)\\
&=2^{-|k|-1}\,\zeta_\C(s+|k|-1)\,\zeta_\C(s)\,\zeta_\R(2s)^{-1}\\
&=2^{-|k|-1}\,L_\infty(s,\varrho_\infty\otimes\ov\varrho_\infty)
\,\zeta_\R(2s)^{-1}.
\intertext{If $p\nmid M$,}
Z_p(s,W_p^0\times\check W_p^0)&=\zeta_p(s+\dot s_p-\ddot s_p)\,
\zeta_p(s)^2\,\zeta_p(s+\ddot s_p-\dot s_p)\,\zeta_p(2s)^{-1}\\
&=L_p(s,\varrho_p\otimes\ov\varrho_p)\,\zeta_p(2s)^{-1}.
\intertext{If $p\mid M^\sharp$,}
Z_p(s,W_p^\sharp\times\check W_p^\sharp)&=\zeta_p(s+1)
=L_p(s,\varrho_p\otimes\ov\varrho_p)\,\zeta_p(s)^{-1}.
\intertext{If $p\mid M^\chi$,}
Z_p(s,W_p^\chi\times\check W_p^\chi)&=\zeta_p(s)
=L_p(s,\varrho_p\otimes\ov\varrho_p)\,\zeta_p(s)^{-1}.
\end{align*}

By unfolding the integral of an incomplete Eisenstein series as
above, one can show that $\underset{s=1}{\rm res}\,
E(s,g,\xi)=\frac12$. Therefore the Petersson norm of $F$ is given
by
\begin{align*}
\langle F,F\rangle&=\tfrac12 \int_{PG_\Q\lmod PG_\A}|F(g)|^2\,dg\\
&=\underset{s=1}{\rm res}\, Z(s,F\times\ov F,\xi)\\
&=\lim_{s\to1} Z(s,F\times\ov F,\xi)\,\zeta^*(s)^{-1}\\
&=\sideset{}{_v}\prod Z_v(1,\widehat{F}_v\times\widehat{\ov F}_v)
\,\zeta_v(1)^{-1}\\
&=\int_{\A\ti} \left|\widehat{F}\left(\left[\begin{mat}{cc}a&\\
&1\end{mat}\right]\right)\right|^2 |a|_\A^{-1}\,da.
\end{align*}
It is easy to check the value of each $c_v$, defined by
\begin{align*}
Z_v(1,\widehat{F}_v\times\widehat{\ov F}_v)=\frac{c_v}{\zeta_v(2)}
\,L_v(1,\varrho\otimes\ov\varrho).
\end{align*}
The restriction $M^\flat=1$ is removed using the calculation
from \S\ref{autr},
\begin{align*}
\langle V_p^\flat,V_p^\flat\rangle=2\big(1+\tfrac{\ov{\veps_p}
\lambda_p}{p+1}\big)\,\langle\dot V_p^0,\dot V_p^0\rangle.
\end{align*}
\end{proof}

Now we consider Eisenstein series which are Hecke eigenforms. For
any choice of signs $\eps\in\prod_{p\mid M}\{\pm1\}$, define
$E(s,g,\xi^\eps)$ as before, with $R_p=R_p(1,(1,M))$ and
\begin{align*}
\xi^\eps(s,g)= \big(\textstyle\prod_{p\mid M}(1+\eps_p
R_p)\big)\,\xi(s,g).
\end{align*}
Then by Hejhal's calculation of the scattering matrix
\cite{hej},
\begin{gather*}
\eta(s,\xi^\eps)\,E(s,g,\xi^\eps)=\eta(1-s,\xi^\eps)
\,E(1-s,g,\xi^\eps),\\
\eta(s,\xi^\eps):=\big(\textstyle\prod_{p\mid M}
(1+\eps_pp^{s})\big)\,\zeta^*(2s).
\end{gather*}
Finally, we make the spectral renormalization \cite{hej,iwan},
\begin{gather*}
E^1(s,g,\xi^\eps)=\frac{2^{-\frac12\#\{p\mid M\}+\frac12}}
{\vol(\Gamma_0(M)\lmod\h)}\, E(s,g,\xi^\eps),\text{ so that}\\
\int_{\Gamma_0(M)\lmod\h}\left|\int_0^\infty h(t)\,
E^1(\tfrac12+\ic t,\sigma_z,\xi^\eps) \,\tfrac{dt}{\pi}
\right|^2\dvh{x}{y}=\int_0^\infty |h(t)|^2\,\tfrac{dt}{\pi}.
\end{gather*}
\begin{lem}\label{eismth}
Let $\psi\in S_k(1,(1,M),\chi,1)$ be a classical Hecke
normalized $\Hal^\star$ eigenform, with Langlands parameters
$\varrho=(\varrho_v)$. Then
\begin{align*}
\frac{\displaystyle\int_{\Gamma_0(M)\lmod\h}|\psi(z)|^2\,
E^1(s,\sigma_z,\xi^\eps) \,\dvh{x}{y}}
{\displaystyle\int_{\Gamma_0(M)\lmod\h}|\psi(z)|^2\,\dvh{x}{y}} =
{\textstyle\big(\prod_{p\mid M}\frac{1+\eps_p}2\big)}
\frac{2^{\frac12\#\{p\mid M \}-\frac32}}{M^{1-s}}\,
\frac{{L{}}^*(s,\varrho\otimes\ov\varrho)}
{{L{}}^*(1,\Ad\varrho)\,\eta(s,{\xi{}}^\veps)}.
\end{align*}
\end{lem}
\begin{proof}
Let $F$,$f$ be as in Lemma \ref{ransel}, with $M^\flat=1$.
Since $R_p\,|\widehat{F}|^2=|\widehat{F}|^2$ for all $p\mid M$,
\begin{align*}
Z(s,F\times\ov F,\xi^\eps)&=\big(\textstyle\prod_{p\mid M}
\frac{1+\eps_p}2\big)\,2^{\#\{p|M\}}\,Z(s,F\times\ov F,\xi)\\
&=\big(\textstyle\prod_{p\mid M}\frac{1+\eps_p}2\big)\,
2^{\#\{p|M\}}\, \displaystyle\frac{c_\infty\,
{L{}}^*(s,\varrho\otimes\ov\varrho)}{{M{}}^{-s}\,
\eta(s,{\xi{}}^\eps)},\hspace{10pt}\text{and}\\
Z(s,F\times\ov F,\xi^\eps)&=2^{\frac12\#\{p\mid M\}+\frac12}\,
\int_{\Gamma_0(M)\lmod\h}|f(z)|^2\, E^1(s,\sigma_z,\xi^\eps)\,
\dvh{x}{y}.
\end{align*}
The result now follows by dividing out the similar formula of
Lemma \ref{ransel}. Since the new formula is self-normalizing,
we only require $\psi\in\C\ti\!f$ to be an eigenform.
\end{proof}

Note that the identity of Lemma \ref{eismth} is consistent with
the functional equations of $E^1(s,g,\xi^\eps)$ and
${L{}}^*(s,\varrho\otimes\ov{\varrho})$. For the purpose of
comparison with Theorem \ref{mth}, we observe that the
Langlands parameters of the unitary Eisenstein series
$E^1(\frac12+\ic t,g,\xi^\eps)$ are given by $\varrho^E_v=
\|\,\|^{\ic t}\oplus\|\,\|^{-\ic t}$, and that
\begin{align*}
\lefteqn{\frac{\displaystyle \left|\int_{\Gamma_0(M)\lmod\h}
|\psi(z)|^2\, E^1(\tfrac12+\ic t,\sigma_z,\xi^\eps)
\,\dvh{x}{y}\right|^2}
{\displaystyle\left(\int_{\Gamma_0(M)\lmod\h}
|\psi(z)|^2\,\dvh{x}{y}\right)^2}}\hspace{100pt}\\
&={\textstyle\big(\prod_{p|M}Q_p^E\big)}\,
\frac{2^{\#\{p\mid M\}-3}}{M^2}\,\frac{{L{}}^*
(\frac12,\varrho\otimes\ov\varrho\otimes\varrho^E)}
{{L{}}^*{(1,\Ad\varrho)}^2\,\underset{s=1}{\rm res}
{L{}}^*(s,\Ad\varrho^E)},\\[12pt]
Q_p^E&=\big(\tfrac{1+\eps_p}2\big)\,\big(1+p^{-\frac12-\ic t}
\big)^{-1}\, \big(1+p^{-\frac12+\ic t}\big)^{-1}.
\end{align*}

\subsection{Garrett/Piatetski-Shapiro--Rallis}

In this section we explicitly calculate zeta integrals for the
Rankin triple L-function, as defined in \cite{psr}.  Let
$F_j\in\pi_j\cap\tilde S_k(1,(1,d_B N^\sharp),1,1)$ be three Hecke
normalized $\Hal_v^\star$ eigenforms as in Theorem \ref{adjshim},
with matching data $\vp_j$ as defined in \S\ref{locdata}. Consider
$\Fb=F_1\times F_2\times F_3$ as an automorphic form on $P\Gb_\A$
and $\vpb=\vp_1\times\vp_2\times\vp_3\in\Sz(B_\A^3)$. We
distinguish totally unramified data as $\vpb^0$.

Following \cite{psr,hk1}, we consider $\Phi_v(s,h)$ satisfying
\begin{align*}
\Phi_v(s,th)=|\det AD^{-1}|_v^{s+1}\,\Phi_v(s,h),\hspace{20pt}\text{for \,}
h\in H_v,\ t=\left[\begin{mat}{cc}A&B\\0&D\end{mat}\right]\in T_v.
\end{align*}
Note the Iwasawa decomposition $H_v=T_v K_v$, where
\begin{align*}
K_\infty= {\rm U}(3),\hspace{40pt}K_p=\GSp_6(\Z_p).
\end{align*}
Now for $h\in H_v$, choose any $h'\in H'_v$ s.t.\ $\nu(h')=\nu(h)$
and define
\begin{align*}
\Phi_v(0,h)=\omega(h,h')\vpb_v(0).
\end{align*}
Since $\Phi_v^0(0,h)$ is constant on $K_v$,
\begin{align*}
\Phi_v(s,h)=\left(\tfrac{\Phi_v^0(0,h)}{\Phi_v^0(0,1)}\right)^{s}
\,\Phi_v(0,h).
\end{align*}

The local zeta integral is defined as in \cite{psr,gh,gk} by
\begin{gather*}
Z_v(s,\widehat{\ov\Fb}_v,\vpb_v)= \int_{\Q_v\ti
N_v^0\lmod\Gb_v}\Phi_v(s,\gamma_0 g)
\,\widehat{\ov\Fb}_v(g)\,dg,\displaybreak[0]\\
\lefteqn{\hspace{-50pt}\text{where}}
N_v^0=\left\{\left[\begin{mat}{cc}\text{\Large $1$}_3&
\begin{array}{ccc}x_1&&\\&x_2&\\&&x_3\end{array}\\
\text{\Large $0$}_3&\text{\Large $1$}_3\end{mat}\right];\
x_1+x_2+x_3=0\right\},\displaybreak[0]\\
H_\Q^{(1)}\ni\gamma_0=\left[\begin{mat}{rrrrrr}
 1 & 1 & 1 & -1 & 0 & 0 \\
 0 & 1 & 0 & -1 & 1 & 0 \\
 0 & 0 & 1 & -1 & 0 & 1 \\
 1 & 1 & 1 & 0 & 0 & 0 \\
 0 & 0 & 0 & -1 & 1 & 0 \\
 0 & 0 & 0 & -1 & 0 & 1 \\
\end{mat}\right].
\end{gather*}

\begin{lem}
For any $\vpb_v\in\Sz(B_v^3)$,
\begin{align*}
\omega(\gamma_0)\vpb_v(0,0,0) = (-1)^{\ind_{v\mid
d_B}}\int_{B_v}\vpb_v(\beta,\beta,\beta)\,d\beta.
\end{align*}
\end{lem}
\begin{proof}
Write $\gamma_0=J_1^{-1}EJ_1FG$,
\begin{alignat*}{2}
&E=\left[\begin{mat}{cc}
  \text{\Large $1$}_3 & \begin{array}{rrr}
    \hspace{-9pt}-1 &  &  \\
     & 0 &  \\
     &   & 0 \\
  \end{array} \\
  \text{\Large $0$}_3 & \hspace{4pt}\text{\Large $1$}_3
\end{mat}\right],&\hspace{30pt}
&\hspace{4pt}J_1=\left[\begin{mat}{rccc}
  0 &  & 1 &  \\
    & \text{\Large $1$}_2 &  & \text{\Large $0$}_2 \\
 -1 &  & 0 &  \\
    & \text{\Large $0$}_2 &  & \text{\Large $1$}_2
\end{mat}\right],\\[12pt]
&F=\left[\begin{mat}{cc}
 \text{\Large $1$}_3 &
 \begin{array}{rrr}
    \hspace{-9pt}-1 &  &  \\
     & 1 &  \\
     &  & 1 \\
  \end{array} \\
  \text{\Large $0$}_3 & \hspace{4pt}\text{\Large $1$}_3
\end{mat}\right],&
&G=\left[\begin{mat}{llllll} 1&1&1&&&\\ &1&0&&&\\ &&1&&&\\
&&&1&&\\ &&&\hspace{-9.2pt}-1&1&\\ &&&\hspace{-9.2pt}-1&0&1\\
\end{mat}\right].
\end{alignat*}
Then
\begin{align*}
\vpb'_v(\beta_1,\beta_2,\beta_3)&=
\omega(FG)\vpb_v(\beta_1,\beta_2,\beta_3)\\
&=\e_v(-\nu(\beta_1)+\nu(\beta_2)+\nu(\beta_3))\\
&\hspace{20pt}\cdot\vpb_v(\beta_1,\beta_1+\beta_2,
\beta_1+\beta_3),\\[12pt]
\omega(\gamma_0)\vpb_v(0,0,0)&=\omega(J_1^{-1}EJ_1)\vpb'_v(0,0,0)\\
&=(\F_1^{-1}\,\e_v(-\nu(\alpha_1))\cdot\F_1\vpb'_v)(0,0,0)\\
&=(-1)^{\ind_{v\mid d_B}}\int_{B_v}\e_v(\nu(\beta))\,
\vpb'_v(\beta,0,0)\,d\beta\\
&=(-1)^{\ind_{v\mid d_B}} \displaystyle\int_{B_v}
\vpb_v(\beta,\beta,\beta)\,d\beta.
\end{align*}
\end{proof}

By the calculations of Piatetski-Shapiro--Rallis, Ikeda,
Gross--Kudla, plus our own which follow, we have
\begin{thm}\label{zeta}
Suppose $k=|k_1|\ge|k_2|\ge|k_3|\ge 0$,  $k_1+k_2+k_3=0$
($\Rightarrow\ k_2k_3\ge0$), and for some fixed square-free
$N$, all $N_j^\sharp=N$ and all $N_j^\flat=N_j^\chi=1$. Then
\begin{gather*}
Z_v(0,\check\Wb\!_v^{\,\star},\vpb_v)=\frac{C_v}{\zeta_v(2)^2}\,
L_v(\tfrac12,\pi_1\times\pi_2\times\pi_3),\\[4pt]
C_\infty=\begin{cases}
\tfrac{\veps_\infty+1}2 & \text{if }k=0,\\
2^{-2k-2} & \text{if }\ge2.
\end{cases}\hspace{40pt}
C_p=\begin{cases}
\phantom{(}1 & \text{if }p\nmid d_B N,\\
(\veps_p-1)\frac{p\,(p+\veps_p)}{(p+1)^3} & \text{if }p\mid d_B,\\
(\veps_p+1)\frac{p\,(p+\veps_p)}{(p+1)^3} & \text{if }p\mid N.
\end{cases}
\end{gather*}
\end{thm}
\begin{proof}
We go through each case separately.  Note that although we only
describe central values above, they are defined by analytic
continuation of $Z_v(s)$.

\subsubsection{Archimedean}

Recall that\vspace{-4pt}
\begin{align*}
\kappa_\theta|F_j=e^{\ic|k_j|\theta}F_j,\hspace{30pt}
\omega(\kappa_{\theta})\vp_j=e^{\ic|k_j|\theta}\vp_j.\\[-24pt]
\end{align*}
Defining\vspace{-4pt}
\begin{gather*}
g_j=\left[\begin{mat}{rr} \eps a_j&\frac{x}3a_j^{-1}
\\&a_j^{-1}\end{mat}\right],\hspace{30pt}
h'=\rho\left(\left[\begin{mat}{cc}\eps&\\
&1\end{mat}\right],1 \right),\\[4pt]
\text{for \,}\eps\in\{\pm1\},\hspace{20pt}
a_j\in\R\ti,\hspace{20pt}
x\in\R,\\[-12pt]
\intertext{we have}\\[-28pt]
\omega(g_j,h')\vp_j(\beta)=\tfrac1{\pi^3}\,|a_j|^{|k_j|+2}\,
X_\eps^{\und k_j}(\beta)\, \e(\tfrac{x}3\nu(\beta)+\ic
a_j^2P(\beta)),\\[-20pt]
\end{gather*}
where $X_1=X,\ X_{-1}=Y$. Then
\begin{align*}
\Phi(0,\delta g)&=\int_{\M_2(\R)}\omega(g)\vpb(\beta,\beta,\beta)
\,d\beta\\
&=\tfrac1{\pi^3}\,|a_1|^{|k_1|+2}|a_2|^{|k_2|+2}|a_3|^{|k_3|+2}
\int_{\M_2(\R)} |X_\eps|^{2k}\,\e\left(z\,|X|^2 -\ov z\,|Y|^2\right)
\,d\beta,\\&\hspace{20pt}\text{where }
z=x+\ic(a_1^2+a_2^2+a_3^2)= x+\ic y,\displaybreak[0]\\
&=\tfrac1{\pi^3}\,|a_1|^{|k_1|+2}|a_2|^{|k_2|+2}|a_3|^{|k_3|+2}
\Big[\F\,|X_\eps|^{2k}\,\e\big(z\,|X|^2 -\ov
z\,|Y|^2\big)\Big](0)\\
&=\tfrac1{\pi^3}\, (2\pi \ic)^{-2k}\,|a_1|^{|k_1|+2}
|a_2|^{|k_2|+2}|a_3|^{|k_3|+2}\\&\hspace{10pt}\cdot\,
\Big[\big(\ddx{X_\eps}\ddx{\ov{X_\eps}}\big)^{k}\,\F\,
\e\big(z\,|X|^2-\ov z\,|Y|^2\big)\Big](0)\\
&=\tfrac1{\pi^3}\,(2\pi \ic )^{-2k}\,
|a_1|^{|k_1|+2}|a_2|^{|k_2|+2}|a_3|^{|k_3|+2}\, |z|^{-2}\\
&\hspace{10pt}\cdot\,
\Big[\big(\ddx{X_\eps}\ddx{\ov{X_\eps}}\big)^{k}\,
\e\big(-\frac{1}{z}\,|X|^2 +\frac1{\ov{z}}\,|Y|^2\big)\Big](0)\\
&=\tfrac1{\pi^3}\, k!\,|a_1|^{|k_1|+2}|a_2|^{|k_2|+2}
|a_3|^{|k_3|+2}\,(2\pi\,\eps z^{\und\eps}/\ic)^{-k}\,|z|^{-2},
\displaybreak[0]\\[12pt]
\Phi(s,\delta g) &= \tfrac1{\pi^3}\, k!\,
|a_1|^{2s+|k_1|+2}|a_2|^{2s+|k_2|+2} |a_3|^{2s+|k_3|+2}\,
(2\pi\,\eps z^{\und\eps}/\ic)^{-k}\,|z|^{-2s-2}.
\end{align*}
We used the formulas
\begin{align*}
\int_{\R^n}\exp(-\pi x^tQx+2\pi iy^tBx)\,dx&=(\det Q)^{-\frac12}
\exp(-\pi y^tBQ^{-1}B^ty),\\
\left(\ddx{z}\ddx{\ov{z}}\right)^{n}\exp(a|z|^2)\big|_{z=0}
&=n!\,a^n.
\end{align*}

In the case of all $F_j$ being even Maass forms, the difficult
archimedean zeta integral was evaluated by Ikeda \cite{ik}.  It
is trivial to extend his result to arbitrary combinations of
parity. Recall that\vspace{-4pt}
\begin{align*}
\{0,1\}\ni\delta&\equiv\delta_1+\delta_2+\delta_3\mod2,\\
y&=y_1+y_2+y_3.\\[-28pt]
\end{align*}
Then\vspace{-2pt}
\begin{align*}
\lefteqn{Z_\infty(s,\check\Wb\!_\infty^{\,0},\vpb_\infty)}
\hspace{40pt}\\[6pt]
&=\sum_{\{\pm1\}}\int_{(\R\ti)^3}^{\prod>0}\int_\R\int_{[0,\pi)^3}
\tfrac1{\pi^3}\, |a_1a_2a_3|^{2s}\,|z|^{-2-2s}\\&\hspace{40pt}\cdot
\big({\textstyle\prod_j\eps^{\delta_j}\,w_{0,s_j}(4\pi a_j^2)}\big)\,
\e(x)\,2\,\dti a_j\,dx\,d\theta_j\\
&=(1-\delta)\,2^4\int_{(\R^+)^3}\int_\R|z|^{-2-2s}\,\e(x)\,
\textstyle\prod_j a_j^{2s}\, w_{0,s_j}(4\pi a_j^2)\,\dti a_j\,dx\\
&=(1-\delta)\,2^4\int_{(\R^+)^3}\int_\R |z|^{-2-2s}\,\e(x)\,
\textstyle\prod_j y_j^{s+\frac12}\,K_{s_j}(2\pi y_j)\,\dti y_j\,dx\\
&=\frac{(1-\delta)\,2^5\,\pi^{s+1}} {\Gamma(s+1)}
\int_{(\R^+)^3}y^{-s-\frac12} \,K_{s+\frac12}(2\pi y)\,
\textstyle\prod_j y_j^{s+\frac12}\,K_{s_j}(2\pi y_j)\,\dti y_j\\
&=\frac{(1-\delta)\,\pi^{-s}} {\Gamma(s+1)\,\Gamma(2s+1)}
\prod_{\ \{\pm\}^3}\Gamma(\tfrac s2+\tfrac14\pm\tfrac{s_1}2
\pm\tfrac{s_2}2 \pm\tfrac{s_3}2)\\
&=\frac{(1-\delta)}{\zeta_\R(2s+2)\,\zeta_\R(4s+2)}
\prod_{\ \{\pm\}^3}\zeta_\R(s+\hf\pm s_1\pm s_2\pm s_3)\\
&=(1-\delta)\,\frac{L_\infty(s+\tfrac12,\pi_1\times\pi_2\times\pi_3)}
{\zeta_\R(2s+2)\,\zeta_\R(4s+2)}.
\end{align*}

Our calculation in the case $|k_3|\ge2$ is very similar to that of
Gross--Kudla for all-equal positive weights \cite{gk}.
\begin{align*}
\lefteqn{Z_\infty(s,\check W_\infty^{|k_1|}\times\check
W_\infty^{|k_2|}\times \check W_\infty^{|k_3|},\vpb_\infty)}
\hspace{40pt}\\[6pt]
&=\sum_{\{\pm1\}}4\int_{(\R^+)^3}\ind_{\R^-}(\eps)\,a_1^{2s+2|k_1|}
a_2^{2s+2|k_2|}a_3^{2s+2|k_3|}\\&\hspace{20pt}\cdot\int_\R k!\,
(2\pi\,\eps z^{\und\eps}/\ic)^{-k} |z|^{-2s-2}\,\e(z)\,2\,
\dti a_j\,dx\displaybreak[0]\\
&=(2\pi\ic)^{-k}\,k! \int_{(\R^+)^3}y_1^{s+|k_1|}y_2^{s+|k_2|}
y_3^{s+|k_3|}\\&\hspace{20pt}\cdot\int_\R(\ov z)^{-k}|z|^{-2s-2}\,
\e(z)\,\dti y_j\,dx\displaybreak[0]\\
&=(2\pi\ic)^{-k}\,k!\,\frac{(-2\pi\ic)^{s+1}\,
(2\pi\ic)^{s+k+1}}{\Gamma(s+1)\,\Gamma(s+k+1)}\int_{(\R^+)^3}
y_1^{s+|k_1|}y_2^{s+|k_2|}y_3^{s+|k_3|}\\
&\hspace{20pt}\cdot\int_0^\infty e^{-4\pi(y_1+y_2+y_3)(1+t)}\,
t^s(1+t)^{s+k}\,\dti y_j\,dt\displaybreak[0]\\
&=\frac{(2\pi)^{2s+2}\,k!}{\Gamma(s+1)\,\Gamma(s+k+1)}\,
\frac{\Gamma(s+|k_1|)\,\Gamma(s+|k_2|)\,\Gamma(s+|k_3|)}
{(4\pi)^{3s+2k}}\\&\hspace{20pt}\cdot
\int_0^\infty t^s(1+t)^{-2s-k}\,dt\displaybreak[0]\\
&=\frac{k!\,\Gamma(s+|k_1|)\, \Gamma(s+|k_2|)\,\Gamma(s+|k_3|)}
{2^{4s+4k-2}\,\pi^{s+2k-2}\,\Gamma(s+1)\,\Gamma(s+k+1)}\,
\frac{\Gamma(s+1)\,\Gamma(s+k-1)}{\Gamma(2s+k)}\displaybreak[0]\\
&=\frac{2^{-2k-2}\, k!}{(s+k)\,(2s+1)_k}\,\frac{2^4\,(2\pi)^{-4s-2k}}
{\pi^{-3s-2}}\\&\hspace{20pt}\cdot\,
\frac{\Gamma(s+k-1)\,\Gamma(s+1)\,\Gamma(s+|k_2|)\,\Gamma(s+|k_3|)}
{\Gamma(s+1)\,\Gamma(2s+1)}\displaybreak[0]\\
&=\frac{2^{-2k-2}\, k!}{(s+k)\,(2s+1)_k}\,
\frac{L_\infty(s+\hf,\pi_1\times\pi_2\times\pi_3)}
{\zeta_\R(2s+2)\zeta_\R(4s+2)},\displaybreak[0]\\[12pt]
\lefteqn{Z_\infty(0,\check W_\infty^{|k_1|}\times \check
W_\infty^{|k_2|}\times \check W_\infty^{|k_3|},\vpb_\infty)=
\frac{2^{-2k-2}}{\zeta_\R(2)^2}\,L_\infty(\hf,
\pi_1\times\pi_2\times\pi_3).}\hspace{40pt}
\end{align*}
It is easy to check that our evaluation of the above integral is
consistent with Ikeda's result for the overlapping parameter
values $k_j=0$, $s_j=-\tfrac12$.

Finally, in the case $|k_2|\ge2$, $|k_3|=0$, we have
\begin{align*}
\lefteqn{Z_\infty(s,\check W_\infty^{k}\times\check
W_\infty^{k}\times\check W_\infty^{0},\vpb_\infty)}
\hspace{40pt}\\[6pt]
&=(2\pi\ic)^{-k}\,k!\int_{(\R^+)^3}y_1^{s+k}y_2^{s+k}y_3^{s}\,
w_{0,s_3}(4\pi y_3)\,\e^{2\pi y_3}\\&\hspace{20pt}\cdot\int_\R
(\ov{z})^{-k}\,|z|^{-2s-2}\,\e(z)\,\dti y_j\,dx\displaybreak[0]\\
&=(2\pi\ic)^{-k}\,k! \,\frac{(-2\pi\ic)^{s+1}\,(2\pi\ic)^{s+k+1}}
{\Gamma(s+1)\,\Gamma(s+k+1)}\int_{(\R^+)^3}y_1^{s+k}y_2^{s+k}y_3^{s}\,
w_{0,s_3}(4\pi y_3)\\&\hspace{20pt}\cdot \int_0^\infty
e^{-4\pi(y_1+y_2+y_3)(1+t)+2\pi y_3}\, t^s(1+t)^{s+k}\,\dti
y_j\,dt\displaybreak[0]\\
&= \frac{(2\pi)^{2s+2}\,k!}{\Gamma(s+1)\,\Gamma(s+k+1)}\,
\frac{\Gamma(s+k)^2}{(4\pi)^{2s+2k}}\int_{\R^+}y_3^{s}\,
w_{0,s_3}(4\pi y_3)\\&\hspace{20pt}\cdot\int_0^\infty
e^{-2\pi y_3(1+2t)}\, t^s(1+t)^{-s-k}\,\dti y_3\,dt\displaybreak[0]\\
&=2^{-2s-4k+2}\,\pi^{-2k+2}\, \frac{k!\,\Gamma(s+k)}
{\Gamma(s+1)\,(s+k)}\,\frac{\Gamma(s+1)}{(4\pi)^s}\\
&\hspace{20pt}\cdot\int_{\R^+}(4\pi y_3)^{s+\frac{k}2-1}\,
w_{-s-\frac{k}2,\frac{1-k}{2}}(4\pi y_3)\, w_{0,s_3}(4\pi y_3)\,
d\ti y_3\displaybreak[0]\\
&=2^{-4s-4k+2}\,\pi^{-s-2k+2}\, \frac{k!\,\Gamma(s+k)}{(s+k)}\\
&\hspace{20pt}\cdot\sum_{\ \{\pm\}^1}\frac{\Gamma(s+\hf\pm s_3)\,
\Gamma(s+k-\hf\pm s_3)\,\Gamma(\mp 2s_3)}{\Gamma(\hf\mp s_3)\,
\Gamma(2s+k+\hf\pm s_3)}\\&\hspace{40pt}\cdot\, {}_3F_2
\left[\begin{mat}{ccc} s+\hf\pm s_3,&s+k-\hf\pm s_3,&\hf\pm
s_3;\vspace{2pt}\\1\pm2s_3,&2s+k+\hf\pm s_3;&1\end{mat}\right].
\end{align*}
${}_3F_2$ denotes a generalized hypergeometric series. We can
simplify this further at $s=0$ using the three-term relation,
\begin{align*}
{}_3F_2\left[\begin{mat}{ccc}a,&b,&c;\\d,&e;&1\end{mat}\right]
&=\sum_{b\leftrightarrow c}\frac{\Gamma(1-a)\,\Gamma(d)\,\Gamma(e)\,
\Gamma(c-b)}{\Gamma(d-b)\,\Gamma(e-b)\,\Gamma(1+b-a)\,\Gamma(c)}\\
&\hspace{20pt}\cdot\,{}_3F_2\left[\begin{mat}{ccc}b,&1+b-d,&1+b-e
;\vspace{2pt}\\1+b-c,&1+b-a;&1\end{mat}\right].
\end{align*}

\begin{align*}
\lefteqn{Z_\infty(0,\check W_\infty^{k}\times\check W_\infty^{k}
\times\check W_\infty^{0},\vpb_\infty)}\hspace{40pt}\\[6pt]
&=\frac{\Gamma(k)^2}{2^{4k-2}\,\pi^{2k-2}}\sum_{\ \{\pm\}^1}
\frac{\Gamma(\hf\pm s_3)\,\Gamma(k-\hf\pm s_3)\,
\Gamma(\mp2s_3)}{\Gamma(\hf\mp s_3)\,\Gamma(k+\hf\pm s_3)\,}\\
&\hspace{80pt}\cdot\, {}_3F_2\left[\begin{mat}{ccc}\hf\pm s_3,&
k-\hf\pm s_3,&\hf\pm s_3;\vspace{2pt}\\1\pm2s_3,&k+\hf\pm s_3;&1
\end{mat}\right]\displaybreak[0]\\
&=\frac{\Gamma(k)}{2^{4k-2}\,\pi^{2k-2}}\sum_{\ \{\pm\}^1}
\Gamma(\hf\mp s_3)\,\Gamma(-k+\tfrac32\mp s_3)\,\Gamma(\hf\pm s_3)\,
\Gamma(k-\hf\pm s_3)\\&\hspace{80pt}\cdot\,\frac{\Gamma(k)\,
\Gamma(\mp2s_3)}{\Gamma(\hf\mp s_3)\,\Gamma(-k+\tfrac32\mp s_3)\,
\Gamma(\hf\mp s_3)\,\Gamma(k+\hf\pm s_3)}\\&\hspace{100pt}\cdot\,
{}_3F_2\left[\begin{mat}{ccc} \hf\pm s_3,&k-\hf\pm s_3,&\hf\pm
s_3;\vspace{2pt}\\1\pm2s_3,& k+\hf\pm s_3;&1\end{mat}\right]
\displaybreak[0]\\
&=\frac{\Gamma(k)}{2^{4k-2}\,\pi^{2k-2}}\,\frac{(-1)^{k-1}\,\pi^2}
{\cos^2(\pi s_3)}\lim_{f\to2-k}\frac{1}{\Gamma(f)}\\&\hspace{20pt}
\sum_{\ \{\pm\}^1}\frac{\Gamma(k)\,\Gamma(1)\,\Gamma(f)\,\Gamma(\mp2s_3)}
{\Gamma(\hf\mp s_3)\,\Gamma(f-\hf\mp s_3)\,\Gamma(k+\hf\pm s_3)\,
\Gamma(\hf\mp s_3)}\\&\hspace{40pt}\cdot\,{}_3F_2\left[\begin{mat}{ccc}
\hf\pm s_3,&\hf\pm s_3,&\frac32\pm s_3-f; \vspace{2pt}\\
1\pm2s_3,&k+\hf\pm s_3;&1\end{mat}\right]\displaybreak[0]\\
&=\frac{\Gamma(k)\,(-1)^{k-1}\,\pi^2}
{2^{4k-2}\,\pi^{2k-2}\,\cos^2(\pi s_3)} \lim_{f\to2-k}
\frac1{\Gamma(f)}\ {}_3F_2\left[\begin{mat}{ccc}
-k+1,&\hf+s_3,& \hf-s_3;\vspace{2pt}\\1,&f;&1\end{mat}\right]
\\&=\frac{\Gamma(k)\,(-1)^{k-1}\,\pi^2}
{2^{4k-2}\,\pi^{2k-2}\,\cos^2(\pi s_3)} \sum_{n=k-1}^\infty
\frac{(-k+1)_n\,(\hf+s_3)_n\,(\hf-s_3)_n}
{n!\,(1)_n\,\Gamma(2-k+n)}\displaybreak[0]\\
&=\frac{\Gamma(k)\,(-1)^{k-1}\,\pi^2}
{2^{4k-2}\,\pi^{2k-2}\,\cos^2(\pi s_3)}\,
\frac{(-1)^{k-1}\,\Gamma(k)\,\Gamma(k-\hf+s_3)\,\Gamma(k-\hf-s_3)}
{\Gamma(k)^2\,\Gamma(\hf+s_3)\, \Gamma(\hf-s_3)}
\\&=2^{-4k+2}\,\pi^{-2k+2}\,\Gamma(k-\hf+s_3)\,\Gamma(k-\hf-s_3)\,
\Gamma(\hf+s_3)\,\Gamma(\hf-s_3)\\
&=\frac{2^{-2k-2}}
{\zeta_\R(2)^2}\,L_\infty(\hf,\pi_1\times\pi_2\times\pi_3).
\end{align*}
This is consistent with the previous case if we set
$s_3=-\tfrac12$\,. It would be interesting to directly factor the
function $Z_\infty(s)$, for either $k=0$ or $s_3=-\tfrac12$, by
comparing with the previous cases.

\subsubsection{Non-Archimedean}

First consider $p\nmid d_BN$. A more illuminating treatment was
given in \cite{psr}, however our calculation is elementary and
generalizes in an obvious way to the ramified cases. Note that the
restriction $p\neq3$ is unnecessary.

Writing $\alpha=\left[\begin{mat}{cc}a&\\ &1\end{mat}\right],
g_1=\left[\begin{mat}{cc}a&x\\ &1\end{mat}\right],
g_j=\left[\begin{mat}{cc}aa_j&\\ &a_j^{-1}\end{mat}\right]$ for
$j=2,3$, and choosing $h'=\rho\left(\alpha^\iota,1\right)$, we
compute
\begin{align*}
\Phi_p(0,\gamma_0
g)&=\int_{M_{2,\Q_p}}\omega(g)\vpb_p(\beta,\beta,\beta)\,d\beta\\
&=|a^3a_2^2a_3^2|_p\int_{M_{2,\Q_p}}\e_p(x\nu(\beta))\,
\vp_{p1}(\alpha\beta)\,\vp_{p2}(\alpha a_2\beta)\,
\vp_{p3}(\alpha a_3\beta)\,d\beta\\
&=\zeta_p(2)^3\,|a^3a_2^2a_3^2|_p\int_{(\Q_p)^4}
\e_p(x(\beta_{11}\beta_{22}-\beta_{12}\beta_{21}))\\
&\hspace{10pt}\cdot\,\ind_{\Z_p}(a\tilde a\beta_{11})\,
\ind_{\dot N\Z_p}(a\tilde a\beta_{12})\, \ind_{\ddot N\Z_p}(\tilde
a\beta_{21})\, \ind_{\Z_p}(\tilde a\beta_{22})\,d\beta\\
&\hspace{40pt}(\ |\tilde a |_p:=\max\{1,|a_2|_p,|a_3|_p\}\ )\\
&=\zeta_p(2)^3\,|a^3a_2^2a_3^2|_p\int_{\Q_p}|a\tilde a|_p^{-1}\,
\ind_{a\tilde a\Z_p}(x\beta_{22})\,\ind_{\Z_p}(\tilde a\beta_{22})
\,d\beta_{22}\\&\hspace{10pt}\cdot \int_{\Q_p}|\dot N|_p\,|a\tilde
a|_p^{-1}\,\ind_{a\tilde a\ddot N \Z_p}(x\beta_{21})
\,\ind_{\ddot N\Z_p}(\tilde a\beta_{21})\,d\beta_{21}\\
&=\zeta_p(2)^3\,|a^3a_2^2a_3^2|_p\,|z|_p^{-2}\\
&\hspace{40pt}(\ |z|_p:=\max\{|a\tilde a^2|_p,|x|_p\}\ ),\\[12pt]
\Phi_p(s,\gamma_0g)&=\zeta_p(2)^3\,|a^3a_2^2a_3^2|^{s+1}_p\,
|z|_p^{-2s-2}.
\end{align*}
It is straightforward to check that for $|a|_p\le1$,
\begin{align*}
\int_{\Q_p} |z|_p^{\lambda-1}\,\e_p(x)\,dx=\left(
\tfrac{1-p^{\lambda-1}}{1-p^{-\lambda}}\right)
(1-|pa\tilde a^2|_p^\lambda).
\end{align*}
Then
\begin{align*}
Z_p(s,\check\Wb\!_p^{\,0},\vpb_p)&=\int_{(\Q_p\ti)^3}\int_{\Q_p}
|a^3a_2^2a_3^2|^{s+1}\,|z|^{-2s-2}\,\e_p(x)\left(\tfrac{|ap|^{-\dot
s_1}-|ap|^{-\ddot s_1}}{|p|^{-\dot s_1}-|p|^{-\ddot s_1}}\right)
|a|^{\frac12}\,\ind_{\Z_p}(a)\\&\hspace{20pt}\cdot
\left(\tfrac{|aa_2p|^{-\dot s_2}|a_2|^{\ddot s_2}-|aa_2p|^{-\ddot
s_2}|a_2|^{\dot s_2}}{|p|^{-\dot s_2}-|p|^{-\ddot s_2}}\right)
|aa_2^2|^{\frac12}\,\ind_{\Z_p}(aa_2^2)\\&\hspace{20pt}\cdot
\left(\tfrac{|aa_3p|^{-\dot s_3}|a_3|^{\ddot s_3}-|aa_3p|^{-\ddot
s_3}|a_3|^{\dot s_3}}{|p|^{-\dot s_3}-|p|^{-\ddot
s_3}}\right)|aa_3^2|^{\frac12}\,\ind_{\Z_p}(aa_3^2)\\
&\hspace{20pt}\cdot\,|a^3a_2^2a_3^2|^{-1}\,dx
\,d\ti\!a\,d\ti\!a_2\,d\ti\!a_3\\
&=\int_{|a|\le1}\int_{|a_2|\le|a|^{-\frac12}}\int_{|a_3|\le
|a|^{-\frac12}}\left(\tfrac{1-p^{-2s-2}} {1-p^{2s+1}}\right)
(1-|pa\tilde a^2|^{-2s-1})\\&\hspace{20pt}\cdot
\left(\tfrac{|ap|^{-\dot s_1}-|ap|^{-\ddot s_1}}{|p|^{-\dot
s_1}-|p|^{-\ddot s_1}}\right) \left(\tfrac{|aa_2p|^{-\dot
s_2}|a_2|^{\ddot s_2}-|aa_2p|^{-\ddot s_2}|a_2|^{\dot
s_2}}{|p|^{-\dot s_2}-|p|^{-\ddot s_2}}\right)\\
&\hspace{20pt}\cdot\left(\tfrac{|aa_3p|^{-\dot s_3}
|a_3|^{\ddot s_3}-|aa_3p|^{-\ddot s_3}|a_3|^{\dot s_3}}
{|p|^{-\dot s_3}-|p|^{-\ddot s_3}}\right)
|a^3a_2^2a_3^2|^{s+\frac12}\,d\ti\!a\,d\ti\!a_2\,d\ti\!a_3.
\end{align*}
Now write $|a|_p=p^{-k}$, $|a_j|_p=p^{-k_j}$, so
\begin{align*}
Z_p(s,\check\Wb\!_p^{\,0},\vpb_p)={}&\sum_{k\ge0}\
\sum_{k_2\ge -\frac{k}2}\ \sum_{k_3\ge-\frac{k}2}
S(k,k_2,k_3,\min\{k,k+2k_2,k+2k_3\}),\\
S(k,k_2,k_3,l):={}&\left(\tfrac{1-p^{-2s-2}} {1-p^{2s+1}}\right)
p^{-(3k+2k_2+2k_3)(s+\frac12)}\\&\hspace{10pt}\cdot
\left(1-p^{(l+1)(2s+1)}\right)\left(\tfrac{p^{(k+1)\dot
s_1}-p^{(k+1)\ddot s_1}}{p^{\dot s_1}-p^{\ddot s_1}}\right)\\
&\hspace{10pt}\cdot\left(\tfrac{p^{(k+k_2+1)\dot s_2-k_2\ddot
s_2}-p^{(k+k_2+1)\ddot s_2-k_2\dot s_2}}{p^{\dot s_2}-p^{\ddot
s_2}}\right)\\&\hspace{10pt}\cdot\left(\tfrac{p^{(k+k_3+1)\dot
s_3-k_3\ddot s_3}-p^{(k+k_3+1)\ddot s_3-k_3\dot s_3}}{p^{\dot
s_3}-p^{\ddot s_3}}\right).
\end{align*}
We can make the sums independent:
\begin{align*}
Z_p(s,\check\Wb\!_p^{\,0},\vpb_p)&=\sum_{\mu=0,1}\ \sum_{k'\ge0}
\ \sum_{k_2'\ge0}\ \sum_{k_3'\ge0}\\&\hspace{40pt}
S(2k'+\mu,k_2'-k',k_3'-k',2\min\{k',k_2',k_3'\}+\mu)\\
&=\sum_{\eta=0,1}\ \sum_{l\ge0}\ \sum_{\mu=0,1}\ \sum_{k''\ge0}\
\sum_{k_2''\ge0}\ \sum_{k_3''\ge0}\\&\hspace{40pt}
(-1)^{\eta}\,S(2(k''+l+\eta)+\mu,k_2''-k'',k_3''-k'',2l+\mu).
\end{align*}
After multiplying out the summand as a polynomial and rearranging
the exponents (as linear combinations of
$\{1,l,k'',k''_2,k''_3\}$), we recognize $Z_p$ to be a finite sum
of products of geometric series.  This evaluates to a complicated
rational expression, which we factor using Mathematica. It is
necessary to impose the constraint on central characters, $\sum_j
(\dot s_j+\ddot s_j)=0$, e.g.\ by substitution for $\dot s_1$.
Then
\begin{align*}
Z_p(s,\check\Wb\!_p^{\,0},\vpb_p)=\frac{L_p(s+\frac12,
\pi_1\times\pi_2\times\pi_3)}{\zeta_p(2s+2)\,\zeta_p(4s+2)}.
\end{align*}

In the cases $p\mid d_B N^\sharp$, the difficult ramified zeta
integrals were calculated by Gross--Kudla \cite{gk}. We have
quoted their results, taking into account our differing
normalizations of measures and local data.
\end{proof}

\chapter{Conclusion}

\section{Triple Product Identities}

Let $\psi_j\in S_{k_j}(d_B,\vec N)$ be three Hecke-eigen
newforms of the same square-free level $N$, with
$k_1+k_2+k_3=0$, and let $\varrho_j$ denote the corresponding
Langlands parameters. Define $\veps_v=\prod_j\veps_{jv}$ for
the eigenvalues $\veps_{jv}$ as in \S\ref{autr}:\vspace{-10pt}
\begin{alignat*}{2}
T_\infty^-\,\psi_j&=\phantom{-}\veps_{j\infty}\,\psi_j&
\hspace{20pt}&\text{for \,}k_j=0,\\
T_p^{[p]}\,\psi_j&=-\veps_{jp}\,\psi_j&&\text{for \,}p\mid
d_BN.
\end{alignat*}
Writing $X=\ord^{(1)}(d_B,\vec N)\lmod\h$, we have
\begin{thm}\label{mth}
\begin{gather*}
\frac{\displaystyle\left|\int_{X}\psi_1(z)\,\psi_2(z)\,\psi_3(z)\,
\dvh{x}{y}\right|^2}{\displaystyle\sideset{}{_j}\prod\int_X
|\psi_j(z)|^2\,\dvh{x}{y}}=\big(\textstyle\prod_v Q_v\big)\,
\displaystyle\frac{2^{\#\{p\mid d_BN\}-3}}{(d_BN)^2}\,
\frac{L^*(\tfrac12,\varrho_1\otimes\varrho_2\otimes\varrho_3)}
{\text{$\prod_j L^*(1,\Ad\varrho_j)$}},\\[8pt]
Q_\infty =\begin{cases}
\frac{1+\veps_\infty}2 & \text{if }k_1=k_2=k_3=0,\\
\hspace{9.6pt}1 & \text{if }|k_1|=|k_2|>|k_3|=0,\\
\hspace{9.6pt}2 & \text{if }|k_1|>|k_2|\ge|k_3|>0,
\end{cases}\hspace{30pt}
Q_p=
\begin{cases}
\frac{1-\veps_p}2 & \text{if }p\mid d_B,\\
\frac{1+\veps_p}2 & \text{if }p\mid N,\\
\hspace{8pt}1 & \text{if }p\nmid d_BN.
\end{cases}
\end{gather*}
\end{thm}
\begin{proof}
Define the corresponding adelic forms $\Psi_j$ on $B_\A\ti$,
$F_j$ on $G_\A$, $F'_j$ on $H'_\A$, and the local data
$\vp_j\in\Sz(B_\A)$, as in \S\ref{autf},\,\S\ref{jls}. As
Harris--Kudla showed, it follows from the see-saw identity,
\begin{align*}
\big\langle 1,\ov{\Theta_{\vpb}(\ov{\Fb})}\big\rangle_{PH'_\A}
=\big\langle\Theta'_{\vpb}(1),\Fb\big\rangle_{P\Gb_\A},
\end{align*}
and the Siegel--Weil formula,
$\Theta'_{\vpb}(1)=\tfrac12E(0,h,\Phi),$ that
\begin{align*}
\int_{PH'_\Q\lmod PH'_\A} \Theta_{\vpb}(\ov{\Fb})(h')\,dh'
=Z(0,\ov{\Fb},\vpb).
\end{align*}
Then by Theorem \ref{adjshim},
\begin{align*}
\frac{\displaystyle\left|\int_{PB\ti_\Q\lmod PB\ti_\A}
\Psi_1(\beta)\,\Psi_2(\beta)\,\Psi_3(\beta)\,\dti\beta\right|^2}
{\prod_j\big\langle\Psi_j,\Psi_j\big\rangle_{PB\ti_A}}
=\frac{Z(0,\ov{\Fb},\vpb)}{\prod_j\big\langle F_j,F_j\big\rangle_{PG_\A}},
\end{align*}
and by Theorem \ref{zeta} and Lemma \ref{ransel},
\begin{align*}
\frac{\displaystyle\left|\int_{X}\psi_1(z)\,\psi_2(z)\,\psi_3(z)\,
\dvh{x}{y}\right|^2} {\displaystyle\sideset{}{_j}\prod\int_{X}
|\psi_j(z)|\,\dvh{x}{y}}=\frac{\zeta^*(2)}{4\,\vol(X)}\,
\frac{\prod_vC_v}{\prod_{j,v}c_{jv}}\,\frac{L^*(\frac12,
\varrho_1\otimes\varrho_2\otimes\varrho_3)}
{\prod_jL^*(1,\Ad\varrho_j)}.
\end{align*}
Noting that
\begin{align*}
\vol(X)=2\zeta^*(2)\textstyle\prod_{p\mid d_B}(p-1)
\prod_{p\mid N}(p+1),
\end{align*}
it is straightforward to compute the constants.
\end{proof}

\section{Applications to Quantum Chaos}

Consider the geodesic flow on a finite volume hyperbolic
surface $X=\Gamma\lmod\h$.  This is a classical Hamiltonian
dynamical system with phase space $T^*\!X$, the cotangent
bundle of $X$, and Hamiltonian \hbox{$H(v^*_x)=|v^*_x|^2$}, the
Riemannian norm squared.  We restrict our attention to
$S^*\!X$, the constant energy submanifold of unit cotangent
vectors, which is invariant under the flow.  Under the
homeomorphic identification
\begin{gather*}
S^*\!X\stackrel\sim\longrightarrow
\Gamma\lmod\PGL_2^+(\R),\\
(dy)_\ic\mapsto\Gamma \left[\begin{mat}{cc}1 &\\& 1
\end{mat}\right]\hspace{-2.5pt},
\end{gather*}
the flow map for time $t$ is given explicitly by right
multiplication,\vspace{-3pt}
\begin{align*}
T_t(\Gamma g)=\Gamma g\left[\begin{mat}{cc}e^t &\\& 1
\end{mat}\right]\hspace{-2.5pt}.
\end{align*}
The normalized Liouville measure on $S^*\!X$ is ergodic for the
geodesic flow and has entropy equal to 1. Furthermore,
Ornstein--Weiss proved this system is measurably isomorphic to a
Bernoulli flow. Topologically, the geodesic flow is less simple.
There are no stable periodic orbits, but since the flow is Anosov,
unstable periodic orbits form a dense subset. Each of these closed
geodesics obviously carries an ergodic measure, but there are also
ergodic measures with supports having any Hausdorff dimension
between 1 and 3.

The standard quantization of the classical geodesic flow has state
space $L^2(X)$ and Hamiltonian $\tilde H=-\Delta$, a
pseudo-differential operator with principal symbol $H$.  For a
particle described by the normalized quantum state $\psi\in
L^2(X)$, measuring its position is equivalent to sampling an
$X$-valued random variable defined by the probability distribution
\begin{align*}
d\mu(z)=|\psi(z)|^2\,\dvh{x}{y}\hspace{30pt}
\big(\,\|\psi\|_{L^2}=1\,\big).
\end{align*}
Similarly, measuring the particle's energy is equivalent to
sampling a ${\rm Spec}(\tilde H)$-valued random variable,
defined according to the spectral expansion of $\psi$.  If $X$
is compact, then $L^2(X)$ has an orthonormal basis of
eigenforms,
\begin{align*}
\tilde H\psi_j=\lambda_j\psi_j\,,\hspace{30pt}
0=\lambda_0<\lambda_1\le\lambda_2\le\ldots\,,
\end{align*} in which case the
corresponding spectral measure is simply
\begin{align*}
\sideset{}{_j}\sum|\langle\psi,\psi_j\rangle|^2\,\delta_{\lambda_j}.
\end{align*}
For non-compact $X$, $\tilde H$ also has absolutely continuous
spectrum equal to $\big[\frac14,\infty\big)$.  All of the
corresponding (non-square-integrable) eigenspaces are spanned by
unitary Eisenstein series and have dimension equal to the number
of cusps of $X$.

The special arithmetic $X$ we will consider are distinguished by
their numerous symmetries, giving rise to Hecke operators. Since
$\Hal^\star$ is commutative, it represents a collection of
simultaneously observable quantities.  Furthermore, since $\tilde
H\in\Hal_\infty^\star$\,, we may assume that the $\psi_j$
described above are in fact $\Hal^\star$-eigenforms.  As we saw in
\hbox{Lemma \ref{basis}} and the discussion before Lemma
\ref{eismth}, $\Hal^\star$ completely decomposes $L^2(X)$ with
multiplicity 1.  (Note that for these $X$, the residual spectrum
consists of $\{0\}$, i.e.\ all $\psi_j$ with $j>0$ are cuspidal.)

Quantum chaos is concerned with the behavior of eigenstates in
quantizations of classically chaotic systems, particularly
under semi-classical/high energy limits \cite{schur}. An
important general result due to
Colin-de-Verdiere/Shnirelman/Zelditch is that
\begin{align*}
\frac1N\sum_{j=1}^N\Big(\int_X f\,d\mu_j-\int_X f\,d\mu_0\Big)^2
\longrightarrow0\hspace{20pt}\text{as \,}N\to\infty,\text{ \,for all \,}
f\in C^\infty_c(X).
\end{align*} This is called
quantum ergodicity, and it is equivalent to the assertion that
almost all high energy states are nearly equidistributed. The
only assumption on $X$ is ergodicity of the Liouville measure
under geodesic flow, one of the mildest chaotic properties
ensured by negative curvature. However, quantum ergodicity does
not exclude the possibility that there may be other weak limits
of the $\mu_j$ along subsequences, a phenomenon known as
scarring. Scarring has been observed numerically in some
related systems, such as the Bunimovich Stadium billiard. In
the special cases of congruence hyperbolic surfaces (and
3-manifolds), Rudnick--Sarnak ruled out strong scarring along
closed geodesics \cite{rudsar}, leading them to conjecture that
\emph{all} high energy states become equidistributed.  This is
called quantum unique ergodicity (QUE). In \cite{luosar},
Luo--Sarnak made a quantitative formulation of the QUE
conjecture, predicting the rate of equidistribution: For all
$f\in C^\infty_c(X)$,
\begin{align*}
\int_X f\,d\mu_j-\int_X f\,d\mu_0 \ll_{f,\eps}\lambda_j^{-\frac14+\eps}
\hspace{20pt}\text{as \,}\lambda_j\to\infty.
\end{align*}
They proved this, in the case $X=\SL_2\Z\lmod\h$, on average
over $\mu_j$ and for the individual measures associated to
unitary Eisenstein series. We now reduce their quantitative
conjecture for individual $\mu_j$ to appropriate Lindel\"of
Hypotheses, using triple product identities.  To quantify the
smoothness of $f$, we use the Sobolev norms
\begin{align*}
\big\|f\big\|_{L^{2,r}}= \big\|(1-\Delta)^{\frac{r}2}
f\big\|_{L^{2}}\hspace{20pt}\text{for \,}r\in\R.
\end{align*}
Let $L^{2,r}(X)$ denote the Hilbert space completion of
$C^\infty_c(X)$ under $\|\,\|_{L^{2,r}}$.

\begin{thm}\label{sob}
Fix $X=\ord^{(1)}\hspace{-2pt}(d_B)\lmod\h$ and let the
$\psi_j\in L^2(X)$ be as before, with Langlands parameters
$\varrho_j$. Suppose that for some $0\le a\le1$ and all
$j,j'\ge1$,
\begin{align*}
L(\tfrac12,\varrho_j\otimes\ov{\varrho_j}\otimes\varrho_{j'})&\ll_{X,\eps}
C_\infty(0,\varrho_j\otimes\ov{\varrho_j}\otimes\varrho_{j'})
^{\frac{a}4+\eps}\text{,\hspace{10pt}and also}\\
L(\tfrac12,\varrho_j\otimes\ov{\varrho_j}\otimes\varrho_E)&\ll_{X,\eps}
C_\infty(0,\varrho_j\otimes\ov{\varrho_j}\otimes\varrho_E)^{\frac{a}4+\eps}
\hspace{10pt}\text{if }d_B=1.
\end{align*}
Then for $0\le r\le\frac{a+1}2$ and $\lambda_j\ge\tfrac14$,
\begin{align*}
\big\||\psi_j|^2-\tfrac1{\vol(X)}\big\|_{L^{2,-r}}\ll_{X,\eps}
\lambda_j^{\frac{a-r}2+\eps}\,.
\end{align*}
\end{thm}

\begin{cor}
Under the same hypotheses, set $r=\frac{a+1}2$.  Then
\begin{align}\label{qud}
\int_X f\,d\mu_j- \int_X f\,d\mu_0 \ll_{\eps}
\lambda_j^{\frac{a-1}4+\eps}\, \|f\|_{L^{2,r}}\,.
\end{align}
\end{cor}

For $a=1$\,, these hypotheses are Phragmen--Lindel\"of
convexity bounds, while an estimate quite similar to
(\ref{qud}) follows trivially from the Sobolev inequality,
\begin{align*}
\|f\|_{L^{\infty}}\ll_{X,\eps}\|f\|_{L^{2,1+\eps}}\,.
\end{align*}
Any value $a<1$ (`breaking convexity') implies the QUE
conjecture. Furthermore, the Grand Riemann Hypothesis implies
$a=0$, a generalization of Lindel\"of's Hypothesis, and this is
best-possible. It follows from our proof that the corresponding
exponent $-\frac14$ of $\lambda_j$ in (\ref{qud}) can't be
lowered. Our methods also apply to weight-$k$ eigenforms
$\psi$, giving the same result, although more smoothness is
required from $f$.  We will present these details later.

\begin{cor}
Under the same hypotheses, set $r=0$.  Then
\begin{align}\label{l4}
\|\psi_j\|_{L^{4}} \ll_{\eps}\lambda_j^{\frac{a}4+\eps}\,.
\end{align}
\end{cor}

The exponent $\frac14$ of $\lambda_j$ in (\ref{l4})
corresponding to $a=1$ is not the best which is presently known
for general surfaces ($\frac1{16}$ is proved in
\cite{seegsogg}), and it is not the best that can be done using
Theorem \ref{mth}. In joint work with Sarnak, we have obtained
the exponent $\frac1{24}$ unconditionally, and we hope to
improve this further.

Another related problem in quantum chaos concerns the amplitude
distribution of high energy eigenstates.  Viewing $(X,\mu_0)$
as a probability space, each $\psi_j$ becomes an $\R$-valued
random variable with cumulative distribution function
\begin{align*}
f_j(t)=\mu_0(\psi_j^{-1}(-\infty,t]).
\end{align*}
Berry/Hejhal's random wave conjecture asserts that these $f_j$
converge in a suitable sense to the normal distribution $\mathcal
N(0,\vol(X)^{-\frac12})$ as $\lambda_j\to\infty$.  It follows from
the proof of QE that the truth of the random wave conjecture for
fourth moments alone would imply QUE. Regarding third moments,
this conjecture says
\begin{align*}
\tfrac1{\vol(X)}\int_X \psi_j^3(z)\,\dvh{x}{y}\longrightarrow0
\hspace{20pt}\text{as \,}\lambda_j\to\infty.
\end{align*}
We now prove a stronger form of this (unconditionally) for the
full modular group.

\begin{thm}\label{thrd}
Fix $X=\SL_2\Z\lmod\h$ and let the $\psi_j\in L^2(X)$ be
defined as before.  Then
\begin{align*}
\int_X \psi_j^3(z)\,\dvh{x}{y}\ll_\eps\lambda_j^{-\frac1{12}+\eps}.
\end{align*}
\end{thm}

\subsection{Proofs}

\begin{lem}{\rm(Weyl's Law)}
Writing $\lambda_j=\tfrac14-s_j^2,\ s_j=\sigma_j+\ic t_j$,
\begin{align*}
\sum_{0\le t_j<T}1=\tfrac{{\rm vol}(X)}{4\pi}\,T^2+O_{X,\eps}
\big((1+T)^{1+\eps}\big),
\end{align*}
so for continuous $f$ of bounded variation on
$[a,b)\subset[0,\infty)$,
\begin{align*}
&\sum_{a\le t_j<b}f(t_j)-\tfrac{{\rm vol}(X)}{2\pi}
\int_a^b tf(t)\,dt\\[-4pt]&\hspace{40pt}\ll_{X,\eps}
(1+a)^{1+\eps}|f(a)|+(1+b)^{1+\eps}|f(b)|+\int_a^b
(1+t)^{1+\eps}\,|df|(t).
\end{align*}
\end{lem}

\begin{proof}[Proof of Theorem \ref{sob}]
Since $|\psi_j|^2\in L^2(X)$, we can apply Parseval's formula:
\begin{align*}
\big\||\psi_j|^2-\tfrac1{{\rm vol}(X)} \big\|^2_{L^{2,-r}}
&=\sum_{j'>0} (1+\lambda_{j'})^{-r}\,\big|\big\langle|\psi_j|^2,
\psi_{j'}\big\rangle\big|^2\\
&\hspace{20pt}+\int_{\R^+}(\tfrac54+t^2)^{-r}\,\big|\big\langle
|\psi_j|^2,E^1(\,\cdot\,,\tfrac12+\ic t)
\big\rangle\big|^2\,\tfrac{dt}\pi\,,
\end{align*}
where the integral only appears if $d_B=1$. First we analyze the
individual terms, using Theorem \ref{mth} and Lemma \ref{eismth}.
We may assume $\veps_{j'\infty}=1$, or else the corresponding term
vanishes by symmetry, and also $\lambda_j\ge\tfrac14$. Then
\begin{align*}
L_\infty(\tfrac12,\varrho_j\times\ov{\varrho_j}\times\varrho_{j'})
&=\zeta_\R(\tfrac12+2s_j+s_{j'}) \,\zeta_\R(\tfrac12+s_{j'})^2\,
\zeta_\R(\tfrac12-2s_j+s_{j'})\\&\hspace{20pt}
\cdot\,\zeta_\R(\tfrac12+2s_j-s_{j'})\,\zeta_\R(\tfrac12-s_{j'})^2\,
\zeta_\R(\tfrac12-2s_j-s_{j'}),\\[2pt]
C_\infty(0,\varrho_j\times\ov{\varrho_j}\times\varrho_{j'})&=
(1+|2s_j+s_{j'}|)^2\, (1+|s_{j'}|)^4\,
(1+|2s_j-s_{j'}|)^2,\\[8pt]
L_\infty(\tfrac12,\varrho_j\times\ov{\varrho_j}\times\varrho^E)
&=\zeta_\R(\tfrac12+2s_j+\ic t)\,\zeta_\R(\tfrac12+\ic t)^2\,
\zeta_\R(\tfrac12-2s_j+\ic t)\\&\hspace{20pt}
\cdot\,\zeta_\R(\tfrac12+2s_j-\ic t)\,\zeta_\R(\tfrac12-\ic t)^2\,
\zeta_\R(\tfrac12-2s_j-\ic t),\\[2pt]
C_\infty(0,\varrho_j\times\ov{\varrho_j}\times\varrho^E)&=
(1+|2s_\infty+\ic t|)^2\,(1+|\ic t|)^4\,(1+|2s_\infty-\ic t|)^2.
\end{align*}
Recall Stirling's approximation: For $s=\sigma+\ic t$ and
$|\arg(s)|\le\pi-\eps<\pi$,
\begin{align*}
\Gamma(s)\phantom{|}&=(1+O_\eps(|s|^{-1}))\,(2\pi)^{\frac12}\,
s^{s-\frac12}\,e^{-s},\\
|\Gamma(s)|&=(1+O_\eps(|s|^{-1}))\,(2\pi)^{\frac12}\,
|s|^{\sigma-\frac12}\,e^{-\sigma-t\arg(s)}.
\end{align*}
We will write $X\asymp Y$ in place of $X\ll Y\ll X$. Then
\begin{align*}
L_\infty(\tfrac12,\varrho_j\times\ov{\varrho_j}\times\varrho_{j'})
&\asymp\frac{e^{-\frac\pi2(|2t_j+t_{j'}|+2|t_{j'}|+ |2t_j-t_{j'}|)}}
{C(0,\varrho_j\times\ov{\varrho_j}\times\varrho_{j'})^{\frac14}}\,,\\[4pt]
L_\infty(1,{\rm Ad}\varrho_{j'})&\asymp e^{-\pi |t_{j'}|}\,,\\[4pt]
\frac{L_\infty(\tfrac12,\varrho_j\times\ov{\varrho_j}\times\varrho_{j'})}
{L_\infty{(1,{\rm Ad}\varrho_j)}^2\, L_\infty(1,{\rm Ad}\varrho_{j'})}
&\asymp\frac{e^{\frac\pi2(|2t_j+t_{j'}|-2|2t_j|+ |2t_j-t_{j'}|)}}
{C(0,\varrho_j\times\ov{\varrho_j}\times\varrho_{j'})^{\frac14}}\,,\\[12pt]
L_\infty(\tfrac12,\varrho_j\times\ov{\varrho_j}\times\varrho^E)
&\asymp\frac{e^{-\frac\pi2(|2t_j+t|+2|t|+|2t_j-t|)}}
{C(0,\varrho_j\times\ov{\varrho_j}\times\varrho^E)^{\frac14}}\,,\\[4pt]
\underset{z=1}{\rm res}\,L_\infty(z,{\rm Ad}\varrho^E)&\asymp
e^{-\pi|t|}\,,\\[2pt]
\frac{L_\infty(\tfrac12,\varrho_j\times\ov{\varrho_j}\times\varrho^E)}
{L_\infty{(1,{\rm Ad}\varrho_j)}^2\,\underset{z=1}{\rm res}\,
L_\infty(z,{\rm Ad}{\varrho}^E)}&\asymp\frac{e^{-\frac\pi2
(|2t_j+t|-2|2t_j|+ |2t_j-t|)}}{C(0,\varrho_j\times
\ov{\varrho_j}\times\varrho^E)^{\frac14}}\,.
\end{align*}
Note that
\begin{align*}
|a+b|-2|a|+|a-b|=2\max\{0,|b|-|a|\}.
\end{align*}
Finally, recall the estimate of Hoffstein--Lockhart \cite{hl},
\begin{align*}
L(1,{\rm Ad}\varrho_{j'})^{-1}\ll_\eps (1+|t_{j'}|)^{\eps}
\hspace{20pt}\text{for all \,}\eps>0.
\end{align*}

We divide the discrete spectrum contributions into
\begin{align*}
{\rm I}=\{0\le t_{j'}< t_j\},\hspace{20pt}
{\rm II}=\{t_j\le t_{j'}<3t_{j}\},\hspace{20pt}
{\rm III}=\{3 t_j\le t_{j'}\}.
\end{align*}
Then
\begin{align*}
\sideset{}{_I}\sum&\ll_\eps (1+t_j)^{a-1+\eps}\,
\sideset{}{_I}\sum(1+t_{j'})^{a-1-2r+\eps}\\
&\ll_\eps(1+t_j)^{\max\{a-1,2a-2r\}+\eps}
\ll_\eps(1+t_j)^{2a-2r+\eps},\\[10pt]
\sideset{}{_II}\sum&\ll_\eps(1+t_j)^{\frac32(a-1)-2r+\eps}\,
\sideset{}{_II}\sum(1+|2t_j-t_{j'}|)^{\frac12(a-1)+\eps}\\
&\ll_\eps(1+t_j)^{2a-2r+\eps},\\[10pt]
\sideset{}{_III}\sum&\ll_\eps(1+t_j)^{\eps}\,
\sideset{}{_III}\sum(1+t_{j'})^{2a-2-2r+\eps}\,
e^{-\pi(t_{j'}-2t_j)}\\
&\ll_\eps e^{-(\pi-\eps) t_j}
\ll (1+t_j)^{2a-2r}.
\end{align*}
Similarly, we divide up the continuous spectrum:
\begin{align*}
\int_0^{t_j}&\ll_\eps(1+t_j)^{a-1+\eps}
\int_0^{t_j}(1+t)^{a-1-2r+\eps}\,dt\\
&\ll_\eps(1+t_j)^{\max\{a-1,2a-2r-1\}+\eps}
\ll_\eps(1+t_j)^{2a-2r+\eps},\\[10pt]
\int_{t_j}^{3t_j}&\ll_\eps(1+t_j)^{\frac32(a-1)-2r+\eps}
\int_{t_j}^{3t_j}(1+|2t_j-t|)^{\frac12(a-1)+\eps}\,dt\\
&\ll_\eps(1+t_j)^{2a-2r-1+\eps},\\[10pt]
\int_{3t_j}^{\infty}&\ll_\eps(1+t_j)^{\eps}
\int_{3t_j}^{\infty} (1+t)^{2a-2-2r+\eps}\,
e^{-\pi(t-2t_j)}\,dt\\
&\ll_\eps e^{-(\pi-\eps) t_j}
\ll(1+t_j)^{2a-2r}.
\end{align*}
\end{proof}

\begin{proof}[Proof of Theorem \ref{thrd}]
First note that
\begin{align*}
L_v(s,\otimes^3\varrho_j)=L_v(s,{\rm Sym}^3\varrho_j)\,
L_v(s,\varrho_j)^2.
\end{align*}
Then for $\lambda_j\ge\frac14$,
\begin{align*}
\frac{L_\infty(\tfrac12,\otimes^3\varrho_j)}
{L_\infty{(\tfrac12,{\rm Ad}\varrho_j)}^3}\asymp
(1+|t_j|)^{-2},\hspace{30pt}
C_\infty(0,{\rm Sym}^3\varrho_j)\asymp(1+|t_j|)^4.
\end{align*}
To prove the convexity bound for $L(s,{\rm Sym}^3\varrho_j)$,
we first note that it is entire by the work of Kim--Shahidi
\cite{kimshah}, and has finite order by the work of
\mbox{Gelbart--Shahidi \cite{gelshah}}. Furthermore, using the
converse theorem of Cogdell--Piatetski-Shapiro \cite{cogshap},
Kim--Shahidi proved that ${\rm Sym}^3\varrho_j$ is cuspidal on
$\GL_4$ \cite{kimshah2}. Recent results of Molteni \cite{mol}
then imply
\begin{align*}
L(1+\ic t,{\rm Sym}^3\varrho_j)&\ll_\eps
C_\infty(t,{\rm Sym}^3\varrho_j)^{\eps},\\
L(0+\ic t,{\rm Sym}^3\varrho_j)&\ll_\eps
C_\infty(t,{\rm Sym}^3\varrho_j)^{\frac12+\eps},
\end{align*}
and hence by the Phragmen--Lindel\"of/Hadamard three-circles
method,
\begin{align*}
L(\tfrac12,{\rm Sym}^3\varrho_j)\ll_\eps
C(0,{\rm Sym}^3\varrho_j)^{\frac14+\eps}
\ll(1+|t_j|)^{1+\eps}.
\end{align*}
Now to prove that the third moment tends to zero, we only need
subconvexity for $L(\tfrac12,\varrho_j)$. The first such
estimate was proved conditionally by Iwaniec in \cite{iwan};
his result is made unconditional at a small cost to the
exponent by the same technique as in \cite{iwsar}. The current
record bound is due to Ivi\'c and subsequently Jutila
\cite{ivic,jut}:
\begin{align*}
L(\tfrac12,\varrho_j)\ll_\eps(1+|t_j|)^{\frac13+\eps}.
\end{align*}
Therefore
\begin{align*}
\left|\int_X\psi_j^3(z)\,\dvh{x}{y}\right|^2\ll_\eps
(1+|t_j|)^{-\frac13+\eps}.
\end{align*}
\end{proof}

\end{document}